\documentclass[12pt]{amsart}
\usepackage{amsmath}
\usepackage{amssymb}
\usepackage{amstext}
\usepackage{amscd}
\usepackage[matrix,arrow,ps]{xy}
\usepackage{enumitem}

\usepackage[OT2,T1]{fontenc}
\DeclareSymbolFont{cyrletters}{OT2}{wncyr}{m}{n}
\DeclareMathSymbol{\Sha}{\mathalpha}{cyrletters}{"58}

\usepackage{tikz-cd} 

\usepackage{mathtools} % for \DeclarePairedDelimiter
\usepackage[bookmarks=true,backref,colorlinks=true,citecolor=blue,urlcolor=blue,linkcolor=magenta]{hyperref}
\DeclarePairedDelimiter{\abs}{\lvert}{\rvert}
\DeclarePairedDelimiter{\norm}{\lVert}{\rVert}

\newcommand{\tens}{\otimes}
\newtheorem{theorem}{Theorem}[section]
\newtheorem{definition}[theorem]{Definition}
\newtheorem{proposition}[theorem]{Proposition}
\newtheorem{lemma}[theorem]{Lemma}
\newtheorem{corollary}[theorem]{Corollary}

\DeclareMathOperator {\Spec}{\operatorname{Spec}}
\DeclareMathOperator {\Stab}{\operatorname{Stab}}
\newcommand{\set}[2]{\left\{\,#1~\middle|~#2\,\right\}}

\newcommand{\BT}{\mathcal{BT}}
\newtheorem{conj}[theorem]{Conjecture}

\newtheorem{rem}[theorem]{Remark}

\newcommand{\citestacks}[2][Tag]{\cite[\href{https://stacks.math.columbia.edu/tag/#2}{#1 (#2)}]{Stacks}}

\newcommand{\Hom}{\mbox{Hom}}

\newcommand{\End}{{\rm End}}

\newcommand{\wh}{\widehat}

\newcommand{\Gal}{{\rm Gal}}

\newcommand{\ab}{{\rm ab}}

\newcommand{\ad}{{\rm ad}}

\newcommand{\FF}{{\mathbb F}}
\newcommand{\CC}{{\mathbb C}}
\newcommand{\RR}{{\mathbb R}}
\newcommand{\ZZ}{{\mathbb Z}}
\newcommand{\QQ}{{\mathbb Q}}

\newcommand{\AAA}{{\mathbb A}}

\newcommand{\lto}{\longrightarrow}

\newcommand{\cH}{\mathcal{H}}

\newcommand{\ol}{\overline}

\newcommand{\wt}{\widetilde}

\newcommand{\Q}{\mathbb{Q}}
\newcommand{\Z}{\mathbb{Z}}
\newcommand{\F}{\mathbb{F}}
\newcommand{\R}{\mathbb{R}}

\newcommand{\A}{\mathbb{A}}

\title[Andr\'e-Pink-Zannier for Shimura varieties of abelian type]{Generalised Andr\'e-Pink-Zannier conjecture for  Shimura varieties of abelian type}
\author{Rodolphe Richard}
\address{UCL Department of Mathematics,
University College London, 
Gower Street, 
London, WC1E 6BT
}
\email{r.richard@ucl.ac.uk}
 \author{Andrei Yafaev}
\address{UCL Department of Mathematics,
University College London, 
Gower Street, 
London, WC1E 6BT
}
\email{yafaev@ucl.ac.uk}

\begin{document}

\begin{abstract}
In this paper we prove the generalised Andr\'e-Pink-Zannier conjecture (an important case of the Zilber-Pink conjecture)
for all Shimura varieties of abelian type. Questions of this type were first asked by Y.~André in 1989.
We actually prove a general statement for all Shimura varieties, subject to certain assumptions
that are satisfied for Shimura varieties of abelian type and are expected to hold in general.

We also prove another result, a~$p$-adic Kempf-Ness theorem, on the relation between good reduction of homogeneous spaces over~$p$-adic integers with Mumford stability property in $p$-adic geometric invariant theory.
\end{abstract}

\maketitle
\setcounter{tocdepth}{1	}

\tableofcontents

\section{Introduction}
%RÉCRIRE ÉNÓNÇÉS PRINCIPAUX
 The central object of study in this article is the following conjecture.
\begin{conj}[Andr\'e-Pink-Zannier] \label{APZ}
Let $S$ be a Shimura variety and $\Sigma$ a subset of 
a generalised Hecke orbit in $S$ (as in~\cite[Def.~2.1]{RY}). Then the irreducible components of the Zariski closure of 
$\Sigma$ are weakly special subvarieties.
\end{conj}
This conjecture is an important special case of the 
 Zilber-Pink conjectures for Shimura varieties, which
has recently been and continues to be a subject of active research.

A special case of Conjecture~\ref{APZ} was first formulated
in 1989 by Y.~André in~\cite[\S{}X 4.5, p.\,216  (Problem 3)]{Andre}. Conjecture~\ref{APZ} was then stated in the introduction to the second author’s 2000 PhD thesis~\cite{Y}\footnote{The statement there uses the terminology ‘totally geodesic subvarieties’ instead of ‘weakly special’, but Moonen had proved in~\cite{MoMo} that the two notions are equivalent.}, following discussions
with Bas Edixhoven. Both statements refer to classical Hecke orbits, rather than \emph{generalised} Hecke orbits (cf.~\cite[\S2.5.1]{RY}).

Zannier has considered questions of this type in the context of abelian
schemes and tori. Richard Pink, in his 2005 paper~\cite{Pink}, has formulated
and studied this question. %he used a generalised notion of Hecke orbit,
%defined using auxiliary linear representations 
% (cf.~\cite[\S3.4.2]{RY}).

Pink proves it for “Galois generic” points of Shimura varieties\footnote{Roughly, 
the image of the corresponding Galois representation intersects the 
derived subgroup of the ambient
group in an adélically open subgroup. This may be too strong to hold in general. See the first author's 2009 PhD thesis~\cite[III.\S7, p.\,59]{R-PhDfull} for a weaker assumption.}: this implies in particular that such points are Hodge generic in their connected component.
Pink uses equidistribution
of Hecke points proved in~\cite{COU} (or in~\cite{EO}).

We refer to the introduction of~\cite{RY} for further background on Conjecture~\ref{APZ}.

In the Pila-Zannier approach and most other approaches to Zilber-Pink conjectures, one of the major difficulties is to obtain suitable lower
bounds for Galois orbits of points in the ``unlikely locus'' (see~\cite{DR}). 

In~\cite{RY}, we develop a general approach to Conjecture~\ref{APZ} based on the Pila-Zannier strategy (o-minimality and functional transcendence). In~\cite{RY}, we define generalised Hecke orbits, we define a natural height function on these
orbits, and we prove precise lower Galois bounds~\cite[Th.~6.4]{RY} under the ``weakly adélic Mumford-Tate conjecture''~\cite[\S6.1]{RY}.

%a number of relevant notions are introduced and studied, and the conjecture proved under the assumption of the 
%`weak adelic the Mumford-Tate conjecture'. 
%In particular, in~\cite{RY} we define a natural height function on a generalised Hecke
%orbit, and prove precise lower Galois bounds~\cite[Th.~7.4]{RY} under the weakly adelic Mumford-Tate conjecture.% (This conjecture, and the bounds, are expected 
Let~$(G,X)$ be a Shimura datum and let~$K\leq G(\AAA_f)$ be a compact open subgroup
and let~$S=Sh_K(G,X)$ be the associated Shimura variety.
The main result of \cite{RY} is as follows.
\begin{theorem}[Theorem 1.2 of \cite{RY}]\label{main theorem RY}
Let $x_0 \in X$.
Assume that~$x_0$ satisfies the weakly adélic Mumford-Tate conjecture.

Then the conclusion of Conjecture~\ref{APZ} holds for any subset of the generalised Hecke orbit of $[x_0,1]$.
\end{theorem}

In the present article we prove conclusions of this theorem \emph{unconditionally} for all Shimura varieties 
\emph{of abelian type}. This completely generalises the main result of~\cite{Orr} by M.~Orr.
%The latter was also obtained, with refinements, in~\cite{RY2} by the authors.

Our main result is as follows.
\begin{theorem} \label{main theorem}
Let~$s_0$ be a point in a Shimura variety~$Sh_K(G,X)$ of abelian type.
Let~$Z$ be a subvariety whose intersection with the generalised Hecke orbit of~$s_0$ is Zariski dense
in~$Z$. 
Then~$Z$ is a finite union of weakly special subvarieties of~$S$.
\end{theorem}

We actually prove the more general statement below, which we believe to be of independent
interest. Its assumption is weaker than~`weakly adélic Mumford-Tate conjecture' in~Th.~\ref{main theorem RY}. It is the `uniform integral Tate conjecture' assumption explained in~\S\ref{sec:Def:Tate}.
We refer to~\cite[Def.~2.1]{RY} for the notion of geometric Hecke orbit. By~\cite[Th.~2.4]{RY}, a generalised Hecke orbit is a finite union of geometric Hecke orbits. 
\begin{theorem} \label{main theorem 2}
Let~$s_0=[x_0,1]$ be a point in a Shimura variety~$Sh_K(G,X)$, and assume the uniform integral Tate conjecture
for~$x_0$ in~$X$ in the sense of~Definition~\ref{defi:Tate bis}.
Let~$Z$ be a subvariety whose intersection with the geometric Hecke orbit of~$s_0$ is Zariski dense
in~$Z$. 
Then~$Z$ is a finite union of weakly special subvarieties of~$S$.
\end{theorem}

Using Faltings' theorems, we prove in~\S\ref{Tate:abelian type} that points on Shimura varieties of adjoint abelian type and satisfy this `uniform integral Tate assumption'.
Thus Theorem~\ref{main theorem}, in the adjoint type case, is a special case of Theorem~\ref{main theorem 2}. Because Conjecture~\ref{APZ} can be reduced to the adjoint case, we deduce Theorem~\ref{main theorem} for any Shimura variety of abelian type.

At the heart of this article is obtaining polynomial lower bounds~\cite[Th.~6.4]{RY} which are unconditional for Shimura varieties
of abelian type, or in general under the assumption of the Tate hypothesis. We emphasize that Shimura varieties of abelian type constitute the most important class
of Shimura varieties.
 
The Tate hypothesis is used to compare the sizes of Galois orbits with that of the adélic orbits of~\cite[App.~B]{RY}. In our setting,
we can easily recover former results of~\cite{OrrPhD} which were only concerned with~$S$-Hecke orbits (involving a finite set~$S$ of primes).
In order to work with whole Hecke orbits, and even geometric Hecke orbits, we use an ``integral and uniform'' refined version of the Tate conjecture. Using generalised Hecke orbits is important for our strategy to work, in particular for the reduction steps in ~\cite[\S7]{RY}.

Some of the new ideas in this article relate the notion of ``Stability'' in the Mumford sense to the Tate hypothesis. The fine estimates we need
use stability not only over complex numbers, but in a broader context, over~$\Z_p$ and~$\Z$. This is where the ``uniformity and integrality''
in our Tate hypothesis is essential. These ideas originate from~\cite{R-PhD}, part of the first author's 2009 PhD thesis.

This article also develops several results of independent interest. Theorem~\ref{pKN} is a~$p$-adic version of a Theorem of Kempf-Ness~\cite{KN}. We expect it to be useful in other contexts, and it is proved in more generality than needed here. Theorem~\ref{thm:compare reductive} gives precise and uniform comparison on norms along two closed orbits of reductive groups.
In Appendix~\ref{sec:AppB}, we give some consequences of Faltings theorems, in the axiomatic form given in~\S\ref{sec:Def:Tate}, for factorisations of Galois images, and in particular~$\ell$-independence. Our arguments rely only on group theory, they do not involve ramification properties, and therefore apply to more general groups than images of Galois representations. 

%
%\begin{itemize}
%\item Theorem~\ref{thm:compare reductive} gives precise and uniform comparison on norms along two closed orbits of reductive groups.
%\item Theorem~\ref{pKN} is a~$p$-adic version of a Theorem of Kempf-Ness~\cite{KN}.
%\end{itemize}
%
%In a future work, we hope to extend further Theorem~\ref{main theorem 2} by proving Conjecture~\ref{APZ} under a weaker hypothesis than `uniform integral Tate conjecture', a hypothesis
%related to the `generalised Shafarevich conjecture' in the sense of~\cite{UY} or~\cite{RY2}.

\subsection*{Outline of the paper}
We define the uniform integral hypothesis in section~\ref{sec:Def:Tate}.

In section~\ref{sec:proof}, we reduce Th.~\ref{main theorem 2}
to the bounds on Galois orbits established in the rest of the paper, and the functorial
invariance properties of the Tate hypothesis of section~\ref{sec:functoriality}.
Since the formal strategy is almost identical to that of \cite[\S7]{RY} 
we only give a sketch indicating necessary adjustments and provide precise references to~\cite{RY}.

In section~\ref{sec:functoriality}, we also derive the refined version of Faltings' theorems that we use, 
using arguments of Serre and Noot. We deduce that the uniform integral Tate
hypothesis holds in Shimura varieties which are of abelian type and also of adjoint type.

The central and technically hardest parts of the paper are \S\S\ref{sec:bounds}--\ref{sec:pKN}.
There we establish the lower bounds for the Galois orbits of points in 
geometric Hecke orbits as in \cite{RY} under  assumptions of Th.~\ref{main theorem 2}.
%We believe the work carried out in these sections is of independent interest and can
%be used in many other contexts. This is why we attempted to state and prove 
%results in these sections is greatest possible generality.

The main result~Th.~\ref{thm:compare reductive} of section~\ref{sec:reductive} is 
essential to the proofs in section~\ref{sec:bounds}. We derive it in section~\ref{sec:reductive}
from the results of sections~\ref{sec:pKN} and~\ref{sec:slopes}.

Section~\ref{sec:pKN} gives a~$p$-adic analogue Th.~\ref{pKN} of a Theorem of Kempf-Ness.
We prove in greater generality than required for Th.~\ref{thm:compare reductive}, as we believe
it will be useful in other contexts. It involves good reduction properties 
of homogeneous spaces of reductive groups over $\ZZ_p$, and of closed orbits in linear representations over $\ZZ_p$.

%In section~\ref{sec:reductive} the main result~Th.~\ref{thm:compare reductive} is deduced from the results obtained in  
%subsequent sections~\S\S\ref{sec:pKN},\ref{sec:slopes}.
%
%These properties are necessary for the strategy of the proof of the main result.

%Another major result in this section is verification of the uniform Tate condition
%for Shimura varieties of abelian and adjoint type.

%Let us briefly review the new technical tools and results from~\S\S\ref{sec:pKN},\ref{sec:slopes}. The ideas are inspired by~\cite{R-PhD}, which part of the first author's 2009 PhD~Thesis. We believe that these results will be useful in other contexts, which is why 
%stated and proved them in greater generality than what  is needed for our main result: not just for conjugacy classes,
%but also for more general homogeneous spaces.

%A key observation is that the uniform integral Tate property are translated into properties related to ``stability'' in the sense of Mumford. This notion of stability is understood in a broader context, over~$\ZZ_p$ and~$\ZZ$.

%The result we directly use in the proof of our lower Galois is the Theorem is Th.~\ref{thm:compare reductive}. 
%The proof of Th.~\ref{thm:compare reductive} relies on a $p$-adic analogue Th.~\ref{pKN} of a Theorem of Kempf-Ness,
%and on estimates on in \S\ref{sec:slopes} on piecewise affine convex functions. 

The ideas behind the convexity and slope estimates in \S\ref{sec:slopes} can be better understood in the context of Bruhat-Tits buildings as in~\cite{R-PhD}. The height functions which are central in our implementation of the Pila-Zannier strategy give examples of the type of functions studied in~\S\ref{sec:slopes}.

The Appendix~\ref{AppA} describes results about closed orbits of tuple in the adjoint representation of reductive groups in arbitrary characteristic. These are used in the proof of Prop.~\ref{Prop5.5}.

The Appendix~\ref{sec:AppB} is used in~\S\ref{thm 51 reduction produit} in the proof of Theorem~\ref{Galois bounds}.

\subsubsection*{Acknowledgements} We would like to express our greatest gratitude to Laurent Moret-Bailly for 
discussions and suggestions regarding the content of section~\ref{sec:pKN}. The first author is grateful to Ahmed Abbes and  Emmanuel Ullmo for discussions about \S\ref{sec:flat recalls}, and Ofer Gabber about Th.~\ref{pKN}. The authors are especially grateful to the referees, for their thorough reading and their numerous useful comments.

Both authors were supported by Leverhulme Trust Grant RPG-2019-
180. The support of the Leverhulme Trust is gratefully acknowledged.

The first author is grateful to the IHÉS for its invitation during the preparation of this article.

\section{Uniform integral Tate conjecture}\label{sec:Def:Tate}
In this section, we define in~Def.~\ref{defi:Tate bis} our main assumption in this paper, the `uniform integral Tate conjecture' property.
This is an extension of the conclusions of Faltings' theorem in the form given in Th.~\ref{Faltings}, to all Shimura varieties.
\subsection{Uniform integral Tate conjecture}
In~\S\ref{defTate1} and~\S\ref{defTate2} we consider an abstract setting. In~\S\ref{Shimura applied def} we specialise it to the context of Shimura varieties.

\subsubsection{}\label{defTate1}
%We first state in Def.~\ref{defi:Tate} the property in a more general context than situations involving Shimura data
%To explain the conditions in the Theorem~\ref{main theorem 2},

Let~$M\leq G$ be (connected) reductive algebraic groups over~$\QQ$.
%, let~$M,Z\leq G$ be (connected? Z) reductive $\Q$-algebraic subgroup and fix a faithful $\Q$-linear representation
We identify~$G$ with its image by a faithful representation% of $G$ which we denote by
\[
\rho_G:G\to GL(d).
\]
Def.~\ref{defi:Tate} and Theorem~\ref{main theorem 2} will not depend on this choice.
The Zariski closure in~$GL(d)_{\Z}$ of the algebraic groups~$M$ and~$G$ and~$Z_{G}(M)$ define models over~$\Z$. We write~$G_{\F_p}$ for the special fibre\footnote{For almost all primes~$p$ the group~$G(\ZZ_p)$ is hyperspecial and
%there is a model~$G_{\ZZ_p}$ over~$\ZZ_p$ 
%of~$G_{\QQ_p}$ such that~$G_{\ZZ_p}(\ZZ_p)=G(\ZZ_p)$ and such that the special fibre
~$G_{\FF_p}$ is a connected reductive
algebraic group over~$\FF_p$.} and
\[
G(\ZZ_p)=G(\QQ_p)\cap GL(d,\ZZ_p)\text{ and }G(\widehat{\ZZ})=\prod_p G(\ZZ_p)=G(\AAA_f)\cap GL(d,\widehat{\ZZ}).
\]
We also have a reduction map~$G(\ZZ_p)\to G_{\FF_p}(\FF_p)$. These constructions apply to~$M$ and~$Z_G(M)$ as well.

\subsubsection{} \label{defTate2}
Let~$U\leq M(\AAA_f)$ be a compact subgroup.

For every prime~$p$, we define~$U_p=M(\ZZ_p)\cap U$.
We denote by~$U(p)$ the image of~$U_p$ in~$G(\FF_p)$.
We define~${U_p}^0=U_p\cap {H_p}^0(\QQ_p)$ where~$H_p=\overline{U_p}^{Zar}\leq G_{\QQ_p}$ is the Zariski closure as a~$\QQ_p$-algebraic subgroup, and~${H_p}^0$ is its neutral Zariski connected component.

%The following will be applied to Galois images associated to~$x_0$ and~$\rho$.
\begin{definition}[{\bf Uniform integral Tate property}]\label{defi:Tate}
 We say that a compact subgroup~$U\leq M(\AAA_f)$ \emph{``satisfies the uniform integral Tate'' property with respect to~$M$,~$G$ and~$\rho_G$} if:
\begin{enumerate}
\item \label{defi:Tate1}
For every~$p$,
\begin{subequations}
\begin{equation}\label{defi:tate eq1}
Z_{G_{\Q_p}}(U_p)=
Z_{G_{\Q_p}}({U_p}^0)=
Z_{G}(M)_{\Q_p}.
\end{equation}
and
\begin{equation}\label{defi:tate eq 1.2}
\text{the action of $U_p$ on ${\Q_p}^d$ is semisimple.}\footnote{Some authors refer to this as completely reducible.}
\end{equation}
(This~\eqref{defi:tate eq 1.2} is equivalent to: $H_p$ is reductive.)
\end{subequations}
\item \label{defi:Tate2}
For every~$D$, there exists an integer~$M(D)$ such that for every~$p \geq M(D)$ and every~$U'\leq U_p$ of index~$[U_p:U']\leq D$, we have 
\begin{subequations}
\begin{equation}\label{defi:tate eq 2}
Z_{G_{\F_p}}(U'(p))=Z_{G_{\F_p}}(M_{\F_p})
\end{equation}
and
\begin{equation}\label{defi:tate eq 2.2}
\text{the action of $U'(p)$ on $\overline{\F_p}^d$ is semisimple.}
\end{equation}
\end{subequations}
(When~$p>d$,~\eqref{defi:tate eq 2.2} is equivalent to: the Nori group, defined below, of~$U'(p)$ is semisimple.)
\end{enumerate}
\end{definition}
In our terminology, \emph{integrality} refers to the second property over~$\F_p$ on~$U(p)$ and \emph{uniformity} to the fact that
the integer~$M(D)$ depends on~$D$ only.

\subsubsection{Remarks} \label{rem:Tate}
We collect here some facts that will be used throughout this article.
\begin{enumerate}
%\item \label{rem1}
\item \label{rem2}
For~$p$ large enough, in terms of~$d$, we can use Nori theory~\cite{N}. For a subgroup~$U'(p)\leq G(\F_p)$,
the group ${U'(p)}^{\dagger}$ generated by unipotent elements of~$U'(p)$ is of the form~$H(\FF_p)^\dagger$ for an algebraic group~$H\leq G_{\FF_p}$ over $\FF_p$. We call this~$H$ the \textbf{Nori group} of~$U'(p)$.
The property~\eqref{defi:tate eq 2.2} is then equivalent to the fact that~$H$ is a \textbf{reductive} group~$H\leq G_{\FF_p}$ over $\FF_p$(see~\cite[Th.~5.3]{SCR}).
We also note that~$[H(\FF_p):H(\FF_p)^\dagger]$
can be bounded in terms of~$\dim(G)$ (see.~\cite[3.6(v)]{N}).
\item \label{rem2.2} If~$U'\leq U$ has index~$[U:U']\leq p$, then~$U'(p)^\dagger= U(p)^\dagger$.  
\item \label{rem3}
This ``uniform integral Tate'' property does not depend\,\footnote{Indeed, Def.~\ref{defi:Tate} does not involve~$\rho_G$ itself, but only the induced models of~$G$ and~$M$. The algebraic groups~$G_{\Q_p}$ and~$M_{\Q_p}$ do not depend on the integral models, and two models, for almost all~$p$,  induce the same local models~$G_{\Z_p}$ and~$M_{\Z_p}$.} on the choice of a faithful representation~$\rho_G$. 
\item \label{rem4}
The semisimplicity of the action over~$\overline{\F_p}$ is equivalent to the semisimplicity over~$\F_p$.(cf. 
\cite[Alg. VIII, \S12. N.8 Prop. 8 Cor 1 a)]{BBKA8} with~$K=\F_p$, $L=\ol{\F_p}$, $A=K[U'(p)]$, $M=\ol{\F_p}^d$.)
\item \label{passage aux Up}
The group~$U\leq M({\AAA_f})$ ``satisfies the uniform integral Tate'' property with respect to~$M$,~$G$ and~$\rho$ if and only if the subgroup~$\prod_p {U_p}\leq U$ does so.
%\item For some~$p$ let~$V_p\leq U_p$ be a subgroup of finite index. Then~\eqref{defi:tate eq1} (resp.~\eqref{defi:tate eq1.2}) is satisfied for~$U_p$ if and only if~\eqref{defi:tate eq1} (resp.~\eqref{defi:tate eq1.2}) is satisfied for~$V_p$.
%Using Lem.~\ref{lem:Gcr} and Cor.~\ref{cor:Gcr} we have the following. For some~$p$ let~$V_p\leq U_p$ be a subgroup.
%If~\eqref{defi:tate eq1} and~\eqref{defi:tate eq1.2} are both satisfied for~$V_p$, then~\eqref{defi:tate eq1} and~\eqref{defi:tate eq1.2} are both satisfied for~$U_p$.

\item 
Part~\eqref{defi:Tate1} of Def.~\ref{defi:Tate} is satisfied for~$U$ if and only it is satisfied for some subgroup of finite index in~$U$. It follows from Lem.~\ref{lem:Gcr} and Cor.~\ref{cor:Gcr} that, for any subgroup~$U'\leq U$, if Part~\eqref{defi:Tate1} of Def.~\ref{defi:Tate} is satisfied for~$U'$, then it is satisfied for~$U$.
\item \label{rem(7)} 
Let~$U'\leq U_p\leq M(\Q_p)$ be subgroups. 

If Property~\eqref{defi:tate eq1} of part~\eqref{defi:Tate1} of Def.~\ref{defi:Tate} is satisfied for~$U'$, then it is satified for~$U_p$: we will have~$U'^0\leq {U_p}^0\leq U_p \leq M_{\Q_p}$, and~$Z_G(M)=Z_G(U'^0)\geq Z_G({U_p}^0)\geq Z_G({U_p})\geq Z_G(M)$.

Assume that~$U'$ is of finite index in~$U_p$. Then Property~\eqref{defi:tate eq 1.2} of part~\eqref{defi:Tate1} of Def.~\ref{defi:Tate} is satisfied for~$U'$ if and only if it is satified for~$U_p$: if~$H'_p$ is the Zariski closure of~$U'$, the finite index property implies~${H'_p}^0=H_p^0$.
\item \label{rem g g' tate}
If a compact subgroup~$U\leq M(\AAA_f)$ \emph{``satisfies the uniform integral Tate'' property with respect to~$M$,~$G$ and~$\rho_G$}, then, for any~$M\leq G'\leq G$, the group~$U$ satisfies the uniform integral Tate property \emph{with respect to~$M$,~$G'$ and~$\rho_G$}.
\item \label{rem(9)} For every~$p$, let~$V_p\leq U_p$ be a finite index subgroup, and assume that there exists~$D'\in\Z_{\geq1}$ such that
\begin{equation}\label{Vp dans Up toujours borne}
\forall p, [U_p:V_p]\leq D'.
\end{equation}
If~\eqref{defi:Tate2} is satisfied for~$U$, then~\eqref{defi:Tate2} is satisfied for~$\prod_p V_p$ with the function~$D\mapsto M(D\cdot D')$.
\end{enumerate}
In view of Remark~\ref{rem:Tate}~(\ref{rem3}) we may, from now on, just say ``satisfies the uniform integral Tate property'' without referring to a particular faithful representation~$\rho_G$.

We deduce from the above facts the following.
\begin{lemma}\label{lem U ast} 
Let~$U''\leq U\leq M(\widehat{\Z})$ be such that~$U''$ satisfies the uniform integral Tate property with respect to~$M$,~$G$ and~$\rho_G$ and such that~$U''$ is of finte index in~$U$. Then~$U$ satisfies the uniform integral Tate'' property with respect to~$M$,~$G$ and~$\rho_G$.
\end{lemma}
%We also have the following.
%\begin{lemma}Let~$U\leq M(\A_f)$ satisfy the uniform integral Tate'' property with respect to~$M$,~$G$ and~$\rho_G$.
%For every~$p$, let~$V_p\leq U_p$ be a finite index subgroup, and assume that there exists~$D'\in\Z_{\geq1}$ such that
%\[
%\forall p, [U_p:V]\leq 1.
%\]
%Then~$\prod_p V_p$ satisfies the uniform integral Tate'' property with respect to~$M$,~$G$ and~$\rho_G$.
%\end{lemma}
\subsubsection{}\label{Shimura applied def}
We denote by~$(G,X)$ a Shimura datum, by~$K\leq G(\A_f)$ a compact open subgroup, and by~$S=Sh_K(G,X)$ the associated Shimura variety. Fix~$x_0\in X$ and let~$M\leq G$ be the Mumford-Tate group of~$x_0$. Let~$E$ be
a field of finite type over $\Q$ such that~$s_0=[x_0,1]\in S(E)$ (such an~$E$ always exists). 
We denote by~$\rho_{x_0}:\Gal(\ol{E}/E)\to M(\A_f)$ 
the representation associated to $x_0$ (see \cite[\S3.1, Def.\,3.1]{RY}), and by~$U\leq M(\A_f)\cap K$ its image.
%We identify~$G$ with its image in a faithful representation~$\rho:G\to GL(d)$, chosen so that~$K\leq G(\hat{\Z}):=G(\A_f)\cap GL(d,\hat{\Z})$.
%Note that Conjecture~\ref{APZ} and  Theorem~\ref{main theorem}
%are independent of the choice of $K$, and we may assume to be of the form~$K= G(\widehat{\Z})$. 

The main hypothesis in Theorem~\ref{main theorem 2} is the following.  
%The passage to a subgroup in Definition~\ref{defi:Tate bis} will be necessary: for example we will
%need to replace $U$ by a subgroup which is a product.

\begin{definition}\label{defi:Tate bis}We say that~$x_0$ ``satisfies the uniform integral Tate conjecture'' if~$U=\rho_{x_0}(\Gal(\ol{E}/E))$ 
``satisfies the uniform integral Tate'' property with respect to~$M$,~$G$ in the sense of Def.~\ref{defi:Tate}. 
\end{definition}

%We say integral for the second property, and uniform because~$M(D)$ depends only on~$D$ (and~$U$), not on~$p$ or~$U'$.

%Remark: we have semisimplictiy for every~$\Q_p$ and for uniformly almost every~$\F_p$.

\subsubsection{}
We will make use of the following terminology.
\begin{definition}\label{defi:indep}
We say that a subgroup~$U\leq M(\AAA_f)$ satisfies the~$\ell$-independence property if it is of the form
\[
U=\prod_p U_p
\]
with~$U_p\leq M(\QQ_p)$ for every prime~$p$.
\end{definition}
%\begin{rem}\label{rem:indep}
%The subgroups~$U$ we will apply the definition to, will be images of Galois representations\footnote{Associated to~$x_0$ in the sense
%of~\cite[\S4]{RY}.}~$Gal(\ol{E}/E)\to M(\AAA_f)$. Moreover, these~$U$s will be contained in~$K$. If we pass to a (finite index) open subgroup~$K'$, there will exist~$K''\leq K'$ open (and of finite index) and of the form~$K''=\prod_p K''_p$. Then~$U''=K''\cap U=\prod_p (U_p\cap K'')$ will satisfy the~$\ell$-independence property.
%\end{rem}

\section{Proof of the main result}\label{sec:proof}
\addtocontents{toc}{\protect\setcounter{tocdepth}{1}}
The structure of the proof of Th.~\ref{main theorem 2} is essentially the same as in~\cite{RY}.
The main difference is that our hypothesis is the integral uniform Tate property 
instead of the ``weakly adélic Mumford-Tate conjecture''. Using the results of
\S\S\ref{sec:functoriality},\ref{sec:bounds}, we may follow the same proof as~\cite[\S7]{RY} making the following changes.

\subsection{}
In the step ``reduction to the Hodge generic case''~\cite[7.1.1]{RY} we make the following changes.

Since we work with geometric Hecke orbits~$\cH^{g}(x_0)$ instead of generalised Hecke orbits~$\cH(x_0)$, we use~\cite[Cor.~2.7]{RY} 
to remark that, with~$\Sigma^g=\cH^g([x_0,1])\cap Z$ the following set is a finite union
\[
\Sigma'^g:=\stackrel{-1}{\Psi}(\Sigma^g)=\cH^g([x'_1,1])\cup\ldots\cup \cH^g([x'_k,1])
\]
of geometric Hecke orbits in~$Sh_{K\cap G'(\AAA_f)}(G',X')$. According to Prop.~\ref{Tate:invariance} and Prop.~\ref{Tate:subdatum}, each of the~$[x'_1,1]$, \ldots, $[x'_k,1]$ satisfy the uniform integral Tate conjecture (relative to~$M$ and~$G'$).

 We replace~``On the other hand, the Mumford-Tate hypothesis [...]'' by the observation that if a point of the geometric Hecke orbit~$\cH^g([x_0,1])$ in $Sh(G,X)$ satisfies the uniform integral Tate conjecture (relative to~$M$ and~$G$), then, by Prop.~\ref{Tate:subdatum}, each point of each of the geometric Hecke orbits~$\cH^g([x'_1,1]),\ldots,\cH^g([x'_k,1])$ satisfy the uniform integral Tate conjecture (relative to~$M$ and~$G'$).

%DEPLACER ARGUMENT :
% Indeed, let us define~$\iota:G'\to G$ the injection and~$x_1=x'_1\circ\iota:\SSS\to G'\to G$. By~\cite[Prop.~4.3]{RY} we have~$\rho_{x_1}=\rho_{x'_1}\circ \iota$. It follows the image~$U'$ of~$\rho_{x'_1}$ and~$U$ of~$\rho_{x_1}$ are the same, seen as subgroup of~$G'(\AAA_f)\leq G(\AAA_f)$ respectively, and for~$p$ big enough we will have~$U'(p)=U(p)$ as subgroups of~$G'(\FF_p)\leq G(\FF_p)$ respectively. That~$x_1$ satisfies the Tate conjecture follows from
%\[
%Z_{G'_{\Q_p}}({U'_p}^{0})=
%G'\cap Z_{G_{\Q_p}}({U_p}^0)=
%G'\cap Z_{G}(M)_{\Q_p}=Z_{G'}(M)_{\Q_p}=
%\]
%and likewise over~$\FF_p$.
\subsection{} In the step ``reduction to the adjoint datum''~\cite[7.1.2]{RY} we make the following changes.

Instead of~``Using §3, the Mumford-Tate hypothesis will still be
valid even [...]'', we use~Prop.~\ref{Tate:invariance}.

Instead of~``In view of § 7, the Mumford-Tate hypothesis [...]'' we use Prop.~\ref{indep:subdatum}.

\subsection{} In~\cite[7.1.3]{RY}, ``Induction argument for factorable subvarieties'', we make the following changes.

Instead of~``As explained in § 7, the Mumford-Tate hypothesis [...]'' we use Prop.~\ref{Tate:products}.

\subsection{} The last change from~\cite[\S\S 3.1--3.3, 7]{RY} is in~\cite[7.2.3]{RY} where we use our Th.~\ref{Galois bounds},
instead of the lower bound on the size of Galois orbits~\cite[Th.~6.4]{RY}.

We may apply Th.~\ref{Galois bounds} to the Galois image~$U$, because: the hypothesis on~$M^{ab}$ is satisfied for Galois
images (cf.~\cite[Lem.~\S6.11]{RY}); the other hypotheses are satisfied by assumption. (In the case of Shimura data of abelian type,
see~\S\ref{Tate:abelian type}.)

\section{Functoriality of the Tate condition and independence condition}\label{sec:functoriality}
\addtocontents{toc}{\protect\setcounter{tocdepth}{1}}
%\ref{defi:Tate}
In this section, we verify that the conditions in definition \ref{defi:Tate} and~\ref{defi:indep} are preserved by various natural
operations. This is necessary to make simplifying assumptions in the proof of the main theorems (cf. \cite[\S7.1]{RY}).
We also show that the conditions of \ref{defi:Tate} and~\ref{defi:indep} hold for all Shimura varieties of abelian type.

According to~Remark~\ref{rem:Tate}, the definition~\ref{defi:Tate} does not depend on~$\rho_G$.
It follows from definition~\ref{defi:indep}, that the property that the Galois image satisfies the~$\ell$-independence property
does not depend on~$G$, nor on~$\rho_G$.

\subsection{Invariance on the geometric Hecke orbit}
We refer to~\cite[Def.~2.1]{RY} for the notion of geometric Hecke orbit~$\cH^g(x_0)$, the variety~$W=G\cdot \phi_0$ and the notation~$x_\phi=x_0\circ\phi$ for~$\phi\in W(\Q)$.
\begin{proposition}\label{Tate:invariance}
Let~$x_\phi\in \cH^g(x_0)$ for some~$\phi\in W(\Q)$.

If~$x_0$ ``satisfies the uniform integral Tate conjecture'',
then~$x_\phi$ ``satisfies the uniform integral Tate conjecture''.

If the image~$U$ of~$\rho_{x_0}$ satisfies the~$\ell$-independence 
property in the sense of~Def.~\ref{defi:indep}, then the image~$U'$ of~$\rho_{x_\phi}$ satisfies the~$\ell$-independence 
property in the sense of~Def.~\ref{defi:indep}.
\end{proposition}

\begin{proof}Let~$g\in G(\ol{\QQ})$ be such that~$\phi=g\phi_0g^{-1}$, and let~$L$ be a number field such that~$g\in G(L)$,
and let~$\AAA_{f,L}=\AAA_f\tens_\QQ L$ be the ring of ad\`eles of the number field~$L$. Denote by~$U'$ the image of~$\rho_{x_\phi}$.
According to~\cite[Prop.~3.4]{RY} we have
\[
U'=\phi(U).
\]

We prove the assertion about~$\ell$-independence. Assume~$U=\prod_p U_p$.   %According to~\cite[Prop.~4.3]{RY}, we have~$U'=\phi(U)$.
 Since $\phi$ is defined over $\QQ$, we have
\[
U'=\prod_p \phi(U_p).
\]
This proves the assertion about~$\ell$-independence.

We treat the semisimplicity over~$\Q_p$ in Def.~\ref{defi:Tate}.  Assume that the action of~$U_p$ is semisimple. Equivalently, the Zariski closure~$\ol{U_p}^{Zar}$  is reductive. As~$\phi$ is defined over~$\QQ$, the algebraic group~$\ol{\phi(U_p)}^{Zar}=\phi(\ol{U_p}^{Zar})$
is reductive, or equivalently, the action of~$U'_p=\phi(U_p)$ is semisimple.

%Since $\phi$ is defined over $\QQ$, we see that the semisimplicity condition is satisfied for $U'$.

We now treat the centraliser property of part~\ref{defi:Tate1} of Def.~\ref{defi:Tate}.
For every prime~$p$, let us choose an embedding~$L\to\ol{\Q_p}$. We have~$gmg^{-1}=\phi(m)$ for~$m\in M(\ol{\Q_p})$, and thus~$U'_p=gU_pg^{-1}\leq \phi(M)(\ol{\Q_p})$. We have
%we have\footnote{A priori, the algebraic groups in~\eqref{comment} are defined over~$\ol{\Q_p}$, but since~$\phi=g\cdot \phi_0:M\to G$ is defined over~$\Q_p$, these are actually algebraic  groups defined over~$\Q_p$. }
\begin{multline}\label{comment}
Z_{G_{\ol{\Q_p}}}(U'_p)=gZ_{G_{\ol{\Q_p}}}(U_p)g^{-1}\\
=gZ_{G_{\ol{\Q_p}}}({U_p}^0)g^{-1}=
gZ_{G}(M)_{\ol{\Q_p}}g^{-1}=Z_{G}(gMg^{-1})_{\ol{\Q_p}}=Z_{G}(\phi(M))_{\ol{\Q_p}}
\end{multline}
As~$\phi(M)$ is the Mumford-Tate group of~$x_\phi$ we have proved~(1) of~Def.~\ref{defi:Tate} for~$x_\phi$.

We now deal with part \ref{defi:Tate2} of Def.~\ref{defi:Tate}.

Note that the component~$g_p$ of~$g$ as an adélic element is in~$G(O_{L\tens\Q_p})$ for~$p$ large enough.
For~$p$ large enough, the group~$\phi(M)(\ZZ_p)=gM(\ZZ_p) g^{-1}$ is hyperspecial and the reduction map~
\[
g\mapsto \ol{g}:\phi(M)(\ol{\ZZ_p})\to \phi(M)(\ol{\FF_p})
\]
is well defined. 

Let~$m_0\in\Z_{\geq1}$ be such that the above apply for~$p\geq m_0$.
Let~$D$ and~$M(D)$ be as in (2) of Def.~\ref{defi:Tate}, and let~$V\leq U_p$ be a subgroup of index~$[U_p:V]\leq D$,
and denote by~$V(p)\leq U(p)$ the corresponding subgroup as in  (2) of Def.~\ref{defi:Tate}, and define~$V':=g V g^{-1}\leq \phi(U)$ and~$V'(p)$ accordingly. 
Then, for~$p\geq M'(D):=\max\{m_0;M(D)\}$ we have
\[
V'(p)=\ol{g_p}V(p)\ol{g_p}^{-1}
\]
with~$\ol{g}$ the reduction of~$g$ in~$G(\kappa_L)\leq G(\ol{\FF_p})$ , where~$\kappa_L$ is the residue field of~$L$ at a prime above~$p$.
The semisimplicity follows.

For~$p\geq M'(D)$, we also have
\begin{multline*}
Z_{G_{\ol{\FF_p}}}(V'(p))=\ol{g_p}Z_{G_{\ol{\FF_p}}}(V(p))\ol{g_p}^{-1}=
\ol{g_p}Z_{G_{\ol{\FF_p}}}(M)\ol{g_p}^{-1}\\
=Z_{G_{\ol{\FF_p}}}(\ol{g_p}M\ol{g_p}^{-1})=Z_{G_{\ol{\FF_p}}}(\phi(M)).\qedhere
\end{multline*}
\end{proof}

\subsection{Passing  to a Shimura subdatum}\label{passage to sub}

\begin{proposition}\label{Tate:subdatum}
Let~$\Psi:(G',X')\to (G,X)$ be an injective morphism of Shimura data.

Let~$x_\phi\in \cH^g(x_0)$ be such that there exists~$x'_\phi$ such that~$x_\phi=\Psi\circ x'_\phi$.

If~$x_0$ ``satisfies the uniform integral Tate conjecture'', 
then~$x'_\phi$ (and~$x_\phi$)``satisfy the uniform integral Tate conjecture''.

If the image~$U$ of~$\rho_{x_0}$ satisfies the~$\ell$-independence 
property in the sense of~Def.~\ref{defi:indep}, then the image~$U'$ of~$\rho_{x'_\phi}$ satisfies the~$\ell$-independence 
property in the sense of~Def.~\ref{defi:indep}.
\end{proposition}
\begin{proof} By Prop.~\ref{Tate:invariance} we may assume~$x_\phi=x_0$. 
We identify~$G'$ with its image in~$G$.

By~\cite[Prop. 3.4]{RY} we have
\[
U=\Psi(U')=U'
\]
where we denote by~$U'$ the image of~$\rho_{x'_0}'$ in $G'(\A_f)$. 

The semisimplicity of the action of $U'$ is automatic.
It follows readily from the definitions and the remark
that
\[
Z_{G_{\Q_p}'}(U'_p)=Z_{G_{\Q_p}}(U'_p)\cap G_{\Q_p}'=Z_{G_{\Q_p}}(M)\cap G_{\Q_p}'=Z_{G_{\Q_p}'}(M).
\]
and similarly for~$\FF_p$ for~$p$ big enough.% so that~$G_{\Q_p}'$ is hyperspecial.

The last statement follows from
\[
U'=\Psi(U')=U=\prod_p U_p.\qedhere
\]
\end{proof}

\subsection{Passing to quotients by central subgroups}\label{passage to quotients}

\begin{proposition}\label{indep:subdatum}
Let $F \subset Z(G)$ ($Z(G)$ is the centre of $G$) be a subgroup and let $G'$ be the quotient $G/F$.
Let~$\Psi:(G,X)\to (G',X')$ be the morphism of Shimura data induced by the quotient $G\lto G'$.

%Choose a faithful representation $\rho' \colon G' \lto \GL(d')$.

If~$x_0$ ``satisfies the uniform integral Tate conjecture'', 
then~$x'_0=\Psi\circ x_0$ ``satisfies the uniform integral Tate conjecture''.

If the image~$U$ of~$\rho_{x_0}$ satisfies the~$\ell$-independence 
property in the sense of~Def.~\ref{defi:indep}, then the image~$U'$ of~$\rho_{x_0'}$ satisfies the~$\ell$-independence 
property in the sense of~Def.~\ref{defi:indep}.
\end{proposition}
\begin{proof}
Arguments are similar. Firstly, by~\cite[Prop.~3.4]{RY} we have
\[
U'=\Psi(U).
\]
We use, remarking~$\Psi(M)$ is the Mumford-Tate group of~$x'_0$,
\[
Z_{G^{ad}}(U'_p)=Z_{G}(U_p)/F=Z_{G}(M)/F=Z_{G^{ad}}(M/Z(G))=Z_{G^{ad}}(\Psi(M)).
\]
(For~$p$ big enough~$\Psi$ will be compatible with the integral models.) Let~$H'_p$ be the Zariski closure of~$U'_p$.
Observe that~$\Psi(U_p)\leq U'_p\leq \Psi(H_p)$, and~$\Psi(H_p)\leq G'_{\Q_p}$ is closed, and~$\Psi(U_p)\leq \Psi(H_p)$ is Zariski dense. Thus~$H'_p=\Psi(H_p)$. As~$H_p$ is reductive, so is~$H'_p$. This proves the semisimplicity~\eqref{defi:tate eq 1.2}.

Finally, if~$U=\prod_p U_p$, then~$U'=\prod_p \Psi(U_p)$. This proves the assertion
about~$\ell$-independence.

For the semisimplicity property over~$\F_p$, let~$\rho:G\to GL(n)$ and~$\rho':G'\to GL(m)$ be faithful representations, and~$G_{\F_p}$ and~$G'_{\F_p}$ be the corresponding models over~$\F_p$, which are reductive for~$p\geq c_1(G,G')$. Let~$U'\leq  U(p)$ be of index~$[U':U(p)]\leq D$. By Def.~\ref{defi:Tate}, the representation~$U(p)\to GL(n)_{\F_p}$ is a semisimple representation for~$p\geq M(D)$. If furthermore~$p\geq c_2(\rho)$,
by~\cite[Th.~5.4 (ii)]{SCR}, the subgroup~$U'\leq G$ is~$G$-cr, for~$p\geq c_2(\rho)$. If furthermore~$p\geq c_3(\rho'\circ \Psi)$, by~\cite[Th.~5.4 (i)]{SCR}, the representation~$\rho'\circ \Psi:U'\to GL(m)$ is semisimple. This proves that, for~$p\geq M'(D):=\max\{M(D);c_1(G,G');c_2(\rho);c_3( \rho'\circ \Psi)\}$, the representation~$U'\to GL(m)$ is semisimple. This proves \eqref{defi:tate eq 2.2} of Def.~\ref{defi:Tate} with~$p>M'(D)$.
 %the~\ref{rem2} of Remark~\ref{rem:Tate} in order apply Nori theory, and the remark below.
\end{proof}

%. Thus~$U'=\prod_p U'_p$ with~$U'_p=\phi(U_p)$.

%\begin{rem} The subgroup~$U(p)^\dagger$ used by Nori is the~$p$-Sylow of~$U(p)$.  From~Cor~\ref{Sylow} one can deduce that the Nori group of~$\Psi(U(p))$ is~$\Psi(H)$.
%\end{rem}

\begin{rem}
%This proposition shows that, when one pass to the case o
This proposition in particular shows that we can restrict ourselves to the case of 
Shimura varieties where $G$ is semisimple of adjoint type (by taking $F = Z(G)$ in this proposition).
\end{rem}

\subsection{Compatibility to products.}

\begin{proposition}\label{Tate:products}
Assume $G$ to be of adjoint type and not simple.
Let
$$
(G,X) = (G_1, X_1) \times (G_2, X_2)
$$
be a decomposition of $(G,X)$ as a product. We denote~$\pi_1:G\to G_1$ and~$\pi_2:G\to G_2$ the projection maps.

If~$x_0$ ``satisfies the uniform integral Tate conjecture'', 
then~$x_i=\pi_i\circ x_0$ ``satisfies the uniform integral Tate conjecture''.

If the image~$U$ of~$\rho_{x_0}$ satisfies the~$\ell$-independence
property in the sense of~Def.~\ref{defi:indep}, then the image~$U_i$ of~$\rho_{x_i}$ satisfies the~$\ell$-independence 
property in the sense of~Def.~\ref{defi:indep}.
%If $(G,X)$ satisfies conditions \ref{defi:Tate} and \ref{defi:indep}, then so do the $(G_i,X_i)$.
\end{proposition}
The proof is the same above. We recall that~$\rho_{x_i}=\pi_i\circ \rho_{x_0}$ by~\cite[Prop.~3.4]{RY}.
We also recall that the Mumford-Tate group of~$x_i$ is~$M_i=\pi_i(M)$, and that
\[
G_i\cap Z_{G}(M)=Z_{G_i}(M_i).
\]

%\subsection{Remonter isogénie} Attention~$\ell$-independance.

\subsection{Shimura varieties of abelian type}\label{Tate:abelian type}

%\begin{proposition}
%All Shimura varieties of abelian type and adjoint type satisfy conditions of
%\ref{defi:Tate} and \ref{defi:indep}.
%\end{proposition}

\begin{proposition}
All Shimura varieties of abelian type and adjoint type satisfy conditions of
Def.~\ref{defi:Tate bis} and Def.~\ref{defi:indep}.

More precisely, let~$(G,X)$ be a Shimura datum of abelian type with~$G$ of adjoint type,
and let~$S$ be an associated Shimura variety. Then for every~$s_0=[x_0,1]\in S$, 
\begin{itemize}
\item the point~$x_0$ satisfies the uniform integral Tate conjecture,
\item and there exists a field of finite type~$E$ over~$\QQ$ such that~$U=\rho_{x_0}(Gal(\ol{E}/E))$
satisfies the~$\ell$-independence condition.
\end{itemize}
\end{proposition}
\begin{proof} By definition of abelian type Shimura data (\cite[\S3.2]{Upr}, \cite[Prop.~2.3.10]{Deligne}), there exists an isomorphism of Shimura data
\[
(G,X)\simeq ({G'}^{ad},{X'}^{ad})
\]
with~$(G',X')$ of Hodge type. Using Prop.~\ref{Tate:invariance} we may replace~$s_0=[x_0,1]$
by any point of its geometric Hecke orbit, and assume~$x_0$ belongs to the image of~$X'$ in~$X\simeq {X'}^{ad}$: there exists~$x_0'\in X'$ such that~$x_0={x'_0}^{ad}$.

According to~\S\ref{passage to quotients}, it is enough to prove the conclusion for~$x_0'$ instead of~$x_0$.

By definition of Hodge type data, there exists an injective morphism of Shimura data~$(G',X')\to (GSp(2g),\mathfrak{H}_{g})$, the latter being the Shimura datum
of the moduli space~$\mathcal{A}_g$. Let~$\tau_0$ be the image of~$x'_0$ in~$\mathfrak{H}_{g}$.

According to~\S\ref{passage to sub}, we may assume~$(G,X)=(GSp(2g),\mathfrak{H}_{g})$ and~$x_0=\tau_0$. It then follows from Th.~\ref{Serre} and Cor.~\ref{coro4.12}.
\end{proof}

\subsection{Uniform integral Faltings' theorem over fields of finite type}
\begin{theorem}[Faltings]\label{Faltings}
Let~$K$ be a field of finite type over~$\QQ$, and let~$A/K$ be an abelian variety.

Fix an algebraic closure~$\overline{K}$ of~$K$. Denote
\[
T_A\approx \widehat{\Z}^{2\dim(A)}
\]
the~$\widehat{\ZZ}$-linear Tate module, on which we have a continuous~$\widehat{\Z}$-linear
representation
\begin{equation}\label{Faltings rep}
\rho=\rho_A:\Gal(\overline{K}/K)\to GL_{\widehat{\Z}}(T_A).
\end{equation}
We assume that~$\End(A/K)=\End(A/\overline{K})$ and we let~$\End(A/K)$ act on~$T_A$ and denote
\[
Z:=\{b\in\End_{\widehat{\ZZ}}(T_A)|\forall a\in\End(A/K),[b,a]=0\}%\leq \End_{\widehat{\Z}}(T_A)\approx Mat({2\dim(A)},\widehat{\ZZ})
\] 
the
$\widehat{\ZZ}$-algebra which is the centraliser  of~$\End(A/K)$ in~$\End_{\widehat{\Z}}(T_A)\approx Mat({2\dim(A)},\widehat{\ZZ})$.

We denote the image of~$\rho$ by
\begin{equation}\label{U in Falt}
U:=\rho(Gal(\ol{K}/K)).
\end{equation} 

Then, for every~$d\in\Z_{\geq1}$, there exists some~$M(A,K,d)\in \Z_{\geq1}$ such that:
for every open subgroup~$U'\leq U$ of index at most~$d$, we have
\[
\widehat{\ZZ}[U']\leq Z
\]
is an open subalgebra of index at most~$M(A,K,d)$.
\end{theorem}
%This statement follows from Faltings' theorems. A reference in the case where~$K$ is a number field is~\cite[Th.~1, Cor.~1.1]{MW}.

\subsubsection{The number field case}\label{Number field case}
This statement follows from Faltings' theorems. In the case where~$K$ is a number field, we deduce it from~\cite[Th.~1, Cor.~1, Th.~2]{MW} as follows.
\begin{proof}[Proof of Th.~\ref{Faltings} if~$K$ is a number field] Consider~$d\in\Z_{\geq1}$, and let~$M'(A,K,d)$ be the~$M$ of~\cite[Th.~1]{MW}, with their~$d$ our~$d\cdot [K:\Q]$.
Consider~$U'$ as in the statement of Theorem~\ref{Faltings}, denote by~$\Gamma:=\stackrel{-1}{\rho_{A}}(U')\leq \Gal(\ol{K}/K)$ its inverse image, and denote by~$k=\ol{K}^{\Gamma}/K$ the corresponding finite field extension.
For every prime~$\ell$, we denote by~$U'(\ell)$ be the image of~$U'$ in~$E_\ell:=\End_{\F_\ell}(A[\ell])$.

We note that (cf. \cite[Th. 2.7]{DeligneBBK})
\[
\widehat{\Z}[U']=\prod_{\ell}\Z_\ell[U']\text{ in }\End_{\widehat{\Z}}(T_A).
\]

We will prove the following.
\begin{enumerate}
\item	\label{MW 1}	 For every prime~$\ell$, we have %the finiteness of the quantity
\[
M_\ell:=\max_{[U:U']\leq d}[Z\tens\Z_\ell:\Z_\ell[U']]<+\infty.
\]
\item 	\label{MW 2}	There exists~$M_A\in\Z_{\geq1}$ such that for every prime~$\ell>\max\{M_A;M'(A,K,d)\}$,
\[
[Z\tens\Z_\ell:\Z_\ell[U']]=1.
\]
\end{enumerate}
The conclusion of Theorem~\ref{Faltings} will then follow with
\[
M(A,K,d):=\prod_{\ell\leq \max\{M_A;M'(A,K,d)\}}M_\ell.
\]

We prove~\eqref{MW 1}. Fix a prime~$\ell$. We recall that~$\End(A/\ol{K})=\End(A/K)\leq \End(A/k)\leq \End(A/\ol{K})$.
By Faltings theorems, for~$A/K$ and~$A/k$, we have
\[
\Q_\ell[U]=\Q_\ell[U']=Z\tens\Q_\ell
\]
in~$\End_{\Q_\ell}(V_\ell)$, where~$V_\ell=T_A\tens_{\widehat{\Z}}\Q_\ell$ is the~$\Q_\ell$ Tate module of~$A$.

It follows that
\[
\Z_\ell[U']\leq \Z_\ell[U]\leq Z\tens\Z_\ell
\]
is of finite index.

We denote by~$U_\ell$ and~$U'_\ell$ the image of~$U$ and~$U'$ in~$\End_{\Q_\ell}(V_\ell)$.
We have~$[U_\ell:U'_\ell]\leq [U:U']\leq d$ and
\[
\Z_\ell[U]=\Z_\ell[U_\ell]\text{ and }
\Z_\ell[U']=\Z_\ell[U'_\ell].
\]
By~\cite[Ch. III,  \S4.1, Prop. 8 and 9]{SerreCG} and~\cite[Prop.~2]{Serre-62}, the group~$U_\ell$ has at most finitely many
open subgroup~$U'_\ell$ of index at most~$d$. We also have~$[U_\ell:U'_\ell]\leq [U:U']$. Thus, as~$U'$ ranges through subgroups of~$U$ of index at most~$d$, we have the finiteness
\[
M_\ell:=\max_{[U:U']\leq d}[Z\tens\Z_\ell:\Z_\ell[U']]<+\infty.
\]
This implies~\eqref{MW 1}.

We prove~\eqref{MW 2}. We define
\[
B:=\wh{\Z}[U'],\quad D:=\End(A/K).
\]
and denote the images of~$B$, resp.~$D$, in~$E_\ell:=\End_{\F_\ell}(A[\ell])$ by
\[
B_\ell=\F_\ell[U'(\ell)]\text{ and } D_\ell.
\] 

When~$\ell>M'(A,K,d)$:
\begin{itemize}
\item the subalgebra~$B_\ell:=\F_\ell[U'(\ell)]\leq E_\ell$ is semisimple, by~\cite[Th.~2]{MW},
\item and its centraliser
\[C_\ell:=Z_{E_\ell}(B_{\ell})=\End_{U'(\ell)}(A[\ell])=\End_{U'}(A[\ell])\]
satisfies, by~\cite[Cor.~1]{MW},
\[C_\ell=\End(A/K)\tens\F_\ell.\]
\end{itemize}
By the double centraliser theorem~(cf.~\cite[VIII\S5, Th.~3]{BBKA8}), we have
\[
B_\ell=Z_{E_\ell}(C_\ell).
\]
Recall that~$D_\ell$ is the image of~$\End(A/K)$ in~$E_\ell$. Let~$Z_\ell:=Z_{E_\ell}(D_\ell)$ be its centraliser.
   
We claim that there exists~$M_A$ such that for~$\ell>M_A$, the map
\begin{equation}\label{Z mod l map}
Z\tens\F_\ell\to Z_{E_\ell}(D_\ell)
\end{equation}
is an isomorphism. 
\begin{proof}[Proof of the claim]

Let us recall a fact. For any subgroup~$\Lambda\leq \Z^n\leq \Q^n$, the map~$\Lambda\tens\F_\ell\to \Z^n\tens\F_\ell$ is injective if and only if~$\ell\nmid[\Lambda\cdot \Q\cap \Z^n:\Lambda]<\infty$. This holds for~$\ell\gg0$, and, if~$\Lambda$ is primitive, for every~$\ell$.

We choose an embedding~$\ol{K}\to\CC$ and consider the Betti cohomology~$H_1(A/\CC;\Z)\approx \Z^{2g}$. We define~$E_0=\End_\Z(H_1(A/\CC;\Z))\approx\Z^{{(2g)}^2}$, and~$Z_0=Z_{E_0}(\End(A/K))$. 

Let us prove the injectivity of~\eqref{Z mod l map}. 
We apply the recalled fact for~$n={(2g)}^2$ and~$\Lambda=Z_0$: the map
\[
Z_0\tens \F_\ell\to E_\ell.
\]
is injective for~$\ell\gg0$ (actually, for every prime~$\ell$). We note that~$Z\simeq Z_0\tens\wh{\Z}$. Thus~$Z\tens \F_\ell\simeq Z_0\tens\F_\ell$, and the map~\eqref{Z mod l map} is injective for~$\ell\gg0$.

Let us prove the surjectivity of~\eqref{Z mod l map}, for~$\ell\gg0$. As we proved that~\eqref{Z mod l map} is injective, it is enough to prove that
\begin{equation}\label{dim Z ell}
\forall\ell\gg0,~\dim_{\F_\ell} Z_0\tens \F_\ell=\dim_{\F_\ell} Z_\ell.
\end{equation}
We may write~$D=\Z\cdot d_1+\ldots+\Z\cdot d_i$. Denote by~$\ad(d):E_0\to E_0$ the map~$e\mapsto [d,e]=d\circ e-e\circ d$.
We choose an isomorphism~$E\simeq \Z^{(2g)^2}$, and, for~$1\leq j\leq i$, denote the coordinates of~$ad(d_j)$ by
\[
ad(d_j)=(\delta_{j,1},\ldots,\delta_{j,(2g)^2}).
\]
We denote by~$W\leq E_0^\vee:=\Hom(E_0,\Z)$ the subgroup generated by~$\{\delta_{j,h}|1\leq j\leq i,1\leq h\leq (2g)^2\}$.
Then~$Z_0\leq E_0$ is the dual~$\{d\in E_0|\forall w\in W, w(d)=0\}$ of~$W$, and~$Z_{E_\ell}(D_\ell)\leq E_{\ell}$ is the dual of the image~$W_\ell$ of~$W$ in~$E_\ell^\vee$. We
have thus~$\dim_{\F_\ell}(Z_0\tens\F_\ell)=(2g)^2-\dim_{\F_\ell} W\tens\F_\ell$. By the recalled fact,
we have~$\dim_{\F_\ell}(W_\ell)=\dim_{\F_\ell} W\tens\F_\ell=\operatorname{rank}(W)$ for~$\ell>M_A:=[W\cdot \Q\cap E_0:W]$. This implies~\eqref{dim Z ell}.
We have proved the claim.
\end{proof}

It follows that, when~$\ell>\max\{M_A;M(A,K,d)\}$ the map
\[
\wh{\Z}[U']\to Z\to Z_\ell =\F_\ell[U'(\ell)] 
\]
is surjective. By Nakayama's lemma~\cite[VIII\S9.3 Cor.~2]{BBKA8}, one deduces that
\[
\Z_\ell[U']=Z\tens\Z_\ell.
\]
We have proven~\eqref{MW 2}. This concludes our proof of Th.~\ref{Faltings} if~$K$ is a number field.
\end{proof}
\subsubsection{General case}
Because we lack a reference, we give a specialisation argument which reduces the theorem for general fields
of finite over~$\QQ$ to the case of number fields.

In view of~Def.~\ref{defi:Tate} we will prove the following refinement of Th.~\ref{Faltings}.
%We will prove a refinement.
\begin{proposition}\label{prop Falt refined}
The same conclusions holds if we replace~\eqref{U in Falt} by
\begin{equation}\label{U in Falt prod}
U:=\prod_p {U_p}^0
\end{equation}
with~$U_p:=\rho(\Gal(\ol{K}/K))\cap GL_{\Z_p}(T_A\tens\Z_p)$.
\end{proposition}
The proof of Prop.~\ref{prop Falt refined} will use the following results. 
\begin{theorem}[Serre]\label{Serre} In the same situation, assume moreover that~$K$ is a number field. Then
there exists a finite extension~$L/K$ such that Galois image~$U:=\rho(\Gal(\ol{L}/L))$ 
\begin{enumerate}
\item is such that~$U$
satisfies the~$\ell$-independence condition in the sense of Def.~\ref{defi:indep}, (\cite[136. Th.~1, p.34]{S4}, \cite[\S3.1]{SCrit})
\item and such that the~$U_p$ are Zariski connected, (\cite[133.~p.\,15 ;135.~2.2.3 p.\,31]{S4}, \cite[6.14, p\,623]{LP})
\item and such that~$U\leq M(\wh{\Z})$ where~$M$ is the Mumford-Tate group of~$A$. (cf.~\cite[\S3]{RY} and~\S\ref{Tate:abelian type}.)
\end{enumerate} 
\end{theorem}
The assertions in Theorem~\ref{Serre} are found in the indicated references. We deduce Prop.~\ref{prop Falt refined}.
\begin{proof}[Proof of Prop.~\ref{prop Falt refined}] Let~$\eta=\Spec(K)$ and~$\overline{\eta}=Spec(\overline{K})$.  Following~\cite[\S1.2 and Cor.~1.5]{Noot}\footnote{We apply~\cite{Noot} with~$F:=\Q$, and~$F(S):=K$. Our~$F$ is a finite extension of their~$\Q(\sigma)$ such that~$ {\End}(A_\sigma/\ol{\Q(\sigma)})\simeq {\End}(A_\sigma/F)$, and our~$A_F$ is the base change of their~$A_\sigma$.}, 
there exists a number field~$F\leq \ol{K}$, and an abelian variety~$A_F$ over~$F$ such that
\begin{itemize}	
\item We have an identification of Tate modules~$T:=T_A\simeq T_{A_F}$
\item We have an identity (cf.~\cite[Cor.~1.5]{Noot}
\begin{equation}\label{Noot End}
{\End}(A/K)\simeq {\End}(A/F)\simeq {\End}(A/\ol{F})
\end{equation}
as subalgebras of~$B:=\End_{\widehat{\ZZ}}(T)$,
\item we have a diagram
\[
\begin{tikzcd}
Gal(\overline{K}/K)\arrow{d}{\rho}\arrow[hookleftarrow]{r} &  D_F \arrow[twoheadrightarrow]{r} & Gal(\overline{F}/F)\arrow{d}{\rho'}\\
\End(T_A)&\arrow[equals]{l}\End(T)\arrow[equals]{r}& \End(T_{A_F}).
\end{tikzcd}
\]
\end{itemize}
The commutativity implies that~$U_F:=\rho(D_F)$ satisfies
\[
\rho(D_F)=U_F:=\rho'(Gal(\overline{F}/F))
\]
and
\[
\rho(D_F)\leq U:=\rho(Gal(\overline{K}/K)).
\]
By~Th.~\ref{Serre}, after possibly passing to a finite extension of~$F$ and the corresponding finite extension of~$K$,
we may assume
\[
U_F=\prod_p {(U_F)_p}^0.
\]
We note that~$(U_F)_p\leq U_p$ and thus~${(U_F)_p}^0\leq {U_p}^0$.
We deduce
\[
U_F\leq \wt{U}=\prod_p {U_p}^0.
\]

We will prove the refinement Prop.~\ref{prop Falt refined} of Th.~\ref{Faltings} with
\[
M(A,K,d)=M(A_F,F,d).
\]

Fix~$d$ and an open subgroup~$U'\leq \wt{U}$ of index at most~$d$.
We denote
\[
U'_F=U'\cap U_F.
\]
We first note that~$U'_F\leq U_F$ is a subgroup of index at most~$d$.
We have, as~$\widehat{\Z}$-subalgebras of~$B:=Mat({2\dim(A)},\widehat{\ZZ})$,
\[
\widehat{\ZZ}[U'_F]\leq \widehat{\ZZ}[U'].
\]

From~\eqref{Noot End} we have
\[
Z=Z_F:=
\{b\in\End_{\widehat{\ZZ}}(T_A)|\forall a\in\End(A_F/F),[b,a]=0\}.
\]

We use the number field case~\ref{Number field case} of the theorem for~$A_F$ and~$d$ and~$U'_F$ and get
\[
\left[Z_F:\widehat{\ZZ}[U'_F]\right]\leq M(A_F,F,d).
\]
We note that~$\widehat{\ZZ}[U'_F]\leq Z$ because~$U\geq U'$ commutes with the action of~$\End(A)$ (all the endomorphisms
are rational over~$K$).

Finally
\[
\widehat{\ZZ}[U'_F]\leq \widehat{\ZZ}[U']\leq Z
\]
hence
\[
\left[Z:\widehat{\ZZ}[U']\right]\leq 
\left[Z:\widehat{\ZZ}[U'_F]\right]
=
\left[Z_F:\widehat{\ZZ}[U'_F]\right]
\leq M(A_F,F,d).
\]\end{proof}

\begin{corollary}\label{coro Faltings} We consider the setting of~Th.~\ref{Faltings} and denote by~$g$ the dimension of~$A$.
 Choose an isomorphism~$H_1(A;\Z)\simeq \Z^{2g}$ and denote~$\sigma:GL(H_1(A;\Z))\to GL(2g)$ 
the corresponding isomorphism. There exists~$c(A)$ such the following holds.

The subgroup~$U\leq M(\widehat{\Z})$ satisfies the uniform integral Tate property with respect to~$M$,~$G=GL(H_1(A;\QQ))$ and~$\rho_G=\sigma:G\to GL(2g)$ 
in the sense of Def.~\ref{defi:Tate}, with~$M(D):=\max\{c(A);M(A,K,D)\}$.

In the following, we denote~$GSp(H_1(A;\QQ))\approx GSp(2g)$ the $\Q$-algebraic group of symplectic similitudes of the Riemann symplectic form on~$H_1(A;\QQ)$.

The subgroup~$U\leq M(\widehat{\Z})$ satisfies the uniform integral Tate property with respect to~$M$,~$G'=GSp(H_1(A;\QQ))$ and~$\rho_{G'}={\rho_G}\restriction_{G'}:G' \to GL(2g)$ 
in the sense of Def.~\ref{defi:Tate} with~$M(D):=\max\{c(A);M(A,K,D)\}$.
\end{corollary}
We only treat the case of~$GL(2g)$, as the case of~$GSp(2g)$ follows directly using Remark~\ref{rem:Tate}~\eqref{rem g g' tate}.
\begin{proof}Thanks to Lemma~\ref{lem U ast}, we may assume~$U=\prod_p {U_p}^0$. We have then
\[
\wh{\Z}[U]=\prod_p\Z_p[{U_p}^0].
\]
We use Prop.~\ref{prop Falt refined}. Then
\begin{itemize}
\item for every~$p$, the algebra~${\Q}_p[U_p]={\Q}_p[{U_p}^0]$ is the commutant of~$\End(A/K)\tens\Q_p=Z_{\End_{\Q_p}(H^1(A;\Q_p))}(M)$. This implies~\eqref{defi:tate eq1}. Because~$\End(A/K)\tens\Q_p$ is a semisimple algebra, so is it commutant in~$\End_{\Q_p}(H^1(A;\Q_p))$ and thus the action of~$U_p$ is semisimple.
\item for every~$D$, and every~$p\geq M(A,K,D)$, and every~$U'\leq U(p)$ of index at most~$D$,
we have~$\Z_p[U']=Z\tens\Z_p$, with~$Z$ as in~\eqref{Z mod l map}. Then the algebra~${\F}_p[U']$ is~$Z\tens\F_p$.
For~$p\gg 0$, depending only on~$\End(A/K)$, the image~$Z\tens\F_p$ of~$Z$ is the commutant of~$\End(A/K)\tens\F_p$, by~\eqref{Z mod l map}. If moreover~$p\gg0$,  depending only on~$\End(A/K)$, the algebra~$\End(A/K)\tens \F_p$ is semisimple,
its action on~$\End_{\F_p}(A[p])$ is semisimple, and the action of its centraliser~$\F_p[U']$ is semisimple.\qedhere
\end{itemize}
\end{proof}

We deduce the following, using the well-known relation between Galois representations on the Tate module,
and Galois action on isogeny classes in the Siegel modular variety~$\mathcal{A}_g$ (see~\cite{UY}).
\begin{corollary}\label{coro4.12} Let~$s_0$ be a point in~$\mathcal{A}_g=Sh_{GSp(2g,\wh{\Z})}(\mathfrak{H}_g,GSp(2g))$.
Then~$s_0$ satisfies Def.~\ref{defi:Tate bis}.
\end{corollary}
%
%DÉPLACÉ
%
%More precisely (see for instance~\cite{RPhD3}), this~$\tau_0$ induces an isomorphisms
%\[
%H^1(A;\Z)\simeq {\Z}^{2g}\qquad H^1(A;\wh{\Z})\simeq \wh{\Z}^{2g}
%\]
%where~$A$ is the associated abelian variety and~$g=\dim(A)$. We identify~$GL(H^1(A;\Z))\simeq GL(2g)$.
%It also induces an isomorphism~$GSp(2g)$ with the symplectic group
%of~$H^1(A;\Q)$ as a Hodge structure, and~$Gal(\ol{E}/E)\xrightarrow{\rho_{\tau_0}}GSp(2g)\to GL(2g)$
%is the action on the Tate module as in Th.~\ref{Faltings}. The Mumford-Tate group of~$\tau_0$ is then identified with Mumford-Tate group of~$A$,
%and we have~$Z_{GL(2g)_{\Q}}(M)=\End(A)\tens\Q^\times$, where~$\End(A)$ acts faithfully on~$H^1(A;\Q)$.
%
%Then Th.~\ref{Faltings}, the Galois image~$U:=\rho_{\tau_0}(\ol{E}/E)\leq M(\hat{\Z})$ ``satisfies the uniform integral Tate'' property with respect to~$M$,~$GL(H^1(A;{\Q}))$ and~$\rho:GL(H^1(A;\wh{\Z}))\simeq GL(2g)$. A fortiori~$U$ ``satisfies the uniform integral Tate'' property with respect to~$M$,~$GSp(2g)$ and~$\rho:GSp(2g)\simeq GL(2g)$''.

% We first embed, using the basis given by~$\tau$,
%\[
%GSp(2g)(\wh{\Z})\to GL_{\widehat{\Z}}(T_A);
%\]
%we identify~$\rho_{x_0}:Gal(\overline{K}/K)\to GSp(2g)(\wh{\Z})$ with~\eqref{Faltings rep}.

\section{Polynomial Galois bounds}\label{sec:bounds}

This section is at the heart of this paper.
We obtain suitable lower bounds for Galois orbits of points in generalised Hecke orbits under much weaker assumptions than
those made in \cite{RY} (in particular, as seen above, they are satisfied by all Shimura varieties of abelian type).
\subsection{Statement }
We use the notations~$\succcurlyeq$ and~$\approx$ of~\cite[Def.~6.1]{RY} for polynomial domination and polynomial equivalence
of functions. For the definition of~$H_f(\phi)$ we refer to~\cite[App. B]{RY}. For the MT property we refer to~\cite[\S5, Def. 6.1]{RY},
and refer to~\cite[\S 6.4]{RY} for the fact that~\eqref{Galois bound 1} is satisfied for Galois images.

\begin{theorem}\label{Galois bounds}
Let~$M \leq G$ be connected reductive $\Q$-groups. Let~$U\leq M(\AAA_f)$ 
be a subgroup satisfying the following.
\begin{enumerate}
\item \label{Galois bound 1}	
 The image of~$U$ in~$M^{ab}$ is MT in $M^{ab}$. \label{thm:galois bounds H1}
%\item Conditions of \ref{defi:Tate} and \ref{defi:indep} hold.
\item \label{Galois bound 2} The group~$U$ satisfies the uniform integral Tate conjectures.% as in
%Def.~\ref{defi:indep} and~\ref{defi:Tate}.
\label{thm:galois bounds H2}
%\item \label{Galois bound 3} For every~prime $p$, $U_p$ is Zariski connected.\label{thm:galois bounds H3}
%The group~$U$ is a product~$U=\prod_p U_p$.
%\item For every~$p$, the centraliser of~$U_p$ in~$G(\Q_p)$ is~$Z(\Q_p)$ where $Z=Z_G(M)$, the centraliser.
%\item For almost all~$p$ the centraliser of~$U(p)$ in~$G_{\F_p}$ is~$Z_{\F_p}$.% (variante uniforme, pour shafa).
\end{enumerate}

Denote by~$\phi_0:M\to G$ the identity homomorphism and~$W=G\cdot \phi_0$ 
its conjugacy class. Then as~$\phi$ varies in~$W(\A_f)$, 
we have, for any compact open subgroup~$K\leq G(\A_f)$, where~$K_M=K\cap M(\A_f)$,
\begin{equation}\label{CCL:Th51}
[\phi(U):\phi(U)\cap K]\approx [\phi(K_M):\phi(K_M)\cap K]\succcurlyeq H_f(\phi).
\end{equation}
as functions~$W(\A_f)\to \Z_{\geq1}$.
\end{theorem}
We note that, by~\cite[Th.~C1]{RY}, we have
\[
[\phi(U):\phi(U)\cap K]\leq [\phi(K_M):\phi(K_M)\cap K]\preccurlyeq H_f(\phi).
\]
In~\eqref{CCL:Th51} is thus enough to prove
\begin{equation}\label{CCLbis:Th51}
[\phi(U):\phi(U)\cap K]\succcurlyeq H_f(\phi).
\end{equation}
\subsubsection{Reduction to the case~$U=\prod_p U_p$}\label{thm 51 reduction produit}
 Let~$U':=\prod_p \left(U\cap M^{der}(\Z_p)\right)\cdot \left(U\cap Z(M)({\Z_p})\right)$.  As~$U\geq U'$, we have that~$[\phi(U):\phi(U)\cap K]\geq [\phi(U'):\phi(U')\cap K]$ and thus it is enough to prove~\eqref{CCLbis:Th51} with~$U'$ instead of~$U$. 

Let~$W$ and~$W'$ be the image of~$U$ and~$U'$ in~$M^{ab}(\widehat{\Z})$.
The assumption~\eqref{Galois bound 1} of~Th.~\ref{Galois bounds} implies that there exists~$f\in\Z_{\geq1}$ such that~$\prod_p W_p\geq \{u^f| u\in M^{ab}(\widehat{\Z})\}\geq \{w^f| w\in W\}$. Thus~$W$ satisfies~\eqref{hypothese Cor B2}, and by Cor.~\ref{AppB:main cor} we may apply~Cor.~\ref{cor:burnside} to~$U'$ and thus~$U'$ satisfies the assumption~\eqref{Galois bound 2} of Theorem~\ref{Galois bounds}. From Cor~\ref{AppB:main cor} we deduce~$\forall w\in W_p, w^e \in W_p':=W'\cap M^{ab}(\Q_p)$. By Lem.~\ref{lem:burnside} we have~$[M^{ab}(\Z_p):W'_p]\leq [M^{ab}(\Z_p):W_p]\cdot k(n,e)$. Thus~$U'$ satisfies the assumption~\eqref{Galois bound 1} of Theorem~\ref{Galois bounds}.

We may thus substitute~$U$ with~$U'$ in Theorem~\ref{Galois bounds} and thus assume that~$U=\prod_p U_p$.

\subsubsection{Reduction to the case~$U_p=U_p^0$}
Let~$V=\prod_p U_p^0$, where~$U_p^0$ is the neutral component of~$U_p$ for the Zariski topology of~$M(\Q_p)$. By Proposition~\ref{prop:Up0 bounded index} below, there exists~$D\in\Z_{\geq 1}$ such that~$\forall p, [U_p:U_p^0]\leq D$.

It follows that the assumption~\eqref{Galois bound 1} of Theorem~\ref{Galois bounds} is satisfied for~$U=V$.

By remarks~\eqref{rem(7)} and~\eqref{rem(9)} of \S\ref{rem:Tate}, it follows that the assumption~\eqref{Galois bound 2} of Theorem~\ref{Galois bounds} is satisfied for~$U=V$.

As~$V\leq U$, we have~$[\phi(U):\phi(U)\cap K]\geq [\phi(V):\phi(V)\cap K]$  and thus it is enough to prove~\eqref{CCLbis:Th51} with~$V$ instead of~$U$.
%Thus the conclusion of~\ref{Galois bounds} is satisfied for~$U$ if it is satisfied for~$U=V$.

It follows that, in proving Theorem~\ref{Galois bounds}, we may assume
\begin{equation}	\text{For every~prime $p$, $U_p$ is Zariski connected.}\label{thm:galois bounds H3}
\end{equation}

\begin{proposition}\label{prop:Up0 bounded index}
Let~$U$, $M$ and~$G$ be as in Definition~\ref{defi:Tate}.

Assume moreover that the image of~$U$ in~$M^{ab}$ is MT in~$M^{ab}$ as in Th.5.1 (1).

Define~$H_p:=\overline{U_p}^{Zar}\leq M_{\Q_p}$.

Then:
\begin{itemize}
\item we have~$Z(M)_{\Q_p}^0\leq H_p$ for every~$p$,
\item and we have
\[
\sup_{p}\#\pi_0(H_p)<+\infty.
\]
\end{itemize}
\end{proposition}
\begin{proof}
By Definition~\ref{defi:Tate}, the algebraic group~$H_p^0$ is reductive. We write
\[
H_p^0=S_p\cdot T_p
\]
where~$S_p$, the derived subgroup of~${H_p}^0$, is semisimple, and~$T_p$, the centre of~$H_p^0$, is a torus. 

We have~$T_p=Z(H_p^0)\leq Z_{M_{\Q_p}}(H_p^0)=Z_{M_{\Q_p}}(U_p^0)=Z(M_{\Q_p})$, by Definition~\ref{defi:Tate}.

Denote by~$ab_M:M\to M^{ab}$ the abelianisation map. As~$S_p$ is semisimple, we have~$S_p\leq \ker(ab_M)$.
Thus~$ab_M(H_p^0)=ab_M(S_p\cdot T_p)=ab_M(T_p)$.

Since the image of~$U$ in~$M^{ab}$ is MT in~$M^{ab}$, the subgroup~$ab_M(U_p)^0\leq M^{ab}(\Q_p)$
is open in the $p$-adic topology. It follows that~$ab_M(H_p^0)\leq M^{ab}$ is open in the Zariski topology.

Thus~$ {M^{ab}}^0\leq ab_M(H_p^0)=ab_M(T_p^0)$. 

As~$M$ is reductive, the map~$Z(M)\to M^{ab}$ is an isogeny. As~$ab_M(T_p)$ is of finite index in~$M^{ab}$,
the algebraic subgroup~$T_p\leq Z(M_{\Q_p})$ is of finite index. That is,~$Z(M_{\Q_p})^0\leq T_p\leq H_p$.

This proves the first claim.

Denote by~$N:=N_{M_{\Q_p}}(H_p^0)$ the normaliser.
The adjoint action induces an homomorphism
\[
H_p\to N\to Aut(H_p^0).
\]
As~$H_p^0$ is reductive, we have~$N^0=H_p^0\cdot Z_{M_{\Q_p}}(H_p)^0$. The first claim implies~$H_p^0=N^0$.

We thus have an inclusion~$\pi_0(H_p)=H_p/H_p^0\to N/N^0= \pi_0(N)$. Let~$K$ and~$I$ denote the kernel and the image of the map
\[
N/H_p^0\to Out(H_p^0).
\]
A coset~$kH^0_p\in N/H_p^0$ is in~$K$ if and only if there exists~$h\in H_p^0$ such that~$k\cdot h\in Z_{M_{\Q_p}}(H_p^0)$.
But~$Z_{M_{\Q_p}}(H_p^0)=Z(M_{\Q_p})$ and, by the first claim, $H_p^0=Z(M_{\Q_p})^0$.
Thus~$\#K\leq \#\pi_0(Z(M_{\Q_p}))$. For~$p\gg0$, the number~$\#\pi_0(Z(M_{\Q_p}))$ does not depend on~$p$.

As~$S_p$ is semisimple, we have~$\# Out(S_p)<+\infty$. As there are only finitely many conjugacy classes of semisimple groups in~$M_{\mathbb{C}}$, there exists~$C(M)$ such that, for all~$p$, we have~$  \# Out(S_p)<C(M)$.

For~$t\in T_p$, we have~$t\in Z(M_{\Q_p})^0$, by the first claim. Thus, for~$m\in N$, we have~$m\in Z_M(T_p)$.
This implies that the image of~$N$ in~$Out(T_p)=Aut(T_p)$ is trivial. Thus the map~$Out(H_p^0)\to Out(S_p)$ is injective on~$I$.

Thus~$\#I\leq C(M)$.

The second claim follows.
\end{proof}

\subsubsection{Reduction to a local problem} From~\cite[Th.~B.1, Cor.~B.2]{RY} we already have
\[
[\phi(K_M):\phi(K_M)\cap K]\succcurlyeq H_f(\phi),
\]
and because~$U\leq K_M$, we have
\[
[\phi(U):\phi(U)\cap K]\leq [\phi(K_M):\phi(K_M)\cap K].
\] 
By~\cite[Th.~C1]{RY}, we have
\[
[\phi(K_M):\phi(K_M)\cap K]\preccurlyeq H_f(\phi).
\]

Thus it will be enough to prove
\begin{equation}\label{to prove}
[\phi(U):\phi(U)\cap K]\succcurlyeq [\phi(M(\widehat{\Z})):\phi(M(\widehat{\Z}))\cap K].
\end{equation}

Since they are commensurable groups, we may replace~$K$ by~$G(\widehat{\Z})$. 

In view of Remark~\ref{rem:Tate}~\eqref{passage aux Up}, we may replace~$U$ by~$\prod_pU_p$, 
which is smaller.

Then the required inequality~\eqref{to prove}
can be rewritten in the product form
\[
\prod_p[\phi(U_p):\phi(U_p)\cap G(\Z_p)]\succcurlyeq \prod_p[\phi(M(\Z_p)):\phi(M(\Z_p))\cap G(\Z_p)]
\]
%We note that we may replace~$U$ by the smaller~$U'=\prod U_p$, and assume~$U=U'$.
and thus the problem can be studied prime by prime. 

More precisely, it will be enough to prove
\begin{itemize}
\item that there exists~$c\in\R_{>0}$ such that, for almost all primes
\begin{multline}\label{precise-bound}
\forall \phi\in W(\Q_p), [\phi(U_p):\phi(U_p)\cap G(\Z_p)]\geq\\ [\phi(M(\Z_p)):\phi(M(\Z_p))\cap G(\Z_p)]^{c}.
\end{multline}
\item and, for the finitely remaining primes, that we have the polynomial domination, as functions~$W(\Q_p)\to \R_{\geq 0}$, 
\[
[\phi(U_p):\phi(U_p)\cap G(\Z_p)]\succcurlyeq [\phi(M(\Z_p)):\phi(M(\Z_p))\cap G(\Z_p)].
\]
Namely, there exist~$a(p),c(p)\in\R_{>0}$ such that 
\begin{multline}\label{imprecise-bound}
\forall \phi\in W(\Q_p),  [\phi(U_p):\phi(U_p)\cap G(\Z_p)]\geq\\  a(p)\cdot [\phi(M(\Z_p)):\phi(M(\Z_p))\cap G(\Z_p)]^{c(p)}.
\end{multline}
\end{itemize}

It will be sufficient, instead of~\eqref{precise-bound}, to prove:
there exist~$a,c\in\R_{>0}$ such that, for almost all primes
\begin{multline}\label{precise with a}
\forall \phi\in W(\Q_p), [\phi(U_p):\phi(U_p)\cap G(\Z_p)]\geq\\ a\cdot [\phi(M(\Z_p)):\phi(M(\Z_p))\cap G(\Z_p)]^{c}.
\end{multline}
Indeed we have, for every~$a\in\R_{>0}$, 
\[
\prod_{\left\{p\middle|[\phi(M(\Z_p)):\phi(M(\Z_p))\cap G(\Z_p)]\neq 1\right\}} a = \left(\prod_{p}[\phi(M(\Z_p)):\phi(M(\Z_p))\cap G(\Z_p)]\right)^{o(1)}.
\]
(cf. \cite[Proof of Cor. B2]{RY}.)

This is the following statement.
\begin{theorem}[Local Galois bounds]\label{thm:local Galois}
In the setting of Th.~\ref{Galois bounds}, there exist~$a,c\in\RR_{>0}$, and for each~$p$, there exists~$b(p)\in\RR_{>0}$ such that
\[
[\phi(U_p):\phi(U_p)\cap K_p]\geq  b(p)\cdot [\phi(M(\Z_p)):\phi(M(\Z_p))\cap K_p]^{c}
\]
and such that~$b(p)\geq a$ for almost all~$p$.
\end{theorem}

We prove~\eqref{imprecise-bound} in~\ref{every prime}. It is deduced from the functoriality of heights.

We prove, for almost all primes,~\eqref{precise-bound} in~\ref{almost all primes}. It requires new tools developed
in this article.

For reference, we rephrase ``The image of~$U$ in~$M^{ab}$ is MT in $M^{ab}$'' as follows.
We denote~$ab_M:M\to M^{ab}:=M/M^{der}$ the abelianisation map. Denote by~$M^{ab}(\Z_p)$
the maximal compact subgroup of the torus~$M^{ab}$. Then there
exists~$C_{MT}\in \Z_{\geq 1}$ such that
\begin{equation}\label{defi CMT}
\forall p, [M^{ab}(\Z_p):ab_M(U_p)]\leq C_{MT}.
\end{equation}
Because~$\exp(p\mathfrak{m}^{ab}_{\Z_p})$ is a~pro-$p$-group, its action
on~$M^{ab}(\Z_p)/ab_M(U_p)$ is trivial when~$p>C_{MT}$: we have
\begin{equation}\label{defi CMT 2}
\forall p>C_{MT},\exp(p\mathfrak{m}^{ab}_{\Z_p}) \leq ab_M(U_p).
\end{equation}

\subsection{For every prime}\label{every prime}
We fix a prime~$p$. Let~$f_1,\ldots,f_{k}$
be a basis of the~$\Q_p$-Lie algebra~$\mathfrak{u_p}$ of~$U_p$. Replacing each~$f_i$ by a sufficiently 
small scalar multiple, we may assume that each~$u_i=\exp(f_i)$ converges and belongs to~$U_p$.
By~\eqref{Galois bound 2} of~Th.~\ref{Galois bounds} and~\eqref{defi:tate eq1} of~Def.~\ref{defi:Tate},
we have\footnote{We observe the following facts. For a compact subgroup~$U\leq GL(n,\Q_p)$, any open subgroup~$V\leq U$ is of finite index. If~$U$ is connected for the Zariski topology, then~$V$ is Zariski dense in~$U$. By construction, we have~$\mathfrak{u}\subseteq  \Q_p[V]$ in~$\mathfrak{gl}(n,\Q_p)$, and, with~$V\leq U$ generated by~$U\cap \exp\left(\mathfrak{u}\cap p^2\mathfrak{gl(n,\Z_p)}\right)$, we have~$V\subseteq \Q_p[\mathfrak{u}]$.  From these facts, it follows that
\[Z_{GL(n)_{\Q_p}}(\mathfrak{u})=Z_{GL(n)_{\Q_p}}(\Q_p[\mathfrak{u}])=Z_{GL(n)_{\Q_p}}(\Q_p[V])=Z_{GL(n)_{\Q_p}}(V)=Z_{GL(n)_{\Q_p}}(U^0).\]
}
\[
Z_{G_{\Q_p}}({U_p})=Z_{G_{\Q_p}}({U_p}^0)=Z_{G_{\Q_p}}(\mathfrak{u}_p)=Z_{G_{\Q_p}}(\{f_1,\ldots,f_k\}).
\]

Let~$\mathfrak{g}$ be the Lie $\Q_p$-algebra of~$G_{\Q_p}$. We define, for some faithful linear representaton~$\rho:G\to GL(V)$ defined over~$\Q$, 
\[
v=(f_1,\ldots,f_k)\in \mathfrak{g}^k\stackrel{d\rho}{\hookrightarrow} E:=\End(V_{\Q_p})^k.
\]
For the induced action of~$G$ on~$E$ we have
\[
Z_{G_{\Q_p}}(U_p)=Stab_{G_{\Q_p}}(v).
\]
By our assumption,
\[
Z_{G_{\Q_p}}(U_p)=Stab_{G_{\Q_p}}(\phi_0)=Z_G(M)_{\Q_p}.
\]
As a consequence, we have a well defined isomorphism~$W\to G\cdot v$ defined over~$\Q_p$ of homogeneous varieties,
given by
\[
g\cdot Z_G(M)_{\Q_p}\mapsto g\cdot v
\]
From~\eqref{defi:Tate1} of Def.~\ref{defi:Tate}, the Zariski closure~$\ol{U_p}^{Zar}$ is reductive.
We may thus apply~\cite[Th. 3.6]{Ri-Conj-Duke}, and deduce that the induced map
\[
\iota:W\to E
\]
is a closed affine embedding.

We use the standard norm on~$E\simeq {\Q_p}^{\dim(V)^2\cdot k}$. We denote by~$H_\iota$ the local Weil height associated to this embedding, which is given by
\begin{multline}
H_\iota:\phi\mapsto H_p(g\cdot v):=\max\{1;\norm{g\cdot v}\}\\=\max\{1;\norm{g\cdot f_1};\ldots;\norm{g\cdot f_k}\}.
\end{multline}
By  functoriality properties of height functions, the function~$H_\iota$ and~$\phi\mapsto H_p(\phi)$ are polynomially equivalent.

Namely, there are~$a(p)$ and~$c(p)$ such that
\[
H_\iota\geq a(p)\cdot  H_p(\phi)^{c(p)}.
\]

We denote~$U'\leq U_p$ the~$p$-adic Lie subgroup generated by
\[\{\exp(f_1);\ldots;\exp(f_k)\}.\]
We have
\[
[\phi(U_p):\phi(U_p)\cap G(\Z_p)]\geq [\phi(U'):\phi(U')\cap G(\Z_p)].
\]
Using~\cite[Th.~A3]{RY}, we also have, with~$U_i:=\exp(f_i)^\Z\leq U'$, and denoting by~$\norm{~}:\End(V_{\Q_p})\simeq \Q_p^{\dim(V)^2}\to\R_{\geq0}$ the standard norm,
\begin{multline}
[\phi(U'): \phi(U')\cap K_p]\geq
\max_{1\leq i\leq k}[\phi(U_i): \phi(U_i)\cap K_p]
\\
\geq 
\max_{1\leq i\leq k} \max\{1;\norm{d\phi(f_i)}\}
= H_\iota(\phi)/\dim(V).
\end{multline}
By~\cite[App. Th.~C2 (102)]{RY}, we have~$[\phi(M(\Z_p)): \phi(M(\Z_p))\cap K_p]\preccurlyeq H_\iota(\phi)$.
We deduce~\eqref{imprecise-bound}.
\subsubsection{Remark}

%\begin{rem}
Note that the above bound already implies the Andr\'e-Pink-Zannier conjecture for
$S$-Hecke orbits. This is more general than the result of Orr (\cite{OrrPhD}) for Shimura varieties of abelian type
and less precise than \cite{RY2} which proves a strong topological form under a weaker hypothesis.
The method of Orr relies on Masser-W\"ustholz bounds, and~\cite{RY2} relies ultimately on $S$-adic Ratner theorems through
the work of~\cite{RZ}.
%\end{rem}

\subsection{For almost all primes}\label{almost all primes}
\subsubsection{Construction of tuples}
We denote by
\[
Y\mapsto \ol{Y}:\mathfrak{m}_{\Z_p}\to\mathfrak{m}_{\F_p}\text{ and }\pi_p: U_p\to M_{\F_p}(\F_p)
\]
the reduction modulo~$p$ maps, and define
\[
U(p):=\pi_p(U_p)
\]
the image of~$U_p$ in~$M_{\F_p}(\F_p)$.
We denote the subgroup of~$U(p)$ generated by its unipotent elements
by
\begin{equation}\label{defi daggers}
U(p)^\dagger
\end{equation}
and its inverse image in~$U_p$ by
\[
U_p^\dagger:=\stackrel{-1}{\pi_p}(U_p).
\]
Let~$\mathfrak{m}^{der}$ and $\mathfrak{z(m)}$ be the derived and centre subalgebras of~$\mathfrak{m}$, and define~$\mathfrak{m}^{der}_{\Z_p}=\mathfrak{m}^{der}\cap\mathfrak{m}_{\Z_p}$ and~$\mathfrak{z(m)}_{\Z_p}=\mathfrak{z(m)}\cap\mathfrak{m}_{\Z_p}$.  We define
\begin{equation}\label{defi nu}
\nu=\mathfrak{m}^{der}_{\Z_p}+p\cdot \mathfrak{z(m)}_{\Z_p}.
\end{equation}

\begin{proposition}\label{prop X Y}
 We consider the setting of Theorems~\ref{thm:local Galois} and~\ref{Galois bounds}. For almost all~$p$, there exists
\[
X_1,\ldots,X_k,Y_1,\ldots,Y_l\in\mathfrak{m}_{\ZZ_p}
\]
satisfying the following
\begin{enumerate}
\item \label{XY1} The exponentials~$\exp(X_1),\ldots,\exp(X_k)$ converge and topologically generate~$U_p^\dagger$.
\item \label{XY2} We have
\[
Y_1,\ldots,Y_l\in \mathfrak{u}:=\Z_p\cdot X_1+\ldots+\Z_p\cdot X_k.
\]
\item \label{XY3} We have
\[
\frac{1}{p}\cdot Y_1\cdot \ZZ_p+\ldots+\frac{1}{p}\cdot  Y_l\cdot \ZZ_p = \mathfrak{z(m)}_{\Z_p}\pmod{p\cdot\mathfrak{m}_{\Z_p}}.
\]
\item \label{XY4} We have
\begin{equation}\label{PropNori:Centalisateurs}
Z_{G_{\FF_p}}(M_{\FF_p})=
Z_{G_{\FF_p}}\left(\left\{\ol{X_1};\ldots;\ol{X_k};\ol{\frac{1}{p}Y_1};\ldots;\ol{\frac{1}{p}Y_l}\right\}\right).
\end{equation}
\end{enumerate}
\end{proposition}
\begin{proof}
Let~$u_1,\ldots,u_i$ be unipotent generators of~$U(p)^\dagger$. Because~$U(p)^\dagger\leq U(p)=\pi_p(U_p)$, we may write
\[
u_1=\pi_p(x_1),\ldots,u_i=\pi_p(x_i)
\]
with~$x_1,\ldots,x_i\in U_p$. By definition of~$U_p^\dagger$, we have~$x_1,\ldots,x_i\in U^\dagger_p$.

The compact group~$\ker(\pi_p)\leq U_p\leq M(\Z_p)$ is topologically of finite type. We choose~$x_{i+1},\ldots,x_k$
a topologically generating family of~$\ker(\pi_p)$. By construction
\begin{equation}\label{defi xs}
x_1,\ldots,x_k\text{ topologically generate~$U_p^\dagger$.}
\end{equation}
Moreover, the~$\pi_p(x_i)$ are unipotent. By~\cite[Prop.~A.1]{RY}, the series~$X_1=\log(x_1),\ldots,X_k=\log(x_k)$
converge, and, for~$p>d$, we have~$X_1,\ldots,X_k\in \mathfrak{m}_{\Z_p}$. Using~\cite[\S6.1.5]{Robert} (cf.~\cite[Proof of Lem. A.2]{RY}), we can deduce,  for~$p>d+1$, that
%By~\cite[Lem. A.2]{RY} we have
\begin{equation}\label{defi X}
x_1=\exp(X_1),\ldots,x_k=\exp(X_k).
\end{equation}
We define
\begin{equation}\label{defi u}
\mathfrak{u}=\Z_p\cdot X_1+\ldots+\Z_p\cdot X_k\qquad
u=(X_1,\ldots,X_k).
\end{equation}
For~$X\in \mathfrak{m}_{\Z_p}\smallsetminus\nu$, its reduction in~$\mathfrak{m}_{\F_p}$
is not nilpotent. Thus~$X^n/n!$ does not converge to~$0$, and the series~$\exp(X)$ does not converge.
Consequently we have
\[
X_1,\ldots,X_k\in\nu\text{ and thus }\mathfrak{u}\leq \nu.
\]
We define
\[
\pi_\nu:\nu\to \ol{\nu}:=\nu\tens\F_p,
\]
and denote the image of~$\mathfrak{u}$ by
\[
\ol{\mathfrak{u}}\leq \ol{\nu}.
\]
From~\eqref{defi nu}, we have
\begin{equation}\label{defi nubar}
\ol{\nu}=\mathfrak{m}_{\Z_p}^{der}~(\bmod~p\nu)+p\cdot \mathfrak{z(m)}_{\Z_p}~(\bmod~{p\nu}).
\end{equation}
We notice that
\begin{equation}\label{same rep}
\mathfrak{m}_{\F_p}=\mathfrak{m}^{der}_{\F_p}+\mathfrak{z(m)}_{\F_p}\text{ and }\ol{\nu}\simeq \mathfrak{m}^{der}_{\F_p}+\frac{p\cdot \mathfrak{z(m)}}{p^2\cdot\mathfrak{z(m)}}
\end{equation}
are isomorphic $\F_p$-linear representation of~$M(\Z_p)$ and~$M(\F_p)$, and thus
as representations of~$U_p$ and~$U(p)$ as well.

We consider
\[
ab_{\mathfrak{m}}:\mathfrak{m}_{\Z_p}\to \mathfrak{m}_{\Z_p}^{ab}:=\mathfrak{m}_{\Z_p}/\mathfrak{m}_{\Z_p}^{der}.
\]
Let us prove the claim
\begin{equation}\label{reciprocite p p2}
p\cdot\mathfrak{m}_{\Z_p}^{ab}\leq ab_{\mathfrak{m}}(\mathfrak{u}).
\end{equation}
\begin{proof}%TODO, L.~\ref{lem:4.1} plus reciprocité.
Let~$Z\in p\cdot\mathfrak{m}_{\Z_p}^{ab}$. Let~$z=\exp(Z)\in M^{ab}(\Z_p)$.

From~\eqref{defi CMT 2}, when~$p>C_{MT}$, there exists~$y\in U_p$ with~$\ab_M(y)=z$.

Assume~$p\gg0$, so that the algebraic tori~$Z(M)_{\Z_p}$ and~$M^{ab}$ have good reduction%\footnote{This is autmatically the case for~$Z(M)$, and~$M^{ab}?$ as soon as~$M$ is hyperspecial at~$p$.}
,
and assume furthermore~$p>\#\ker(Z(M)\to M^{ab})$ so that the differential of the isogeny~$Z(M)\to M^{ab}$
induces a $\Z_p$-isomorphism~$\mathfrak{z(m)}_{\Z_p}\to\mathfrak{m}_{\Z_p}^{ab}$. Thus, there exists~$Z'\in p\mathfrak{z}(\mathfrak{m})_{\Z_p}$ with~$ab_{\mathfrak{m}}(Z')=Z$.

Let~$z'=\exp(Z')$. As~$Z'\in\mathfrak{z}(\mathfrak{m})$, we have~$z'\in Z(M)(\Q_p)$. As~$Z'\in p\cdot \mathfrak{z}(\mathfrak{m})_{\Z_p}$, we have, for~$p>2$,~$z'\in Z(M)(\Z_p)$.
Moreover~$\lim_{n\to \infty}{z'}^{p^n}=\lim_{n\to \infty}\exp(p^nZ')=1$. Thus~$\ol{z'}\in Z(M)(\F_p)$ is unipotent.

We have~$\ol{z'}\in Z(M)(\F_p)^\dagger$ and, because~$Z(M)(\F_p)$  has good reduction, we have~$Z(M)(\F_p)^\dagger=\{1\}$. We also have~$y\in M^{der}(\Z_p)\cdot z'$. 
Thus
\[
\ol{y}\in U(p)\cap M^{der}(\F_p).
\]
Let~$\gamma=[\hat{U}:\hat{U}^\dagger]\leq c(\dim(G))$ from Lem.~\ref{lem:4.1}. 
Then~$\ad_{M}(\ol{y})^{\gamma}\in \ad_M(U(p))^\dagger=\ad_M(U(p)^\dagger)$.
Let~$U'$ be the inverse image of~$\ad_M(U(p))^\dagger$ by~$\ad_{M^{der}}:M^{der}(\F_p)\to M^{ad}(\F_p)$.
Since~$\ad_M(U(p))^\dagger=\ad_M(U(p)^\dagger)$ and~$\ker(\ad_{M^{der}})=Z(M^{der})(\F_p)$, we have~$U'=U(p)^\dagger\cdot Z(M^{der})(\F_p)$. There are~$u\in U(p)^\dagger$ and~$z\in   Z(M^{der})(\F_p)$
such that
\[
\ol{y}^\gamma=u\cdot z.
\]
We use that~$f:=\sup_{p}Z(M^{der})(\F_p)<+\infty$. Since~$z$ commutes with~$u$, we have
\[
(\ol{y}^\gamma)^{f!}=u^{f!}\cdot z^{f!}=u^{f!}\in U(p)^\dagger.
\]
%Then~$\ol{y}^{\gamma}\in U(p)^\dagger$,
Thus~$y^{\gamma\cdot f!}\in U_p^\dagger$. Assume~$p>c(\dim(G))\geq \gamma':=\gamma\cdot f!$, so that~$\gamma'\in\Z_p^\times$.  Because~$Z$ is arbitrary, we have
\[
ab_M(U_p^\dagger)\geq\exp(p\cdot\mathfrak{m}_{\Z_p}^{ab})^{\gamma'}=\exp(\gamma'\cdot p\cdot\mathfrak{m}_{\Z_p}^{ab})
=
\exp(p\cdot\mathfrak{m}_{\Z_p}^{ab}).
\]
Conversely~$ab_M(U_p^\dagger)\leq \exp(p\cdot\mathfrak{m}_{\Z_p}^{ab})=\ker M^{ab}(\Z_p)\to M^{ab}(\F_p)$
because~$ab_{M_{\F_p}}(U(p)^\dagger)\leq ab_{M_{\F_p}}(M^{der}(\F_p)^\dagger)=\{1\}$.

The group~$U_p^\dagger$ is topologically generated by~$\exp(X_1),\ldots,\exp(X_k)$, and thus
$ab_M(U_p^\dagger)$ is topologically generated by~$\exp(Z_1),\ldots,\exp(Z_k)$ with~$Z_i=ab_{\mathfrak{m}}(X_i)$.

Thus the logarithms
\[
Z_i=\log(z_i)=ab_{\mathfrak{m}}(X_i)
\]
topologically generate~$\log(\exp(p\cdot\mathfrak{m}^{ab}))=
p\cdot\mathfrak{m}^{ab}$. The conclusion follows.
\end{proof}

We let
\begin{equation}\label{defi Zs}
Z_1,\ldots,Z_l\text{ be a basis of }
\frac{p\mathfrak{m}^{ab}}{p^2\mathfrak{m}^{ab}}\simeq \nu/\mathfrak{m}^{der}_{\Z_p}\simeq \ol{\nu}/\mathfrak{m}^{der}_{\F_p}
\end{equation}
Pick an arbitrary~$Z\in\{Z_1;\ldots;Z_l\}$, and define
\[
A=\{\ol{Y}\in\ol{\mathfrak{u}}| \ol{Y}\equiv Z\pmod{\mathfrak{m}^{der}_{\F_p}}\}.
\]
From~\eqref{reciprocite p p2}, this~$A$ is non empty. It is thus an affine subspace of~$\ol{\nu}$,
and, for any~$\ol{Y}_0\in A$, we have
\[
A=\ol{Y}_0+V,
\]
{ where }\(V=\ol{\mathfrak{u}}\cap \mathfrak{m}^{der}_{\F_p}\)
is the ``direction'' of~$A$. The $\F_p$-linear vector subspace~$V\leq \ol{\nu}$ is invariant under~$U(p)$,
and because the action of~$U(p)$ is semisimple on~$\mathfrak{m}_{\F_p}$, and thus,
by~\eqref{same rep}, on~$\ol{\nu}$, there exists a supplementary $U(p)$-invariant $\F_p$-linear subspace
\[
W\leq \ol{\nu}.
\]
The following intersection is an affine space of dimension~$0$, hence it is a singleton
\[
A\cap W=\{\ol{Y}\}.
\]
It is also invariant under~$U(p)$. Thus the line
\[
\F_p\cdot \ol{Y}
\]
is fixed by~$U(p)$. But the centraliser of~$U(p)$ and~$M_{\F_p}$ in~$M_{\F_p}$ are the same.
For~$p\gg0$, these centralisers are smooth as group schemes (cf. Lem.~\ref{conj orbit lemma}), and thus have the same Lie algebra
\[
\mathfrak{z}_{\mathfrak{m}_{\F_p}}(U(p))
=
\mathfrak{z(m)}_{\F_p}.
\]
Thus
\begin{equation}\label{Ybar in Z}
\ol{Y}\in{p\cdot \mathfrak{z}(M)}~(\bmod~p\nu).
\end{equation}
We finally choose a representative~$Y'\in \mathfrak{u}$ of~$\ol{Y}\in\ol{\mathfrak{u}}$.

Thus~$Y'\in p\cdot \mathfrak{z(m)}+p\nu=p\mathfrak{m}_{\Z_p}$ and~$Y'\in\mathfrak{u}$. We define
\[
\wt{\mathfrak{m}}:=\frac{p\mathfrak{m}_{\Z_p}}{p^2\mathfrak{m}_{\Z_p}}=(p\mathfrak{m}_{\Z_p})\tens \F_p,
\]
and denote the image of~$Y'\in \mathfrak{u}\cap p\mathfrak{m}_{\Z_p}\leq p\mathfrak{m}_{\Z_p}$ by
\[
\wt{Y}\in\wt{\mathfrak{u}}:=(\mathfrak{u}\cap p\mathfrak{m}_{\Z_p})\tens \F_p\leq \wt{\mathfrak{m}}.
\]
Again~$\wt{\mathfrak{m}}\simeq\mathfrak{m}_{\F_p}$ as a representation. We define~$\wt{\mathfrak{m}^{der}}:=(\mathfrak{m}^{der}\cap p\mathfrak{m}_{\Z_p})/p^2\mathfrak{m}_{\Z_p}$, and
\[
A'=\{\wt{Y}'\in\wt{\mathfrak{u}}| \wt{Y}'\equiv \wt{Y}\pmod{\wt{\mathfrak{m}^{der}}}\}.
\]
and similarly, there exists~$\wt{Y}'\in A'$ which is fixed by~$U(p)$ and thus is in~$\wt{\mathfrak{z(m)}}$.

We choose a lift~$Y$ of~$\wt{Y}'$ in~$\mathfrak{u}$.

Repeating the process for each~$Z\in\{Z_1;\ldots;Z_l\}$ we define
\begin{equation}\label{Ys in u}
Y_1,\ldots,Y_l\in\mathfrak{u}.
\end{equation}

The assertion~\ref{XY1} follows from~\eqref{defi xs}, \eqref{defi X}, \eqref{defi u}.

The assertion~\ref{XY2} follows from~\eqref{defi u}, \eqref{Ys in u}.

The assertion~\ref{XY3} follows from~\eqref{Ybar in Z}, \eqref{defi Zs}, \eqref{defi nu}, \eqref{defi nubar}.
(We observe that~$Y_1,\ldots,Y_l$ generate~$\mathfrak{z}(\mathfrak{m})_{\Z_p}$ modulo~$p\cdot\mathfrak{m}_{\Z_p}$ since their images in~$p\mathfrak{m}^{ab}_{\Z_p}$ are congruent to~$Z_1,\ldots,Z_l$ modulo~$p^2\mathfrak{m}^{ab}_{\Z_p}$. Indeed~$Y\equiv \wt{Y}\pmod{\wt{\mathfrak{m}^{der}}}$ gives
\[
Y\in Y'+p^2\mathfrak{m}_{\Z_p}+p\mathfrak{m}^{der}_{\Z_p}=Y'+p(p\cdot\mathfrak{z}(\mathfrak{m})_{\Z_p}+\mathfrak{m}_{\Z_p}^{der})=Y'+p\nu.
\]
Thus~$Y\equiv Y'\pmod{p\nu}$, and therefore maps to~$Z\in p\mathfrak{m}^{ab}_{\Z_p}$ modulo~$p^2\mathfrak{m}^{ab}_{\Z_p}$.)

We will now prove the assertion~\ref{XY4}. We define
\[
Z:=Z_{G_{\FF_p}}\left(\left\{\overline{X_1};\ldots;\ol{X_k},\ol{\frac{1}{p}Y_1};\ldots;\ol{\frac{1}{p}Y_l}\right\}\right).
\]
and
\[
U':=(U(p)\cap Z(M)^0_{\F_p}(\F_p))
\text{ and }
U'':=U(p)^\dagger\cdot U'.
\]

We first note that~$\pi_p(\exp(X_1)),\ldots,\pi_p(\exp(X_k))$ generates the group~$U(p)^\dagger$
and that~$\ol{Y_1/p},\ldots,\ol{Y_l/p}$ generates the Lie algebra~$\mathfrak{z(m)}_{\F_p}$.

Thus
\begin{equation}%\label{PropNori:Centalisateurs}
Z=Z_{G_{\FF_p}}(U(p)^\dagger)\cap Z_{G_{\FF_p}}(\mathfrak{z(m)}_{\F_p})
\end{equation}
We have\footnote{We use that~$Z(M_{\F_p})^0$ is connected and, for~$p\gg 0$ smooth as a group scheme.}
\[
Z_{G_{\FF_p}}(\mathfrak{z(m)}_{\F_p})=Z_{G_{\FF_p}}(Z(M)^0).
\]

Applying Lem.~\ref{lem:4.1} with~$\delta=C_{MT}$ from~\eqref{defi CMT}, we have,%with~$U':=(U(p)\cap Z(M)^0_{\F_p}(\F_p))$,
\[
[U(p): U(p)^\dagger\cdot U']\leq D:=C_{MT}\cdot \gamma(\dim(G)).
\]
With~$p>M(D)$, with~$M(D)$ as in Def.~\ref{defi:Tate}, we have
\[
Z_{G_{\F_p}}(M_{\F_p})=
Z_{G_{\F_p}}(U(p)^\dagger\cdot U')
=
Z_{G_{\FF_p}}(U(p)^\dagger)\cap Z_{G_{\FF_p}}(U').
\]
From Cor.~\ref{coro big dans Z}, we have~$[Z(M)^0_{\F_p}(\F_p):U']\leq \gamma(n)\cdot C_{MT}$ for some~$n$.
We may thus apply Lemma~\ref{Lemma bounded and centraliser}, with~$\lambda=\gamma(n)\cdot C_{MT}$ and deduce
\[
\forall p\gg 0, Z_{G_{\F_p}}(U')=Z_{G_{\F_p}}(Z(M)^0_{\F_p}).
\]
From~$U'':=U(p)^\dagger\cdot U'\leq U(p)^\dagger\cdot Z(M)^0_{\F_p}\leq M_{\F_p}$
we get
\[
Z_{G_{\F_p}}(M_{\F_p})
=
Z_{G_{\F_p}}(U'')\leq Z_{G_{\F_p}}(U(p)^\dagger)\cap Z_{G_{\F_p}}(Z(M)^0_{\F_p})
\leq 
Z_{G_{\F_p}}(M_{\F_p})
\]
Finally
\[
Z=Z_{G_{\F_p}}(U(p)^\dagger)\cap Z_{G_{\F_p}}(Z(M)^0_{\F_p})=
Z_{G_{\F_p}}(M_{\F_p}).\qedhere
\]
\end{proof}

%\subsubsection{OLD}

%\subsubsection{First step: construction of a tuple: the~$X_i$'s}\label{for every prime}
%We identify~$G$ with its image by a faithful representation~$G\to GL(d)_{\QQ_p}$.
%If~$G$ is adjoint, we can take the adjoint representation.
%
%
%We want to apply Nori theory~\cite{N} to the group~$U(p)\leq G(\F_p)\leq GL(d,\FF_p)$.
%%We use the adjoint representation~$\rho:G\to GL(\mathfrak{g})$ 
%%to view~$U(p)$ as a subgroup of~$GL(\dim(G),\F_p)$. 
%We assume~$p> d+1$
%so that the results of Nori apply, and that~$p$ is big enough so that the~$\ZZ_p$ model
%of~$G$, $M$, and~$Z_G(M)$ induced by~$GL(d)$ are hyperspecial.
%
%We denote~$U(p)^\dagger$ the subgroup generated by unipotent elements,
%and~$\mathfrak{s}$ the Lie algebra generated by their logarithms, and~$S$ 
%the connected semisimple group with Lie algebra~$\mathfrak{s}$. This~$S$
%is the Nori group of~$U(p)$ in the sense of Remark~\ref{rem:Tate} (\ref{rem2}).
%
%We denote~${U_p}^\dagger$ the inverse image of~$U(p)^\dagger$. We claim
%that~${U_p}^\dagger$ is generated by finitely many elements~$u_1,\ldots,u_k$
%whose image in~$U(p)$ is unipotent (which is stronger than just belonging to~$U(p)^\dagger$).
%The proof of this claim is left to the reader. We then have~$u_i=\exp(X_i)$,
%where~$X_i$ is in~$\mathfrak{gl}(d,\ZZ_p)$ and has nilpotent image in~$\mathfrak{gl}(d,\FF_p)$.
%See~\cite[App. A]{RY}, for~$p$ is big enough.
%

\subsubsection{Conjugacy classes of tuples}
The following will be used to check, for almost all primes, one of the hypotheses of Th.~\ref{thm:compare reductive}.
\begin{lemma}\label{conj orbit lemma}
Let~$p$ be a prime, let~$G\leq GL(n)_{\F_p}$ be a reductive algebraic subgroup,
and consider~$v_1,\ldots,v_k\in G(\F_p)$.  Denote by~$U$ the group generated by~$\{v_1;\ldots;v_k\}$,
and define~$v=(v_1,\ldots,v_k)$.

Assume that 
\begin{equation}\label{U ss g}
\text{the action of~$U$ on~$\mathfrak{g}_{\F_p}$ is semisimple.}
\end{equation}

If~$p>2\cdot \dim(G)$ then the simultaneous conjugacy class~$G\cdot v$ is Zariski closed in~$G^k$.

If~$p>c_3(\dim(G))$ then the centraliser of~$v$ in~$G$, as a group scheme over~$\F_p$, is smooth.
\end{lemma}
The quantity~$c_3$ is from~\cite[\S4. Th.~E]{N}.
\begin{proof}
From~\cite[\S5.1]{SCR}, we have~$h(G)\leq \dim(G)$, where~$h(G)$ is defined in loc. cit. From~\cite[Cor. 5.5]{SCR}, if~$p>2h(G)-2$, the assumption~\eqref{U ss g} implies that~$U$ is~$G$-cr, or ``strongly reductive'' in~$G$ in the sense of Richardson. The first assertion follows from~\cite[Th.~3.7]{SCR} (cf. \cite[\S16]{Ri-Conj-Duke}).

Thanks to~\eqref{U ss g} and the condition~$p>c_3(\dim(G))$ we may apply~\cite[\S4. Th.~E]{N} (cf. also~\cite[137. p.\,40]{S4}). Thus the hypothesis
of~\cite[II, \S5.2, 2.8, p.\,240]{DG} is satisfied and we conclude. (cf. also~\cite{BMR10} and~\cite{H13} on the subject, beyond the semi-simplicity assumption.)
\end{proof}
\setcounter{secnumdepth}{3}

\subsubsection{Consequences for heights bounds}We denote~$\norm{~}:\mathfrak{g}_{\Q_p}\to\R_{\geq0}$ the~$p$-adic norm~$\norm{X}=\min(\{0\}\cup\{p^k\in p^\Z|p^k\cdot X\in\mathfrak{g}_{\Z_p}\})$  associated to the~$\Z_p$-structure~$\mathfrak{g}_{\Z_p}$.
We denote~$\norm{\Sigma}=\max\{\norm{s}:s\in \Sigma\}$ for a bounded subset~$\Sigma\subseteq \mathfrak{g}_{\Q_p}$.
We recall that~$H_f(\phi)=\prod_p H_p(\phi)$ with~$H_p(\phi)$ given by
\[
H_p(\phi)=\max\{1;\norm{\phi(\mathfrak{m}_{\ZZ_p})}\}.
\]
%We also define
%\[
%H'(\phi)=\max\{1;\norm{\phi(\nu)}\},\qquad H''(\phi)=\max\{1;\norm{\phi(\mathfrak{m}^{der}_{\Z_p})}\}.
%\]
%We then have~$p\cdot \mathfrak{m}_{\Z_p}\leq \nu\leq \mathfrak{m}_{\Z_p}$ and thus
%\[
%\frac{1}{p}\cdot H'(\phi)\leq H_p(\phi)\leq H'(\phi).
%\]
%We will consider three cases whether~$H''(\phi)=1,p$ or~$H''(\phi)\geq p^2$.

%Using the tools of~\cite{RZ} we deduce the following.
\begin{proposition}\label{Prop5.5}
 Define~
\[
v=(X_1,\ldots,X_k,Y_1,\ldots,Y_l)\text{ and }
v'=(X_1,\ldots,X_k,\frac{1}{p}Y_1,\ldots,\frac{1}{p}Y_l)
\]
and
\[
H_v(\phi)= \max\{1;\norm{g\cdot v}\}=H_p(g\cdot v)\text{ and }H_{v'}(\phi)=H_p(g\cdot v').
\]

Then, for~$p\gg0$, we have
\begin{equation}\label{galois exp bound}
[\phi(U):\phi(U)\cap G(\ZZ_p)]\geq H_v(\phi)\geq H_{v'}(\phi)/p
\end{equation}
and
\begin{equation}\label{eq}
H_{v'}(\phi)\geq H_p(\phi)^{c(\rho)}.
\end{equation}
\end{proposition}
\begin{proof} Let~$U_i:=\exp(X_i\cdot \Z_p)$. We have
\[
[\phi(U):\phi(U)\cap G(\ZZ_p)]\geq \max_{1\leq i\leq k}[\phi(U_i):\phi(U_i)\cap G(\ZZ_p)]
\]
and
\[
H_v(\phi)= \max\{1;\norm{g\cdot v}\}=\max_{1\leq i\leq k}\max\{1;\norm{g\cdot X_i}\}.
\]
The inequality~$[\phi(U_i):\phi(U_i)\cap G(\ZZ_p)]\geq \max\{1;\norm{g\cdot X_i}\}$
follows from the Lemma of the exponential (\cite[Th. A.3]{RY}). We deduce
\[
[\phi(U):\phi(U)\cap G(\ZZ_p)]\geq H_v(\phi)= \max\{1;\norm{g\cdot v}\}=H_p(g\cdot v)
\]
The inequality~$H_v(\phi)\geq H_{v'}(\phi)/p$ follows from the definitions.

We prove~\eqref{eq}. Let~$m_1,\ldots,m_d$ be a generating set for~$\mathfrak{m}_{\ZZ_p}$ and define~$w=(m_1,\ldots,m_d)$.
We recall that by construction, we have
\[
H_p(\phi)=\max\{1;\norm{g\cdot m_1};\ldots;\norm{g\cdot m_d}\}.
\]

We want to apply Th.~\ref{thm:compare reductive}. The assumptions over~$\Q_p$ are satisfied by Def.~\ref{defi:Tate}--\ref{defi:Tate1} and~\cite[Th. 3.6]{Ri-Conj-Duke}.

Let~$\widehat{U}=\ad_M(U(p)^\dagger)$, and let~$V$ be the inverse image of~$\widehat{U}^\dagger$ by~$\ad_M:U(p)\to \widehat{U}$. From Cor.~\ref{Sylow}, we deduce~$V^\dagger=U(p)^\dagger$ and~$\ad_M(V^\dagger)=\widehat{U}^\dagger$. 
Thus~$V\leq V^\dagger\cdot Z(M)_{\F_p}$. By Lem.~\ref{lem:4.2}, we have~$[\widehat{U}:\widehat{U}^\dagger]\leq c'(n)$. 
It follows~$[U(p):V]\leq c'(n)$. Using~\ref{defi:Tate}--\ref{defi:Tate2}, we deduce that, for~$p\gg0$,
the action of~$V$ on~$\mathfrak{g}_{\F_p}$ is semisimple and~$Z_{G_{\F_p}}(V)=Z_{G_{\F_p}}(M_{\F_p})$. We may thus apply Prop.~\ref{propMNich} and Cor.~\ref{corMNich}, for~$(x_1,\ldots,x_k,y_1,\ldots,y_l)=(\ol{X_1},\ldots,\ol{X_k},\ol{Y_1/p},\ldots,\ol{Y_l/p})$.

It follows that the assumptions over~$\F_p$ of Th.~\ref{thm:compare reductive} are satisfied (in the representation~$\mathfrak{g}^{k+l+d}$, sum of adjoint representations). 

We may thus apply Theorem~\ref{thm:compare reductive}, and we deduce
\[
H_p(g\cdot v')\geq H_p(\phi)^{C(\Sigma(\rho))},
\]
where~$\Sigma(\rho)$ is the set of roots of~$G$ and does not depend on~$p$.  This proves~\eqref{eq} with~$c(\rho):=C(\Sigma(\rho))$.
\end{proof}

\begin{corollary}In particular, if~$H_{v'}(\phi)\notin\{1;p\}$ we have
\[
[\phi(U):\phi(U)\cap G(\ZZ_p)]\geq H_p(\phi)^{c(\rho)/2}.
\]
\end{corollary}
\begin{proof}
We recall that, because~$H_{v'}(\phi)\in p^{\ZZ}$, we have~$H_{v'}(\phi)\geq p^2$ as soon as~$H_p(\phi)\notin\{1;p\}$.
It follows
\[
H_v(\phi)\geq H_{v'}(\phi)/p\geq H_{v'}(\phi)^{1/2}\geq H_p(\phi)^{1/(2\cdot c(\rho))}.\qedhere
\]
\end{proof}
In proving Th.~\ref{thm:local Galois} we may now assume that~$H_{v'}(\phi)\leq p$.

We define
\[
H_X(\phi)=\max\{1;\norm{\phi(X_1)};\ldots;\norm{\phi(X_k)}\}
\]
and
\[
H_Y(\phi)=\max\{1/p;\norm{\phi(Y_1)};\ldots;\norm{\phi(Y_k)}\},
\]
so that
\[
H_{v'}(\phi)=\max\{H_X(\phi);p\cdot H_Y(\phi)\}
\]
and
\[
H_{v}(\phi)=\max\{H_X(\phi); H_Y(\phi)\}.
\]

If~$H_X(\phi)=p$, we have, by~\eqref{galois exp bound},
\[
[\phi(U):\phi(U)\cap G(\ZZ_p)]\geq H_v(\phi)=H_X(\phi)=
H_{v'}(\phi)\geq H_p(\phi)^{c(\rho)}. 
\]
We now assume~$H_X(\phi)=1$. We have~$p\cdot H_Y(\phi)=H_{v'}(\phi)\in\{1;p\}$.
\subsubsection{The case~$p\cdot H_Y(\phi)=H_X(\phi)=1$}
In this case we have~$H_{v'}(\phi)=1$, and by~\eqref{eq},
\[
H_p(\phi)=1.
\]
Obviously
\[
[\phi(U_p):\phi(U_p)\cap G(\ZZ_p)]\geq H_p(\phi).
\]

\subsubsection{The case~$p\cdot H_Y(\phi)=H_{v'}(\phi)=p$}
\label{last case}

From~\eqref{XY3} of Prop.~\ref{prop X Y}, for every~$Y_i$, there exists~$Z_i\in \mathfrak{z(m)}_{\Z_p}$ such that
\[
\frac{1}{p}Y_i\equiv Z_i\pmod{p\cdot\mathfrak{m}_{\Z_p}}.
\]
Define~$v''=(X_1,\ldots,X_k,Z_1,\ldots,Z_l)$. Then the reductions modulo~$p$ are equal
\[
\ol{v'}=\ol{v''}\text{ in }{\mathfrak{m}_{\F_p}}^{k+l}\leq{\mathfrak{g}_{\F_p}}^{k+l}.
\]
Thus
\begin{itemize}
\item The orbit~$G_{\F_p}\cdot \ol{v'}$ is equal to~$G_{\F_p}\cdot \ol{v''}$ and is closed;
\item and~$Stab_{G_{\F_p}}(\ol{v'})=Stab_{G_{\F_p}}(\ol{v''})=Z_{G_{\FF_p}}(M_{\FF_p})$ (cf~\eqref{XY4} of Prop.~\ref{prop X Y}).
\end{itemize}
Applying Th.~\ref{pKN}\footnote{If the adjoint representation is not faithful, we apply Th.~7.1 to the representation~$\mathfrak{g}^{k+l}\oplus W$, where~$G\to GL(W)$ is a faithful representation, and vectors~$v=(v',0)\in\mathfrak{g}^{k+l}\oplus W$ and~$v=(v'',0)\in \mathfrak{g}^{k+l}\oplus W$. This ensures that the representation~$G\to GL(\mathfrak{g}^{k+l}\oplus W) $ is faithful, as assumed in Th.~7.1.}
\[
\phi(v')\in{\mathfrak{g}_{\Z_p}}^{k+l}\text{ if and only if }
\phi(v'')\in{\mathfrak{g}_{\Z_p}}^{k+l}.
\]
We have, as functions of~$\phi$,
\[
H_{v''}=\max\{H_X;H_Z\}\text{ with }H_Z:\phi\mapsto\max\{1;\norm{\phi(Z_1)};\ldots;\norm{\phi(Z_l)}\}.
\]
Because~$H_X(\phi)=1$ and~$H_{v'}(\phi)=p\neq 1$ in~\S\ref{last case}, we have
\begin{equation}\label{Z not 1}
H_Z(\phi)\neq 1.
\end{equation}
Denote by~$\pi:G(\Z_p)\to G(\F_p)$  the reduction modulo~$p$ map.
Let~$N:=Z(M)^0(\Z_p)\cap \ker(\pi)$.
For~$p$ large enough, the torus~$Z(M)^0$ has good reduction over~$\Z_p$, and, for~$p$ large enough,~$N=\exp(2p\mathfrak{z}(\mathfrak{m}))$. Thanks to~\eqref{Z not 1} we can apply\footnote{The reference~\cite[4.3.9]{EdYa} is phrased in terms of orbits of lattices.  See~\cite[\S{}B.1, proof of conclusion 3 after Th.~B.5]{RY} for the precise relation with the indices~\eqref{EdYa index}.}~\cite[4.3.9]{EdYa} and~\cite[Th.~B5 (3)]{RY} with the torus~$Z(M)^0$: we have, for some~$c\in\R_{>0}$ that does not depend on~$p$,
\begin{equation}\label{EdYa index}
[\phi(Z(M)^0(\Z_p)):\phi(Z(M)^0(\Z_p))\cap (\phi(N)\cdot G(\Z_p))]\geq p/c.
%[\phi(Z(M)^0(\Z_p)):\phi(Z(M)^0(\Z_p))\cap G(\Z_p)]\geq p/c.
\end{equation}

Define~$\Gamma:=Z(M)^0
(\Z_p)\cdot U_p$. As~$Z(M)^0
(\Z_p)\leq \Gamma$, we have
\[
[\phi(\Gamma):\phi(\Gamma)\cap (\phi(N)\cdot G(\Z_p))]\geq
[\phi(Z(M)^0(\Z_p)):\phi(Z(M)^0(\Z_p))\cap (\phi(N)\cdot G(\Z_p))]
\]
Let~$K:=\Gamma\cap \stackrel{-1}{\phi}(G(\Z_p))$. Then~$\Gamma\cap (N\cdot \stackrel{-1}{\phi}(G(\Z_p)))=N\cdot K$, and
\[
[\Gamma:N\cdot K]\geq p/c.
\]
We have
\begin{multline}\label{}
\#\phi(U_p)\cdot G(\Z_p)/G(\Z_p)=
[\phi(U_p):\phi(U_p)\cap G(\Z_p)]=[U_p:U_p\cap K]\\
\geq [U_p:U_p\cap (K\cdot N)]=[U_p\cdot N\cdot K:N\cdot K].
\end{multline}
The formula~$[G_1:G_3]=[G_1:G_2]\cdot [G_2:G_3]$ with ~$G_1=\Gamma$, and~$G_2=U_p\cdot N\cdot K$ 
and~$G_3=N\cdot K$ gives
\begin{equation}\label{lem:eq:indice}
[\Gamma:N\cdot K]=[\Gamma:U_p\cdot N\cdot K]\cdot [U_p\cdot N\cdot K: N\cdot K].
\end{equation}
We deduce
\begin{equation}\label{Galois frac}
\#\phi(U_p)\cdot G(\Z_p)/G(\Z_p)\geq \frac{p}{c\cdot [\Gamma:U_p\cdot N\cdot K]}.
\end{equation}
Recall that~$g\in G(\Z_p)$ is "topologically $p$-nilpotent" if and only if the order of~$\pi(g)\in G(\F_p)$ is a power of~$p$. For~$H\leq G(\Z_p)$ we denote by~$H^\dagger$ the subgroup generated by the topologically $p$-nilpotent elements contained in~$H$. Note that~$\ker(\pi)\cap H\leq H^\dagger$.

Recall that~$Z(M)^0(\Z_p)$ commutes with~$U_p\leq M(\Z_p)$. The product map~$Z(M)^0(\Z_p)\times U_p\to\Gamma$ is thus a surjective group homomorphism.
It follows from~\ref{lifting p-nilpotent elements} with~$U=Z(M)^0(\Z_p)\times U_p$ and~$U'=\Gamma$ that we have
\[
\Gamma^\dagger=Z(M)^0(\Z_p)^\dagger\cdot U_p^\dagger.
\]
Thus~$\ker(\pi)\cap \Gamma\leq \Gamma^\dagger\leq Z(M)^0(\Z_p)^\dagger\cdot U_p^\dagger$.
As~$Z(M)^0$ is a torus, we have~$p\nmid \#Z(M)^0(\F_p)$, and thus~$Z(M)^0(\Z_p)^\dagger=\ker(\pi)\cap Z(M)^0(\Z_p)=N$. Because~$H_X(\phi)=1$, we have~$\phi(X_i)\in \mathfrak{g}_{\Z_p}$. Thus, for~$p\gg0$,
we have~$\phi(\exp(X_i))\in G(\Z_p)$ by~\cite[Th.~A3 (74)]{RY}. Together with~\eqref{XY1} of Prop.~\ref{prop X Y}, this implies~$\phi(U_p^\dagger)\leq G(\Z_p)$. Thus~$U_p^\dagger \leq K$.

We deduce
\[
\ker(\pi)\cap \Gamma\leq \Gamma^\dagger=Z(M)^0(\Z_p)^\dagger\cdot U_p^\dagger\leq N\cdot K.
\]

Let~$z\in Z(M)^0(\Z_p)$ be such that~$\pi(z)\in  Z(M)^0(\F_p) \cap U(p)$. As~$z\in U(p)=\pi(U_p)$, we have~$\pi(z)\in U_p\cdot (\ker(\pi)\cap \Gamma)\leq U_p\cdot K \cdot N$. Thus~$z\in  Z(M)^0(\Z_p)\cap (U_p\cdot K \cdot N)$. It follows that the map
\[
Z(M)^0(\Z_p)/(Z(M)^0(\Z_p)\cap (U_p\cdot K \cdot N))\to Z(M)^0(\F_p)/(Z(M)^0(\F_p)\cap U(p))
\]
is injective. We deduce
\begin{multline}
[\Gamma:\Gamma\cap (K\cdot N)]=
[U_p\cdot Z(M)^0(\Z_p):(U_p\cdot Z(M)^0(\Z_p))\cap ( K \cdot N)]\\
\leq [Z(M)^0(\Z_p):Z(M)^0(\Z_p)\cap (U_p\cdot K \cdot N)]
\leq [Z(M)^0(\F_p):(Z(M)^0(\F_p)\cap U(p)].
\end{multline}
By Cor.~\ref{coro big dans Z} and~\eqref{Galois frac}, this implies
\[
\#\phi(U_p)\cdot G(\Z_p)/G(\Z_p)\geq \frac{p}{c\cdot \gamma(n)\cdot C_{MT}}=\frac{H_{v'}(\phi)}{c\cdot \gamma(n)\cdot C_{MT}}.
\]
Using~\eqref{eq} we conclude
\[
[\phi(U_p):\phi(U_p)\cap G(\Z_p)]\geq\frac{1}{c\cdot \gamma(n)\cdot C_{MT}}\cdot H_p(\phi)^{c(\rho)}.
\]
This proves~\eqref{precise with a} with~$c=c(\rho)$ and~$a=1/(c\cdot \gamma(n)\cdot C_{MT})$. We have proven Th.~\ref{thm:local Galois} and Th.~\ref{Galois bounds}.

%\begin{lemma}\label{lem:indices}
%Let~$\Gamma$ be a group, and let~$H,N,K\leq \Gamma$ be subgroups
%such that~$H,N\leq \Gamma$ are normal subgroups and~$\Gamma/K$ is finite.
%
%Then
%\begin{equation}\label{lem:eq:indice}
%[\Gamma:N\cdot K]=[\Gamma:H\cdot N\cdot K]\cdot [H\cdot N:(H\cdot N)\cap (N\cdot K)].
%\end{equation}
%\end{lemma}
%\begin{proof}It follows from~$ [H\cdot N:(H\cdot N)\cap (N\cdot K)]= [H\cdot N\cdot K:N\cdot K]$
%and the formula~$[G_1:G_3]=[G_1:G_2]\cdot [G_2:G_3]$ with ~$G_1=\Gamma$, and~$G_2=H\cdot N\cdot K$ 
%and~$G_3=N\cdot K$.
%\end{proof}

\subsection{Some Structure Lemmas} We consider the situation of Theorem~\ref{thm:local Galois}. We identify~$G$ with its image by a faithful representation in~$GL(n)$ such that~$G(\Z_p)=GL(n,\Z_p)\cap G$, and we denote by~$U(p)$ the image of~$U_p$ in~$G(\F_p)\leq GL(n,\F_p)$.
We denote by~$\ol{M}=M_{\F_p}$ the~$\F_p$-algebraic group from the model of~$M$ over~$\Z_p$ induced by~$M\leq GL(n)$, and we denote~$Z(M_{\F_p})$ the centre of~$M_{\F_p}$.

In Lem.~\ref{lem:4.1}, Prop.~\ref{prop4.7} Cor.~\ref{coro big dans Z}, the quantity depending on~$n$ also
depends implicitly on the function~$D\mapsto M(D)$ in~\ref{defi:Tate2} of~Def.~\ref{defi:Tate}.

We recall that the~$U_p$ satisfy the uniform integral Tate conjectures as in
Def.~\ref{defi:indep} and~\ref{defi:Tate}. In this section the statements are valid for almost every prime~$p$.
\begin{proposition}\label{prop4.7}
 There exists~$\gamma(n)$ such that,
\[
[Z(M_{\FF_p})(\FF_p):Z(M_{\FF_p})(\FF_p)\cap U(p)]\leq \gamma(n)\cdot [M_{\FF_p}^{ab}(\FF_p):ab_{M_{\FF_p}}(U(p))].
\]
\end{proposition}
By Hypothesis~\ref{thm:galois bounds H1}
of Th.~\ref{Galois bounds} we may use~\eqref{defi CMT} and deduce the following.
\begin{corollary}\label{coro big dans Z}
With~$C_{MT}$ as in~\eqref{defi CMT}, we have
\[
[Z(M_{\FF_p})(\FF_p):Z(M_{\FF_p})(\FF_p)\cap U(p)]\leq \gamma(n)\cdot C_{MT}.
\]
\end{corollary}
We prove Proposition~\ref{prop4.7}.
\begin{proof}We recall~$\ol{M}:=M_{\F_p}$.
We consider the isogeny
\[
(ad_{\ol{M}},ab_{\ol{M}}):\ol{M}\to \ol{M}^{ad}\times \ol{M}^{ab}
\]
We write~$U=U(p)$ and denote by~$\wt{U}$ its image in~$\ol{M}^{ad}(\FF_p)\times \ol{M}^{ab}(\FF_p)$.

We denote~$\wt{U}_1$ and~$\wt{U}_2$ its images by the projections on the two factors.

From Lemma~\ref{lem:4.1}, we have
\[
[\wt{U}_1:\wt{U}\cap \ol{M}^{ad}(\FF_p)]\leq 
[\wt{U}_1:\wt{U}^\dagger]\leq c(n).
\]
By Goursat's Lemma~\ref{Goursat}
\[
[\wt{U}_2:\wt{U}\cap \ol{M}^{ab}(\FF_p)]=[\wt{U}_1:\wt{U}\cap \ol{M}^{ad}(\FF_p)]\leq c(n).
\]
Let~$U'\leq U$ 
 be the inverse image of~$\wt{U}\cap \ol{M}^{ab}(\FF_p)$   in~$U$, 
 and~$U''\leq \ol{M}(\F_p)$  be the inverse image of~$\wt{U}\cap \ol{M}^{ab}(\FF_p)$ in~$\ol{M}(\FF_p)$.
Because~$Z(\ol{M})$ is the $(ad_{\ol{M}},ab_{\ol{M}})$-inverse image of~$\ol{M}^{ab}$ in~$\ol{M}$, we have
\[
U'\leq U''\leq Z(\ol{M})(\FF_p).
\]
Define~$F:=Z(\ol{M})\cap \ol{M}^{der}$, which is a finite~$\FF_p$-algebraic group of degree at most~$c_2(\dim(M))\leq c_2(n)$.

As we have~$U''\leq F(\ol{\FF_p})\cdot U'$, we have
\[
[U'':U']\leq \# F\leq c_2(n).
\]
On the other hand, we have
\[
[Z(\ol{M})(\FF_p):U'']\leq [ \ol{M}^{ab}(\FF_p):\wt{U}\cap \ol{M}^{ab}(\FF_p)].
\]
It follows
\begin{multline}
[Z(\ol{M})(\FF_p):U']\leq c_2(n)\cdot[ \ol{M}^{ab}(\FF_p):\wt{U}\cap \ol{M}^{ab}(\ol{\F_p})]\\
\leq c(n)\cdot c_2(n)\cdot[ \ol{M}^{ab}(\FF_p):\wt{U}_2]\qedhere
\end{multline}
\end{proof}

\begin{lemma}\label{lem:4.1}
Let~$\hat{U}:=ad_{M_{\FF_p}}(U(p))$~be the image of~$U(p)$ in~$M_{\FF_p}^{ad}(\FF_p)$.
Then~$\hat{U}^\dagger=ad_{M_{\FF_p}}(U(p)^\dagger)$, and there exists~$c(n)$
such that
\[
[\hat{U}:\hat{U}^\dagger]\leq c(n).
\]
\end{lemma}
\begin{proof}
The equality~$\hat{U}^\dagger=ad_{\ol{M}}(U(p)^\dagger)$ follows from Cor.~\ref{Sylow}.

From Jordan's theorem~\cite[Th.~3']{SCrit} (applied to~$G=\hat{U}/\hat{U}^\dagger$) there exists~$\hat{U}^\dagger\leq \hat{U}'\leq \hat{U}$ of index~$[\hat{U}:\hat{U}']\leq C(n):=d(n)$
such that
\[
\hat{U}'/\hat{U}^\dagger
\]
is abelian, where~$d(n)$ is as in~\cite[{\S\S4--5}]{SCrit}.%\cite[134. \S7 p.25 about~$GL_d$]{S4}. 
 Without loss of generality, we may assume~$\hat{U}=\hat{U}'$.

Let~$M^{ad}\to GL(n')$ be the adjoint representation in some~$\Q$-linear basis of~$\mathfrak{m}^{ad}$.
For~$p\gg0$, the $\F_p$-fiber of the induced model of~$M^{ad}$ is~$\ol{M}^{ad}$. For~$p\gg0$ the following holds: for any subgroup~$V\leq M(\F_p)$ whose action on~${\F_p}^n$ is semisimple, the action on~${\F_p}^{n'}$ is semisimple (cf.~\cite[Th.~5.4]{SCR}).
%\cite[\S5]{SerreCR}

Let~$c_1(n)$ be the~$c(n)$ of Lem.~\ref{lem:centralisateur adjoint} applied with~$N=n$. Let~$c_2(n)$ be the~$c(n')$ of Lem.~\ref{lem:4.2}, where~$M\leq GL(n)$ is our~$\overline{M^{ad}}\leq GL(n')_{\F_p}$.
Let~$c_3(n)=c_1(n)\cdot c_2(n)$.

Let us prove that~$\wh{U}$ satisfies the assumptions of~\ref{lem:4.2}, where~$M\leq GL(n)$ is our~$M^{ad}\leq GL(n')$, and~$p\geq M(c_3(n))$.

By construction~$\hat{U}/\hat{U}^\dagger$ is abelian. To simplify notations, we write~$U$ for~$U(p)$, in the rest of the proof.

Let~$U'\leq \wh{U}$ be such that~$[\wh{U}:U']\leq c_2(n)$.
Let~$V$ be the inverse image of~$U'$ in~$U$. We have~$[U:V]\leq [\wh{U}:U']\leq c_2(n)$.
We apply~Lem.~\ref{lem:centralisateur adjoint} to~$V$:
there is~$V'\leq V$ such that~$[V:V']\leq c_1(n)$ and
\[
Z_{\ol{M}^{ad}}(ad_{\ol{M}}(V'))=Z_{\ol{M}}(V')/Z(\ol{M}).
\]
We have
\[
[U:V']=[U:V]\cdot [V:V']\leq c_2(n)\cdot c_1(n)=c_3(n).
\]
Since~$p\geq M(c_3(n))$, we may apply~\ref{defi:Tate2} of the Tate hypothesis Def.~\ref{defi:Tate}, and we have
\[
Z_{\ol{M}}(V')=Z(\ol{M}).
\]
It follows that
\[
Z_{\ol{M}^{ad}}(U)
\leq 
Z_{\ol{M}^{ad}}(ad_{\ol{M}}(V')) =\{1\}.
\]
We can thus apply Lemma~\ref{lem:4.2} to~$\hat{U}\leq \ol{M}^{ad}(\FF_p)$, and the conclusion follows.
\end{proof}

%\begin{proof}The equality~$\hat{U}^\dagger=ad_{\ol{M}}(U(p)^\dagger)$ follows from Cor.~\ref{Sylow}.
%
%%Because~$U(p)^\dagger$ is generated by~$p$-order elements, so is~$ad_M(U(p)^\dagger)$, hence
%%\[
%%\hat{U}^\dagger\geq ad_M(U(p)^\dagger).
%%\]
%%Hence we have a well defined epimorphism
%%\[
%%U/U^\dagger\to \hat{U}/\hat{U}^\dagger.
%%\]
%%As the former is of degree prime to~$p$, so is the latter. %Hence~$\hat{U}^\dagger\geq ad_M(U(p)^\dagger)$.
%%Because~$U(p)^\dagger$ is generated by~$p$-order elements, so is~$ad_M(U(p)^\dagger)$, hence
%
%By construction,
%\[
%Z_{\ol{M}^{ad}}(\hat{U})=Z_{\ol{M}}(U)/Z(\ol{M})\leq \ol{M}^{ad}
%\]
%We let~$j(N)$ be the Jordan constant (cf.~\cite[\S5]{SCrit}).
%From Jordan's theorem there exists~$\hat{U}^\dagger\leq \hat{U}'\leq \hat{U}$ of index~$[\hat{U}:\hat{U}']\leq C(n):=j(d(n))$
%such that
%\[
%\hat{U}'/\hat{U}^\dagger
%\]
%is abelian, where~$d(n)$ is as in~\cite[134. \S7 p.25 about~$GL_d$]{S4}.
%
%We assume~$p\geq M(C(n))$ as in~\ref{defi:Tate2} of the Tate hypothesis Def.~\ref{defi:Tate}.
%We use~$c(n)$ and~$m(n)$ from Lemma~\ref{lem:4.2}, and assume~$p>m(n)$ and~$p>M(C(n)\cdot c(n))$.
%
%Then, by~\ref{defi:Tate2} the Tate hypothesis Def.~\ref{defi:Tate}, the uniform integral Tate conjecture, Lemma~\ref{lem:4.2} applies to~$\hat{U}'\leq \ol{M}^{ad}(\FF_p)$.
%
%The conclusion follows.
%\end{proof}

\begin{lemma}\label{lem:centralisateur adjoint}
 For every~$N\in\Z_{\geq 1}$ there exists~$c(N)$ such that,
for any prime~$p$ and any reductive algebraic subgroup~$M\leq GL(N)_{\F_p}$, and any~$U\leq M(\F_p)$, there exists~$U'\leq U$ with
\[
[U:U']\leq c(N)
\]
and
\[
Z_{M^{ad}}(ad_M(U'))=Z_M(U')/Z(M)
\]
where~$ad_M:M\to M^{ad}$ is the quotient map to the adjoint group.
\end{lemma}
\begin{proof}Let~$M^{der}$ be the derived subgroup of~$M$ and~$F=Z(M^{der})=Z(M)\cap M^{der}$ be its centre.

We claim that there exists~$e\in\Z_{\geq1}$ such that for every reductive subgroup~$M\leq G$, we have
\[
\left.\#{Z(M^{der})(\ol{\F_p})}^{\phantom{l}}\middle|e\right..
\]
\begin{proof}
We have~$\#F({\F_p})\leq \#F(\ol{\F_p})\leq \#\pi_1(M^{ad})$. 
%Because~$\dim(M^{der})$ is bounded by~$N^2$, there are finitely many $\ol{\F_p}$-isomorphism classes for~$M^{ad}$.
%We~$\#F({\F_p})\leq \#F(\ol{\F_p})\leq \#\pi_1(M^{ad})$ can be bounded in terms of~$p$ and~$N$. 
As in~\cite[Proof of Lem.~2.4]{UYtowards}, one can classify~$M^{ad}$ in terms of its Dynkin diagram, and, since~$\dim(M^{ad})$ is bounded, deduce a bound on~$\#\pi_1(M^{ad})$. 
\end{proof}
Define
\[
U[e]:=\bigcap_{\phi\in\mathrm{Hom}{}(U,\Z/(e))}\ker(\phi).
\]
According to~\cite[Prop. 6.7]{RY}, the group~$U$ is generated by~$k(N)$ elements. It follows that
\[
[U:U[e]]\leq e^{k(N)}.
\]
Let~$m\in M(\ol{\F_p})$ such that~$\wh{m}:=ad_M(m)$ belongs to~$Z_{M^{ad}}(ad_M(U))$.
Denote by~$\phi_m$ the map
\[
u\mapsto mum^{-1}u^{-1}:U\to M(\ol{\F_p}).
\]
We claim that~$\phi_m$ is a morphism~$U\to F(\ol{\F_p})$.
\begin{proof}Consider~$u\in U$ and let~$\wh{u}=ad_M(u)$. We have~$\wh{m}\wh{u}\wh{m}^{-1}=\wh{u}$,
and this implies~$mum^{-1}u^{-1}\in \ker(ad_M)$. It follows
\[
\phi_m(U)\subseteq Z(M)(\ol{\F_p}).
\]
Since~$mum^{-1}u^{-1}$ is a commutator, we have~$mum^{-1}u^{-1}\in [M(\ol{\F_p}),M(\ol{\F_p})]=M^{der}(\ol{\F_p})$. It follows
\[
\phi_m(U)\subseteq M^{der}(\ol{\F_p}).
\]
Thus
\[
\phi_m(U)\subseteq F(\ol{\F_p})=M^{der}(\ol{\F_p})\cap Z(M)(\ol{\F_p})
\]
For~$u,u'\in U$, we have
\[
\phi_m(uu')=muu'm^{-1}(uu')^{-1}=mum^{-1}mu'm^{-1}u'^{-1}u^{-1}=mum^{-1}\phi_m(u')u^{-1}.
\]
Since~$\phi_m(u')\in Z(M)(\ol{\F_p})$, we have~$\phi_m(u')u^{-1}=u^{-1}\phi_m(u')$, and
\[
\phi_m(uu')=mum^{-1}u^{-1}\phi_m(u')=\phi_m(u)\phi_m(u').
\]
obviously,~$\phi_m(1)=1$. 
The claim follows.
\end{proof}
Since~$\#F| e$, by construction of~$U[e]$, we have
\[
\phi_m(U[e])=\{1\}.
\]
Equivalently~$m\in Z_M(U[e])$.

This proves the Lemma with~$U'=U[e]$ and~$c(N)=e^{k(N)}$.
\end{proof}

\begin{lemma}\label{lem:4.2}
For every~$n\in\Z_{\geq1}$, there exists~$c(n)$, $c'(n)$, $m(n)$ such that the following holds.

Let~$p>m(n)$ be a prime, let~$M\leq GL(n)$ be adjoint over~$\FF_p$, and let
\[
\hat{U}\leq M(\FF_p)
\]
be a subgroup
\begin{itemize}
\item such that~$\hat{U}/\hat{U}^\dagger$ is abelian 
\item and such that for every~$U'\leq \hat{U}$ of index at most $c(n)$:
\begin{enumerate}
\item \label{cond 1} we have~$Z_M(U')=1$;
%\item \label{cond 2} we have~$\hat{U}/\hat{U}^\dagger$ abelian
\item \label{cond 3} the action of~$U'$ is semisimple;
\end{enumerate}
\end{itemize}
Then we have~$[\hat{U}:\hat{U}^\dagger]\leq c'(n)$.
\end{lemma}
\subsubsection{Remark}\label{Remark Lemma}
 In proving the Lemma, we may substitute~$\hat{U}$ with~$U'$ if~$\hat{U}^\dagger\leq U'\leq \hat{U}$
and~$[\hat{U}:U']\leq f(n)$, where~$f(n)$ depends only on~$n$. We then have to change~$c(n)$ into~$c(n)\cdot f(n)$
accordingly.

\begin{proof}We assume~$p>p(n)$, for some~$p(n)$ depeding only on~$n$, so that we can apply Nori theory~\cite{N}.
 
We denote by~$S\leq M$ the~$\F_p$-algebraic group associated by Nori to~$\hat{U}^\dagger$,
and denote by~$N=N_M(S)$ the normaliser of~$S$ in~$M$. 

We deduce from~\eqref{cond 3} that~$S$ is semisimple, and thus~$N^0=S\cdot Z_M(S)^0$.

We recall that the semisimple Lie subalgebras contained in~$\mathfrak{gl}(n)_{\ol{\F_p}}$ can assume finitely many types, independently from~$p$. We deduce uniform bounds
\begin{equation}
\#N/N^0\leq c_1(n)\text{ and }\#Z(S)\leq c_2(n).
\end{equation}

As~$\wh{U}^\dagger$ is a characteristic subgroup of~$\wh{U}$ it is normalised by~$\wh{U}$. It follows
that~$\wh{U}$ normalises the associated semisimple subgroup~$S$ by Nori.
We have~$\hat{U}\leq N$. Let~$N^\dagger\leq N$ be the algebraic subgroup generated by one parameter unipotent subgroups. Then~$N^\dagger$ is connected, and we have~$U^\dagger\leq N^\dagger\leq N^0$. If~$U'=\hat{U}\cap N^0$
we have~$[\hat{U}:U']\leq c_1(n)$ and~$U^\dagger\leq U'$. Using the Remark~\ref{Remark Lemma},
we may replace~$\hat{U}$ by~$U'=\hat{U}\cap N^0$.

We denote~$Z(S)=Z_M(S)\cap S$ and we consider
\[
N^0\to S^{ad}\times Z_M(S)/Z(S).
\]
We denote~$\wt{U}$ the image of~$\hat{U}$, by
\[
\wt{U}_1\leq S^{ad}(\F_p)\text{ and~}\wt{U}_2\leq (Z_M(S)/Z(S))(\F_p)
\] the projections of~$\wt{U}$ 
 and define
\[U'_1:=\wt{U}\cap (S^{ad}(\FF_p)\times\{1\})
\text{ and } U'_2:=\wt{U}\cap(\{1\}\times Z_M(S)/Z(S)).\]

 From Cor.~\ref{Sylow}, 
the image, in~$S^{ad}(\FF_p)\times  Z_M(S)/Z(S)$, of~$S(\F_p)^\dagger\leq \hat{U}$ is~$S^{ad}(\FF_p)^\dagger$.
Thus
\[
S^{ad}(\FF_p)^\dagger\times\{1\} \leq U'_1\leq \wt{U}_1\times\{1\}\leq S^{ad}(\FF_p)\times\{1\}.
\]
%The projection of~$\wt{U}_1$
With~$r(n)$ given  by~\cite[3.6(iv-v) and p.\,270]{N}, we have
\begin{equation}\label{rN}
[\wt{U}_1\times\{1\}:U'_1]\leq [S^{ad}(\FF_p):S^{ad}(\FF_p)^\dagger]\leq r(n)=2^{n-1}.
\end{equation}

%\begin{itemize}
%\item the kernel of the projection map on~$Z_M(S)/Z(S)$ contains~$\hat{U}^\dagger$,
%and thus, by~\eqref{cond 2} the image group~$\wt{U}_2$ is abelian of order prime to~$p$;

By Goursat's Lemma~\ref{Goursat} and~\eqref{rN} we have
\[
[\wt{U}_1\times\{1\}:U'_1]=[\{1\}\times\wt{U}_2:U'_2]\leq r(n).
\]
Thus, with~$U':=U'_1\cdot U'_2\simeq U'_1\times U'_2$, we have
\[
[\wt{U}:U']\leq [\wt{U}_1\times\wt{U}_2:U']\leq r(n)^2.
\]
Because~$\wh{U}^\dagger=S(\F_p)^\dagger$ is sent to~$S^{ad}(\F_p)^\dagger\times\{1\}$
(cf. Cor.~\ref{Sylow}) and because~$S^{ad}(\F_p)^\dagger\times\{1\}\leq U'_1\leq U'$
we may use the Remark~\ref{Remark Lemma}, and replace~$\hat{U}$ by the inverse image 
of~$U'$. We denote by
\[
\hat{U}_1,\quad\hat{U}_2
\]
the inverse images of~$U'_1$ and~$U'_2$.

Because~$\hat{U}_2\leq Z_M(S)$ and~$\hat{U}_1\leq S$, the groups~$\hat{U}_1$ and~$\hat{U}_2$ commute with each other.

We reduce the situation to the case where~$\hat{U}_2$ is abelian.

We know that~$\hat{U}/\hat{U}^\dagger$ is abelian and that~$\hat{U}^\dagger\leq S(\F_p)$. It follows that~$\hat{U}_2/(\hat{U}_2\cap S(\F_p))$
is abelian. We have~$F:=\hat{U}_2\cap S(\F_p)\leq Z_M(S)\cap S=Z(S)$, and thus~$\abs{F}\leq c_2(n)$,
and~$\hat{U}_2$ is an extension of the abelian group~$U'_2=\hat{U}_2/(\hat{U}_2\cap S)=\wt{U}\cap Z_M(S)/Z(S)$ by a finite group~$F$ of order at most~$c_2(n)$. Moreover,~$U'_2$ is of order prime to~$p$, and thus is diagonalisable in~$Z_M(S)/Z(S)(\ol{\F_p})$. It follows we can find a monomorphism~$U_2'\leq (\ol{\F_p}^\times)^{rk(Z_M(S))}$ where~$rk(Z_M(S))$ is the rank of~$Z_M(S)$. We have~$rk(Z_M(S))\leq rk(M)\leq n$.
From Cor.~\ref{coro extension}, there exists an abelian subgroup~$U''_2\leq \hat{U}_2$ of index at most
\[
c_3(n)=c_2(n)\cdot e(c_2(n))^n.
\]

Using the remark we may replace~$\hat{U}=\hat{U}_1\cdot \hat{U}_2$ by~$U'=\hat{U_1}\cdot U''_2$.
Then~$U''_2$ commutes with~$\hat{U}_1$ and with itself:~$U'_2$ is in the centre of~$\hat{U}$. By Hypothesis~\eqref{cond 1}, we have~$U''_2=1$.

Thus~$\hat{U}=\hat{U_1}\leq S(\FF_p)$ and
\[
[\hat{U}:\hat{U}^\dagger]\leq [S(\FF_p):S(\FF_p)^\dagger]\leq r(n).\qedhere
\]
\end{proof}
%Claim: see {https://math.stackexchange.com/questions/129228/finite-extensions-of-groups}

\subsubsection{Other lemmas}
\begin{lemma}\label{Lemma bounded and centraliser}
 Let~$H\leq G\leq GL(d)$ be algebraic groups over~$\Q$
with~$H$ Zariski connected.
For every~$\lambda\in\Z_{\geq 1}$, there exists~$N\in\Z_{\geq0}$ such that:
for all prime~$p\geq N$, and for all subgroups~$U\leq H_{\F_p}(\F_p)$ such that
\[
[H_{\F_p}(\F_p):U]\leq \lambda
\]
we have
\[
Z_{G_{\F_p}}(U)=Z_{G_{\F_p}}(H_{\F_p}).
\]
\end{lemma}
\begin{proof}Without loss of generality we may assume~$G=GL(d)$.
We define a scheme~$X\leq Y\times Y$, with~$Y=GL(d)_{\Z}$ by
\[
X=\{(g,h)\in GL(d)_{\Z}\times GL(d)_{\Z}|[g,h]=1, h\in H\}
\]
and denote by~$\phi: X\to Y$ the first projection. According to the Lemma~\ref{lemma schemes},
for every prime~$p$, and every~$g\in G(\ol{\F_p})$,
\[
\pi_0(H\cap Z_{G_{\ol{\F_p}}}(\{g\}))\leq \gamma.
\]
%(si~$H'=H\cap Z_{G_{\ol{\F_p}}}(\{g\})$ est defini sur~$\F_p$:)

For~$p\gg0$, the~$\F_p$-group~$H_{\F_p}$ will be Zariski connected. (\citestacks[Lem. 37.28.5.]{055H} with~$y=\Spec(\Q)\in Y=\Spec(\Z)$ and~$X$ the schematic closure of~$H$ in~$GL(d)_{\Z}$).

From~\cite[Lem.~3.5]{N}, we have, for any~$g\in GL(d,\F_p)$ and~$H':=X_g=H\cap Z_{G_{\ol{\F_p}}}(\{g\})$,
\[
\#H'(\F_p)\leq (p+1)^{\dim(H')}\cdot \gamma
\]
and
\[
\#H(\F_p)\geq (p-1)^{\dim(H)}.
\]
Let~$\lambda(p)=\frac{1}{\gamma}\cdot (p-1) \cdot (\frac{p-1}{p+1})^{\dim(H)-1}$.
Let~$\lambda$ be given.
Since~$\lim_p \lambda(p)=+\infty$, there exists~$N$ such that for~$p\geq N$, we have~$\lambda(p)>\lambda$.

%Let~$\lambda=\gamma\cdot (\frac{p-1}{p+1})^{\dim(G)}$ and
Let
\[
U\leq H(\F_p)
\]
be such that
\[
Z_{G_{\F_p}}(U)\neq Z_{G_{\F_p}}(H_{\F_p}).
\]
We have~$Z_{G_{\F_p}}(U)\geq  Z_{G_{\F_p}}(H_{\F_p})$.

Recall that~$G=GL(d)$.
We remark that~$Z_{G(\ol{\F_p})}(U)$ and~$Z_{G(\ol{\F_p})}(H_{\ol{\F_p}})$ are Zariski connected
because they are non empty Zariski open subsets~$A\cap GL(d,\ol{\F_p})$ in a subalgebra~$A\leq End(\ol{{\F_p}}^d)$.
For~$p-1> \left(\frac{p+1}{p-1}\right)^{\dim(G)}$ (cf.~\cite[Lem.~3.5]{N}), we will have
\[
\#Z_{G_{\F_p}}(U)({\F_p})> \#Z_{G_{\F_p}}(H_{\F_p})({\F_p})
\]
and there exists
\[
g\in Z_{G_{\F_p}}(U)({\F_p})\smallsetminus Z_{G_{\F_p}}(H_{\F_p})({\F_p}).
\]
We have then, with~$H'=X_g=Z_{G_{\F_p}}(\{g\})\cap H_{\F_p}$, which is defined over~$\F_p$,
\[
H'<H.
\]
Because~$H$ is connected, we have~$\dim(H')<\dim(H)$ and
\begin{multline*}
\#U\leq \#H'(\F_p)\leq \#\pi_0(H')\cdot \# H'^0(\F_p)\leq \gamma\cdot (p+1)^{\dim(H)-1}
\\
\leq 
\frac{1}{\lambda}\cdot (p-1)^{\dim(H)}\leq \frac{1}{\lambda}\cdot \#H(\F_p).
\end{multline*}
The Lemma follows.
\end{proof}

\begin{lemma}\label{lemma schemes}
Let~$\phi:X\to Y$ be a morphism of schemes of finite type over~$\Z$. 

Then there exists~$\gamma$ such that: for every field~$K$, and every~$y\in Y(K)$, the number of geometric connected components~$\#\pi_0(X_y)$ of the fibre~$X_y$ satisfies
\[
\#\pi_0(X_y)\leq \gamma.
\]

If~$\phi$ is flat, then~$y\mapsto \dim(X_y)$ is lower semicontinuous on~$Y$.

%If~$\phi$ is flat, then~$y\mapsto \#\pi_0(X_y)$ is lower semicontinuous on~$Y$???(pas tjs)
\end{lemma}
% plat vs univ ouvert (2.4.6) (15.2.2)
%Remplacer Lemma 37.49.7. 

%Definition comptage: Lemma 37.26.3. 

%Lemma 37.25.6. Utiliser?
\begin{proof}If~$Y$ is non-empty, there exists a proper closed subset outside of which the  function~$y\mapsto\#\pi_0(X_y)$  is constant, according to\footnote{Applied to the generic point of an irreducible component of~$Y$. The latter exists because~$Y$
is noetherian.} \citestacks[Lemma 37.28.5.]{055H}. We conclude by noetherian induction.

The second assertion, on dimensions, is~\citestacks[Lemma 37.28.4.]{0D4H}.% (compare \cite[15.5.1]{EGA42}.)
\end{proof}
%\begin{proof}[old proof]We have~$\#\pi_0(X_y)\leq N(y)$ where~$N(y)$ is the number of geometric (irreducible) components.
%According to~\cite[Lemma 37.26.6]{Stacks}, the function~$N(y)$ is constant on locally constructible subsets
%of~$Y$. Because~$Y$ is of finite type, it is quasicompact, and these are constructible subsets. Because~$Y$
%is also noetherian, any constructible partition is finite (ref). Hence~$N(y)$ takes finitely many values, and thus~$\#\pi_0(X_y)$
%is bounded.
%\end{proof}

\begin{lemma}\label{lifting p-nilpotent elements}
Let~$\phi:U\to U'$ be a continuous epimorphism of profinite groups. Let~$p$ be a prime number.

Recall that an element~$u\in U$ is said to be "topologically~$p$-nilpotent" if~$\lim_{n\to\infty}u^{p^n}=1$. 
Equivalently, for every continuous map~$q:U\to F$ to a finite group~$F$, the order of~$q(u)$ is a power of~$p$.

Then every "topologically~$p$-nilpotent" $u'\in U'$ is of the form~$u'=\phi(u)$ for some "topologically~$p$-nilpotent"~$u\in U$.
\end{lemma}
\begin{proof}Let~$u'\in U'$  be arbitrary. As~$\phi$ is surjective, there exists~$v\in U$ such that~$\phi(v)=u'$. 
The homomorphism~$\pi:k\to v^k:\Z\to U$ extends uniquely to a continuous homomorphism~$\widehat{\Z}\to U$ on the profinite completion~$\widehat{\Z}$ of~$\Z$. 

Let~$1\in \widehat{\Z}$ be the multiplicative unit, and, for a prime number~$\ell$, let~$1_\ell\in \Z_\ell\subseteq \widehat{\Z}$ be the multiplicative unit of the~$\ell$-adic factor~$\Z_\ell$. We note that~$1=\lim_{n\to +\infty}\sum_{\ell\leq n}1_\ell$. It follows that~$v=\lim_{n\to +\infty}\sum_{\ell\leq n}\pi(1_\ell)$.

As~$u'$ is "topologically~$p$-nilpotent", the map~$\phi\circ\pi:k\to u'^k:\Z\to U'$ factors through~$\Z\to \Z_p\to U'$. 
The closed subgroup~$\ol{u'^\Z}\leq U'$ generated by~$u'$ is thus a quotient of~$\Z_p$.

Note that, for~$\ell\neq p$, the only continuous additive homomorphism~$\Z_\ell\to \Z_p$ or~$\Z_\ell\to \Z/(p^k)$ maps~$1_\ell$ to~$0$. We deduce that for~$\ell\neq p$, the morphism~$\phi\circ \pi:\widehat{\Z}\to U\to U'$ maps~$1_\ell$ to~$u'^0=1$. It follows that, for~$n\geq p$, we have~$\phi\circ \pi(1_p)=\phi\circ \pi(\sum_{\ell\leq n}1_\ell)$. But~$\lim_{n\to\infty}\phi\circ \pi(\sum_{\ell\leq n}1_\ell)=\phi\circ \pi(1)=u'$. We deduce that
\[
\phi\circ\pi(1_p)=u'.
\]
Since~$1_p\in\Z_p$, it is topologically~$p$-nilpotent in~$\Z_p$. Thus~$u:=\pi(1_p)$ is topologically~$p$-nilpotent in~$U$.
We have~$u'=\phi(u)$ where~$u$ is topologically~$p$-nilpotent in~$U$. This concludes the proof.
\end{proof}

\begin{corollary}\label{Sylow}
Let~$\phi:U\to U'$ be an epimorphism of finite groups and, for some prime~$p$, denote by~$U^\dagger$, resp.~$U'^\dagger$ be the subgroup generated by elements of order a power of~$p$.

Then
\[
\phi(U^\dagger)=U'^{\dagger}.
\]
\end{corollary}
\begin{proof}This follows immediately from Lemma~\ref{lifting p-nilpotent elements}.
\end{proof}
%\begin{proof} If~$u$ is of order a power of~$p$, then so is~$\phi(u)$. Thus
%\[
%\phi(U^\dagger)\leq U'^{\dagger}.
%\]
%The subgroup~$U^\dagger$ is normal and is the smallest normal subgroup such that~$U/U^\dagger$
%does not contain a non trivial element of order a power of~$p$; equivalently~$\#U/U^\dagger$ is prime to~$p$.
%Because~$\phi$ is an epimorphism,~$\phi(U^\dagger)$ is normal in~$U'$ and we have a group epimorphism
%\[
%U/U^\dagger\to U'/\phi(U^\dagger).
%\]
%Thus~$\#U'/\phi(U^\dagger)$ is prime to~$p$, and the mentioned minimality property implies
%\[
%\phi(U^\dagger)\geq U'^{\dagger}.\qedhere
%\]
%\end{proof}
%\begin{corollary}\label{cor:Sylow}
%Let~$U\to U'$ be a continuous epimorphism of profinite groups and, for some prime~$p$, denote by~$U^\dagger$, resp.~$U'^\dagger$ be the subgroup generated by topologically~$p$-nilpotent elements. 
%
%Then
%\[
%\phi(U^\dagger)=U'^{\dagger}.
%\]
%\end{corollary}
%\begin{proof}We may write the profinite groups~$U$ and~$U'$ as profinite limits~$U=\projlim U_i$ and~$U'=\projlim U'_i$ with~$\phi(U_i')\leq U_i$ for every~$i$. We observe that~$U^\dagger=\projlim U^\dagger_i$ and~$U'^\dagger=\projlim U'^\dagger_i$. Thus the Corollary~\ref{cor:Sylow} follows from Cor.~\ref{Sylow}.
%\end{proof}
\begin{lemma}\label{Lemma extension}For every~$n\in\Z_{\geq1}$, there exists~$e(n)$ such that for every
short exact sequence of finite groups
\[
1\to N\to G \xrightarrow{\pi} H\to 1
\]
such that~$\#N\leq n$ and~$H$ is abelian,
there exists an abelian subgroup~$H''\leq G$ 
such that
\begin{equation}\label{Lemma extension CCL}
\forall~h \in  H, e(n)\cdot h\in \pi(H'').
\end{equation}
\end{lemma}
% VERSION BUGUÉE
%\begin{lemma}\label{Lemma extension}
%Let
%\[
%1\to N\to G \xrightarrow{\pi} H\to 1
%\]
%be a short exact sequence of finite groups such that~$H$ is abelian. 
%
%There exist~$e(\#N)\in\Z_{\geq1}$ and~$H'\leq H$ and~$H''\leq G$ such that~$\pi|_H$ is an isomorphism onto~$H''$ and
%\[
%H'\geq e(\#N)\cdot H.
%\]
%\end{lemma}
%The proof is adapted from \texttt{https://math.stackexchange.com/}.
%?https://math.stackexchange.com/questions/129228/finite-extensions-of-groups
\begin{proof} Let~$\rho:G\to\operatorname{Aut}(N)$ be the adjoint action on its normal subgroup~$N$.
Then~$G'=\ker \rho$ has index at most~$\#\operatorname{Aut}(N)\leq (\#N)!$ and we may replace~$G$
with~$G'$ and~$H$ with~$\pi(H)$, that is: we assume the extension of~$H$ by~$N$ is a central 
extension.

Because~$H$ is abelian, the commutator~$n=[a,b]$ of any two~$a,b\in G$ is in~$N$.
We have~$ab=nba$, and, for~$i\in\Z_{\geq 0}$,
\[
a\cdot b^{i+1}=n\cdot b\cdot a\cdot b^i=\ldots=(n\cdot b)^{i+1}\cdot a.
\]
%By induction, we prove
%\[
%a\cdot b^{i}=(n\cdot b)^i\cdot a.
%\]
Because~$N$ is central, we deduce
\[
a\cdot b^{i}=n^i b^{i}\cdot a.
\]
In, particular, when~$i=\gamma:=\#N$, we have
\begin{equation}\label{Lemma extension eq}
[a,b^\gamma]=0.
\end{equation}
Let~$H''\leq G$ be the subgroup  generated by the~$\{g^\gamma|g \in G\}$. Then~$H'=\pi(H'')$ is the subgroup
generated by~$\{\gamma\cdot h |h\in H\}$. Thus~\eqref{Lemma extension CCL} is satisfied.

By~\eqref{Lemma extension eq}, the subgroup~$H''$ has a generating sets made of  elements which are central elements of~$G$. Thus~$H''$ is contained is the centre of~$G$. In particular~$H''$ is abelian.

Lemma~\ref{Lemma extension} is proved.
\end{proof}

%Let~$H''\leq G$ be the subgroup generated by~$\{g^\gamma|g\in G\}$. 
%Let~$\gamma=\#N$, and, for~$h=\gamma\cdot h'$ in~$H':=\gamma\cdot H$, choose~$g'$
%such that~$p(g')=h'$. 
%
%We claim that
%\[
%h\mapsto \sigma(h):=g'^\gamma
%\]
%is a well defined section of~$p$ on~$H'$. This would prove the Lemma for~$e(\#N)=(\#N)!\cdot \#N$
%and~$H''=\sigma(H')$.
%
%We prove that~$\sigma(h)$ does not depend on the choice of~$g'$. Let~$g''=n\cdot g'$ with~$n\in N$.
%Then, because~$N$ is central, we have
%\[
%g''^\gamma=(g'\cdot n)\cdot\ldots\cdot (g'\cdot n)=g'^{\gamma}\cdot n^\gamma=g'^\gamma.
%\]
%We prove~$\sigma(h_1\cdot h_2)=\sigma(h_1)\cdot \sigma(h_2)$. We pick lifts~$g_1,g_2$
%of~$h_1,h_2$. Because~$H$ is commutative,
%we have
%\[
%[g'_1,g'_2]\in N
%\]
%that is~$g_1\cdot g_2=n\cdot g_2\cdot g_1$ for some~$n\in N$. We have, for~$i,j\in\Z_{\geq0}$,
%\[
%g_1^{i+1}\cdot g_2^{j+1}=g_1^{i}\cdot(n\cdot g_2\cdot g_1)\cdot g_2^{j}=
%(g_1^{i}g_2\cdot g_1\cdot g_2^{j})\cdot n
%\]
%and by induction
%\[
%g_1^{i+1}\cdot g_2^{j}=g_2\cdot g_1^{i+1}\cdot g_2^{j}\cdot n^{i+1}.
%\]
%We deduce by induction, for~$i=j=\gamma$,
%\[
%g_1^{\gamma}g_2^{\gamma}=g_2^{\gamma}g_1^{\gamma}n^{\gamma^2}
%=g_2^{\gamma}g_1^{\gamma}.
%\]
%The Lemma is proved.
%\end{proof}
\begin{corollary}\label{coro extension}
If~$H$ is generated by~$k$ elements, we have
\[[H:\pi(H'')]\leq e(\#N)^k.\]
\end{corollary}
We used the following form of Goursat's Lemma.
\begin{lemma}[Goursat's Lemma]\label{Goursat}
Let~$U\leq G_1\times G_2$ be a subgroup, and~$U_1$, $U_2$ be its projections, and define~$U'_1=U\cap(G_1\times\{1\})$ 
and~$U'_2=U\cap(\{1\}\times G_2)$. Then~$(U_1\times\{1\})/U'_1$ and~$(\{1\}\times U_2)/U'_2$ are isomorphic, and hence
\[
\abs{(U_1\times\{1\})/U'_1}=\abs{(\{1\}\times U_2)/U'_2}.
\]
\end{lemma}
\section{Reductive norm estimates from residual stability}\label{sec:reductive}
%{\it  The following is the main new ingredient in this article compared to~\cite{RY}.
%%It relies on new type of tools, developped in~\S\ref{sec:buildings}.
%Sections~\S\ref{sec:reductive} to~\ref{sec:slopes} stem from ideas of~\cite{R-PhD} from the first author thesis.}

\subsection{Standing hypotheses}\label{standing hyp}
Let~$F\leq G\leq GL(d)$ be 
%connected? discuter partie finie: quotient par plat fini lisse: récrire avec~$F^0$.
reductive groups over~$\Q_p$. The ultrametric absolute value is denoted by~$\abs{~}:\CC_p\to \R_{\geq0}$ and the norm on~${\CC_p}^d$ is denoted by
\[
\norm{(v_i)_{i=1}^d}=\max\{\abs{v_1};\ldots;\abs{v_d}\}.
\]
The $\Q_p$-algebraic group~$GL(d)$ has a model~$GL(d)_{\Z_p}$, which induces models~$F_{\Z_p}$ and~$G_{\Z_p}$ over~$\Z_p$.
We denote by~$F_{\F_p}$ and~$G_{\F_p}$ their special fibres, which are algebraic groups over~$\F_p$.

We assume that, in the sense\footnote{An equivalent property is that~$F_{\F_p}$ and~$G_{\F_p}$
are connected reductive algebraic groups.} of~\cite[\S3.8]{Tits}
\begin{equation}\label{hyp:hyp}
\text{$F_{\Z_p}$ and~$G_{\Z_p}$ are ``hyperspecial''	.}
\end{equation}
\setcounter{secnumdepth}{4}
\subsubsection{Some consequences} We review some constructions  and some properties that hold under hypotheses~\eqref{hyp:hyp},
and will be needed later.

\paragraph{}
We consider a maximal torus~$T\leq G_{\ol{\Q_p}}$, a basis~$\Z^d\simeq X(T)$. We denote the set of weights
of the representation~$\rho:T\to G_{\ol{\Q_p}}\to GL(d)_{\ol{\Q_p}}$ by
\[
\Sigma(\rho)\subseteq X(T)
\]
and the weight decomposition of~$V={{\ol{\Q_p}}}^d$ under the action of~$T$, by
\begin{equation}\label{eigen decomp}
{{\ol{\Q_p}}}^d=\bigoplus_{\chi\in\Sigma(\rho)} V_\chi\text{ where }V_\chi:=\{v\in{{\ol{\Q_p}}}^d|\forall t\in T({\ol{\Q_p}}), t\cdot v=\chi(t)\cdot v\}.
\end{equation}
\paragraph{Remark}\label{rem set weights}
 For any other maximal torus~$T'$ there is a conjugation~$t\mapsto gtg^{-1}:T\mapsto T'$ in~$G({\ol{\Q_p}})$.
We deduce a set~$\Sigma(T')$ corresponding to~$\Sigma(T)$. The resulting set~$\Sigma(T)$ does not depend on the choice of
the conjugating element. The weight spaces~$V_\chi$ in the decomposition~\eqref{eigen decomp} depend on~$T$.
\paragraph{}\label{Good torus}
From\footnote{With~$\Omega=\{x_0\}$ if~$x_0\in\BT(G/L)$ is the fixed point of~$G_{\Z_p}(\Z_p)$.}~\cite[\S3.5]{Tits} we know that the induced model~$T_{\ol{\Z_p}}$ 
has good reduction, i.e.~$T_{\ol{\F_p}}$ is a torus, and that we have
\[
X(T)\simeq X(T_{\ol{\F_p}}).
\]
This also implies, c.f. e.g.~\cite[Prop.\,5]{Sesh} that~\eqref{eigen decomp} is compatible with integral structures:
\[
\ol{\Z_p}^d=\bigoplus \Lambda_{\chi}\text{ where }\Lambda_{\chi}:=\ol{\Z_p}^d\cap V_\chi;
\]
and that we have a corresponding weight decomposition
\[
\ol{\F_p}^d=\bigoplus \ol{V}_{\chi}\text{ where }\ol{V}_{\chi}:=\Lambda_{\chi}\tens\ol{\F_p}.
\]

\paragraph{}
There is a Cartan decomposition~\cite[4.4.3]{BT72} (see also\footnote{See~\cite[\S3 and \S3.3]{Tits} for assumptions of~\cite[3.3.3]{Tits}.}~\cite[3.3.3]{Tits}), for~$L/\Q_p$ a finite extension, and~$T_L$ a maximally\footnote{A maximal torus containing a maximal split torus.} split torus of~$G/L$,
%($T$ dépend de~$L$ À regarder)
\begin{equation}\label{Cartan}
G(L)=G_{\Z_p}(O_L)T_L(L)G_{\Z_p}(O_L).
\end{equation}
and consequently over~$\ol{\Q_p}$, when~$T$ is a maximal torus,
\begin{equation}\label{Cartanbar}
G(\ol{\Q_p})=G_{\Z_p}(\ol{\Z_p})T(\ol{\Q_p})G_{\Z_p}(\ol{\Z_p}).
\end{equation}

\subsection{Main statement} The following theorem can be seen as a refined more precise version of the functoriality
of local heights used in~\S\ref{every prime}. 
\begin{theorem}[Local relative stability estimates]\label{thm:compare reductive} Under hypotheses from~\S\ref{standing hyp},
let~$v,v'\in{\Z_p}^d$ be non zero vectors, denote by~$\ol{v},\ol{v'}\in {\FF_p}^d$ their reduction, and assume that
\begin{enumerate}
\item \label{test} the orbits~$G_{\Q_p}\cdot v,G_{\QQ_p}\cdot v'\subseteq \AAA_{\QQ_p}^d$ are closed subvarieties;
\item  the stabiliser groups~$F_v:=\Stab_G(v),F_{v'}:=\Stab_G(v')$ satisfy
\[
F_v=F_{v'}=F;
\]
\end{enumerate}
and that
\begin{enumerate}
\item the orbits~$G_{\FF_p}\cdot \ol{v},G_{\FF_p}\cdot \ol{v'}\subseteq \AAA_{\FF_p}^d$ are closed subvarieties;
\item  the stabiliser groups~$F_{\ol{v}}:=\Stab_G(v),F_{\ol{v'}}:=\Stab_G(v')$ satisfy, as group schemes\footnote{It amounts to the property  that~$F_{\ol{v}}$ and~$F_{\ol{v'}}$ are smooth.},
\begin{equation}\label{Hyp 51}
F_{\ol{v}}=F_{\ol{v'}}=F_{\FF_p}.
\end{equation}
\end{enumerate}
We define two functions~$G(\CC_p)\to \R$ given by
\[
H_{v}:g\mapsto \max\{1;\norm{g\cdot v}\}\text{ and }
H_{v'}:g\mapsto \max\{1;\norm{g\cdot v'}\}.
\]
Then the functions~$h_v=\log H_v$ and~$h_{v'}=\log H_{v'}$ satisfy
\begin{equation}\label{log equiv on reductive}
h_v\leq C\cdot h_{v'}\text{ and }h_{v'}\leq C\cdot h_v,
\end{equation}\label{reductive theorem final estimate}
in which~$C=C(\Sigma(\rho))$ depends only on the set of weights of~$\rho$ (cf.~\ref{rem set weights}).
\end{theorem}
\noindent In our proof, the quantity~$C(\Sigma)$ will depend upon the choice of
an invariant euclidean metric ``in the root system'' of~$G$, and there are canonical choices
of such metrics. The hypothesis~\eqref{Hyp 51} can be replaced by the weaker hypothesis 
in~\eqref{pKN flat}.

Several features that are important to our strategy. 
\begin{itemize}
\item The quantity~$C$ only depends on the weights of~$\rho$. Thus, when~$\rho$ comes from a representation
defined over~$\QQ$, this~$C$ does not depend on the prime~$p$.
\item The inequality does not need an additive constant: we have
\[
H_v\leq A\cdot {H_{v'}}^C
\]
with~$A=1$. Thus, when we multiply the inequalities over infinitely many primes, we don't accumulate an uncontrolled multiplicative factor~$\prod_p A(p)$.
\item The estimate~\eqref{log equiv on reductive} depends upon~$v$ only through its stabiliser group~$F$.
This is precisely information about the stabilisers that we deduce from Tate conjecture.
\end{itemize}
\subsection{Proof}We use the notions and notations of~\S6.4. 
Because~$G_{\Z_p}(\ol{\Z_p})\leq GL(d,\ol{\Z_p})$ acts isometrically on~$\ol{\Z_p}^d$, the functions~$h_v$
and~$h_{v'}$ are left~$G_{\Z_p}(\ol{\Z_p})$-invariant.
% We denote the quotient functions by
%\[
%h'_v,h'_{v'}:
%G_{\Z_p}(\ol{\Z_p})\sous G(\ol{\Q_p})\to \R.
%\]

Choose an arbitrary~$g\in G(\ol{\Q_p})$. It is sufficient to prove~\eqref{CCL proof KN} (see below) with this element~$g$,
as the other inequality in~\eqref{log equiv on reductive} can be deduced after swapping~$v$ and~$v'$.

Let~$T\leq G$ be a maximal torus defined over~$\ol{\Q_p}$. We endow~$A_T$, defined in~\eqref{defi appartment}, with a canonical euclidean 
distance~$d(~,~)=d_G(~,~)$, invariant under~$N_G(T)$ and depending only on~$G$ (using e.g.~\cite[LIE VI.12]{BBK}). We denote~$\Sigma(\rho)$
the set of weights of the action~$T\to G\xrightarrow{\rho} GL(n)$. We denote~$\gamma(\Sigma(\rho))$ the quantity from 
Prop.~\ref{coro slopes}, which does not depend on the maximal torus~$T$ up to conjugation, and only on the weights of~$\rho$.

Because~$G_{\Z_p}$ is hyperspecial,
there is a Cartan decomposition~\eqref{Cartan}. Thus there are some~$t'\in T(\ol{\Q_p})$ and~$k\in G_{\Z_p}(\ol{\Z_p})$
such that
\[
G_{\Z_p}(\ol{\Z_p})\cdot g=G_{\Z_p}(\ol{\Z_p})\cdot t
\] with~$t=k t' k^{-1}$. %We may? also assume~$k,t\in G_{\Z_p}(O_{\ol{\Q_p}})?$, that is~$k,t\in G_{\Z_p}(O_{K})$ for a finite Galois extension~$K/\Q_p$.

We may thus assume~$g=t$. We may write, as in~\eqref{hv to hmu}
\[
h_v\restriction_{T(\CC_p)}=h_{\mu}\circ a_T\text{ and }h_{v'}\restriction_{T(\CC_p)}=h_{\mu'}\circ a_T.
\]

According to Proposition~\ref{prop slopes comparison}, we have
\[
c(\Sigma(v))\cdot d(a,C_\mu)\leq h_{\mu}(a)\text{ and }h_{\mu'}(a)\leq c'(\Sigma(v'))\cdot d(a,C_{\mu'}).
\]
Thanks to hypotheses of~Theorem~\ref{thm:compare reductive} we may apply 
%\footnote{Here~\eqref{universal integral} would suffice to prove~$C_\mu=C_{\mu'}$.}
 Theorem~\ref{pKN} and~\eqref{gF fv} (we note that~$h_v(t)=0$  if and only if~$t\cdot v\in \ol{\Z_p}^d$). Thus,
\begin{multline*}
\set{t\in T(\ol{\Q_p})}{h_v(t)=0}
\\=\set{t\in T(\ol{\Q_p})}{tF\in (G/F)(\ol{\Z_p})}
\\=T(\ol{\Q_p})\cap (G(\ol{\Z_p})\cdot F(\ol{\Q_p}))
\\=\set{t\in T(\ol{\Q_p})}{h_{v'}(t)=0}.
\end{multline*}
Let~$C^\Q_{\mu}$ and~$C^\Q_{\mu'}$ be defined as in Lemma~\ref{C Q density}.
As the valuation group~$\Gamma(\ol{\Q_p})$ is~$\Q$, we deduce~$C^\Q_{\mu}=C^\Q_{\mu'}$,
and, by Lemma~\ref{C Q density}, we deduce
\[
C:=C_\mu=C_{\mu'}.
\]
Applying Corollary~\ref{coro slopes}, we conclude

%As~$0<c(\Sigma(v)),c'(\Sigma(v'))<\infty$, we deduce, with~$\frac{c'(\Sigma(v))}{c(\Sigma(v'))}\leq \gamma(\rho)^2\in\R_{>0}$, that 
%\[
%h_{\mu'}(a)\leq c'(\Sigma(v'))\cdot d(a,C)\leq \gamma(\rho)^2\cdot c(\Sigma(v))\cdot d(a,C)\leq \gamma(\rho)^2\cdot h_{\mu}(a).
%\]
%It follows
\begin{multline}\label{CCL proof KN}
h_{v'}(g)=h_{v'}(t)=h_{\mu'}\circ a_T(t)\\
\leq \gamma(\Sigma(\rho))\cdot h_\mu\circ a_T(t)=\gamma(\Sigma(\rho))\cdot h_{v}(t)=\gamma(\Sigma(\rho))\cdot h_{v}(g).
\end{multline}

\subsection{Norms on Toric orbits and the Apartment}\label{appartments}
For a torus~$T$ over an ultrametric extension~$L/\Q_p$, the associated ``apartment'' is defined as
\begin{equation}\label{defi appartment}
A_T=A_{T/L}=Y(T/L)\tens\R\simeq \Hom(X(T),\R)
\end{equation}
where~$Y(T)=Y(T/L):=\Hom(GL(1)_L,T)$ and~$X(T)=X(T/L):=\Hom(T,GL(1)_L)$ are the group of cocharacters and characters,
and are~$\Z$-linear dual to each other.

Then the pairing
\[
(t,\chi)\mapsto \log_p\abs{\chi(t)}:T(L)\times X(T)\to \R.
\]
induces a map
\begin{equation}\label{T to A}
a_T:T(L)\to A_T,
\end{equation}

Denote by~$\Z\leq \Gamma_L:=\log_p\abs{L^\times}\leq \R$ the valuation group of~$L$. 

When~$T$ has a model over~$L$ which is a torus~$T_{O_L}$ over~$O_L$, the map~$a_T$ factors as
\[
T(L)\twoheadrightarrow \frac{T(L)}{T_{O_L}(O_L)}\xrightarrow{\sim} Y(T)\tens\Gamma_L\hookrightarrow A_T.
\]

For a character~$\chi\in X(T)$ the function
\[
\log_p\abs{\chi}:T(L)\xrightarrow{\chi} L^\times\xrightarrow{\abs{~}} \R_{>0}\xrightarrow{\log_p} \R
\]
passes to the quotient to~$\frac{T(L)}{T_{O_L}(O_L)}$ and extends to a~$\R$-linear form which we
denote by
\[
\omega_\chi:A_T\to \R,
\] 
which is also the one deduced from~$A_T\simeq \Hom(X(T),\R)$.

Assume~$T\leq GL(n)$ is a torus over~$L$ with good reduction: denoting the eigenspace decomposition
of~$L^n$ for the action of~$T$ by
\[
L^n=\sum_{\chi \in X(T)} V_\chi
\]
we have (\ref{Good torus}, \cite[Prop.\,5]{Sesh})
\begin{equation}\label{integral eigen}
{O_L}^n=\sum_{\chi \in X(T)} V_\chi\cap {O_L}^n.
\end{equation}
It follows, denoting by~$\norm{~}$ the standard norm on~$L^n$, that, for~$v\in L^n$,
\begin{equation}\label{norm tore}
\norm{v}=\max\{0\}\cup\{\norm{v_\chi}~|~{\chi\in X(T)}\}.
\end{equation}
We denote by~$\Sigma(T)\subseteq X(T)$ the set of weights for the action of~$T$, and denote 
by
\[
\Sigma(v)=\{\chi\in X(T)~|~v_\chi\neq 0\}\subseteq \Sigma(T),
\]
and, if~$v\in{O_L}^n$, we define
\[
\ol{\Sigma}(v)=\{\chi\in X(T)~|~\norm{v_\chi}=1\}\subseteq \Sigma(v)
\]
and a function~$\mu:\Sigma(v)\to \R_{\leq0}$ given by
\[
\mu(\chi)=\log_p \norm{v_\chi}.
\]
The functions~$H_v:T(\ol{\Q_p})\to\R_{\geq 0}$ and~$h_v=\log(H_v)$
defined by
\[
H_v(t)=\max\{1;\norm{t\cdot v}\}
\]
can be computed from the formula
\begin{equation}\label{hv to hmu}
h_v=h_\mu\circ a_T\text{ with }h_\mu(a):=\max\{0\}\cup\set{\omega_\chi(a)+\mu(\chi)}{\chi \in \Sigma(v)}.
\end{equation}
\begin{lemma}\label{C Q density}
Define
\[
C_\mu=\set{a\in A_T}{h_{\mu}(a)=0}\text{ and }A_T^\Q=Y(T)\tens\Q\subseteq A_T.
\]
Then~$C_\mu^\Q:=C_\mu\cap A_T^\Q$ satisfies
\[
C_\mu=\ol{C_\mu^\Q}.
\]
\end{lemma}
%TODO
Lemma~\ref{C Q density} holds because the convex set~$C_\mu$ is constructed from affine forms~$\omega_{\chi}+\mu(\chi)$  on~$A_T$ 
which are \emph{defined over~$\QQ$}, with respect to the~$\QQ$-structure~$A_T^\QQ$.

\section{Residual stability and {$p$-adic} Kempf-Ness Theorem} \label{sec:pKN}

The estimates of Th.~\ref{thm:compare reductive} rely on the following result which we believe to be of independent interest. This is an analogue of  a theorem of Kempf-Ness~(\cite[Th. 0.1 b)]{KN}) in the context~\cite{B92} of~$p$-adic Mumford's stability. It relies on a careful analysis of the reduction of models of homogeneous spaces given by the invariant theory~\cite{Sesh} or of closed orbits in a linear representation.

\begin{theorem}[$p$-adic Kempf-Ness Theorem]\label{pKN}
Let~$F_{\Z_p}\leq G_{\Z_p}\leq GL(n)_{\Z_p}$ be smooth reductive group schemes, such that~$F_{\Z_p}\to G_{\Z_p}\to GL(n)_{\Z_p}$ are closed immersions, and~$G_{\Z_p}$ is connected.

Let~$v\in{\Z_p}^n$. Denote by~${\ol{v}\in\F_p}^n$ its reduction and assume that
\begin{equation}\label{pKN flat}
\Stab_{G_{\Q_p}}(v)=F_{\Q_p}\text{ and }
\dim
(\Stab_{G_{\F_p}}(\ol{v}))=\dim (F_{\F_p}),
\end{equation}
(using Krull dimensions) and assume that the orbits
\begin{equation}\label{pKN stab}
G_{\Q_p}\cdot v\subseteq \A^n_{\Q_p}\text{ and }
G_{\F_p}\cdot\ol{v}\subseteq \A^n_{\F_p}
\end{equation}
are closed.

Then, for all~$g\in G(\ol{\Q_p})$, we have, denoting by~$\Z_p[G/F]:=\Z_p[G]\cap \Q_p[G]^{F}$
the algebra of $F$-invariant functions~$G\to\A^1$ defined over~$\Z_p$,
\begin{equation}\label{universal integral}
g\cdot v \in \ol{\Z_p}^n~\text{ if and only if }~\forall f\in \Z_p[G/F], f(g)\in\ol{\Z_p}.
\end{equation}
Moreover,~$ \Spec(\Z_p[G/F])$ is smooth over~$\Z_p$, and we have
\begin{equation}\label{KN CCL}
(G(\ol{\Q_p})\cdot v)\cap\ol{\Z_p}^n=G(\ol{\Z_p})\cdot v.
\end{equation}
\end{theorem}

\subsubsection*{Remarks}
Some of the assumptions can be rephrased as follows. 

The~$\Q_p$-algebraic groups~$F$ and~$G$ are reductive, the compact subgroups
~$F_{\Z_p}(\Z_p)\leq F(\Q_p)$ and~$G_{\Z_p}(\Z_p)\leq G(\Q_p)$ are hyperspecial subgroups, and we have~$F_{\Z_p}(\Z_p)=F(\Q_p)\cap GL(n,\Z_p)$ and~$G_{\Z_p}(\Z_p)=G(\Q_p)\cap GL(n,\Z_p)$. 

The property~\eqref{pKN stab} is related to semi-stability and residual semi-stability of the vector~$v$ in the sense of~\cite{B92}.

In~\eqref{pKN flat}, the assumption on dimensions means that~$\Stab_{G_{\F_p}}(\ol{v})^{0,red}$ (the reduced subgroup of the neutral component)
is equal to~$(F_{\F_p})^0$. Equivalently~$\Stab_{G_{\F_p}}(\ol{v})^{0}(\ol{\F_p})=F^0(\ol{\F_p})$. 
This is implied by the stronger condition
\begin{equation}\label{red scheme identity}
\Stab_{G(\ol{\F_p})}(\ol{v})=F(\ol{\F_p})
\end{equation}
and the stronger one
\begin{equation}\label{Scheme stab identity}
\Stab_{G_{\F_p}}(\ol{v})=F_{\F_p}\text{ as group schemes.}
\end{equation}
  
%We note that~$\Stab_{G(\ol{\F_p})}(\ol{v})=F(\ol{\F_p})$
%is weaker than the identity, as group schemes,
%\begin{equation}\label{Scheme stab identity}
%\Stab_{G_{\F_p}}(\ol{v})=F_{\F_p}.
%\end{equation}
%The latter fails when~$\Stab_{G_{\F_p}}(\ol{v})$ is not a reduced
%scheme. It seems our argument only requires equality of Krull dimension,
%which is insentive to reducedness.

%It seems that~$\Z_p[G/F]$ is flat, and even smooth, in the context of~\cite{}.

\subsubsection*{Proof of Theorem~\ref{pKN}} Let~$G/F=\Spec(\Z_p[G/F])$ and let~$S=\ol{G_{\Q_p}\cdot v}^{Zar(\A^n_{\Z_p})}$ be the Zariski closure of the orbit~$G_{\Q_p}\cdot v\subseteq \A^n_{\Q_p}$ in the affine scheme~$ \A^n_{\Z_p}$. We denote by~$gF\in(G/F)(\ol{\Q_p})$ the image of~$g\in G(\ol{\Q_p})$.

%A reformulation of~\eqref{universal integral} is: denoting by~$gF\in(G/F)(\ol{\Q_p})$ the image of~$g\in G(\ol{\Q_p})$, we have
%\begin{equation}\label{universal integral bis}
%g\cdot v \in \ol{\Z_p}^n~\text{ if and only if }~gF\in(G/F)(\ol{\Z_p}).
%\end{equation}

Then, for~$g\in G(\ol{\Q_p})$, we have~$g\cdot v\in \ol{\Z_p}^n\Leftrightarrow g\cdot v\in S(\ol{\Z_p})$, and we have~$gF\in (G/F)(\ol{\Z_p})\Leftrightarrow\forall f\in \Z_p[G/F], f(g)\in\ol{\Z_p}$. Thus~\eqref{universal integral} is equivalent to
\begin{equation}\label{gF fv}
gF\in (G/F)(\ol{\Z_p})\Leftrightarrow g\cdot v\in S(\ol{\Z_p}).
\end{equation}

Since~$G_{\Z_p}\leq GL(n)_{\Z_p}$ is a smooth reductive subgroup scheme, we may apply Lem.~\ref{lemma KN connected}. Using the assumptions~\eqref{pKN flat} and ~\eqref{pKN stab}, we may apply Lem.~\ref{Lemma S S'}.
Since~$F_{\Z_p}\leq G_{\Z_p}$ is a smooth reductive subgroup scheme, we may apply Lem.~\ref{Seshadri} and~\ref{Lemma integral}. Thus,  the map~$G/F\to S$ is integral. We may thus apply Prop.~\ref{prop:integral lift}, and \eqref{gF fv} follows. This proves~\eqref{universal integral}.

According to Lem.~\ref{platitude}, the scheme~$G/F$ is smooth over~$\Z_p$.

We prove~\eqref{KN CCL} by double inclusion. We trivially have~$G(\ol{\Z_p})\cdot v\subseteq (G(\ol{\Q_p})\cdot v)\cap \ol{\Z_p}^n$. 
We prove the other inclusion. Let~$g\in G(\ol{\Q_p})$ be such that~$g\cdot v \in S(\ol{\Z_p})$.
By~\eqref{universal integral}, we have~$gF\in (G/F)(\ol{\Z_p})$. By~Prop.~\ref{prop:Sesh GIT}, the map~$\pi:G\to G/F$ satisfies~$\pi(G(\ol{\F_p}))=(G/F)(\ol{\F_p})$. As the morphism~$\pi:G\to G/F$ is smooth, it is in particular flat. We may thus apply Prop.~\ref{flat lift 2} with~$y=gF$ and~$\ol{y}=\ol{gF}\in (G/F)(\ol{\F_p})\subseteq \pi(G(\ol{\F_p}))$. Thus there exists~$g'\in G(\ol{\Z_p})$ such that~$\pi(g')=gF$.
Thus~$g\cdot v=g'\cdot v\in G(\ol{\Z_p})\cdot v$. This proves the other inclusion.

This concludes the proof of Theorem~\ref{pKN}.

\subsection{Summary of the section} In the rest of~\S\ref{sec:pKN}, we develop results used in the proof above.

In~\S\ref{sec:orbit varieties}, we consider a smooth reductive group scheme~$G_{\Z_p}\leq GL(n)_{\Z_p}$ and we study the Zariski closure~$S=\ol{G_{\Q_p}\cdot v}^{Zar(\A^n_{\Z_p})}$ of an orbit of of~$G_{\Q_p}$.

We prove in Lemma~\ref{lemma KN connected} that the special fibre~$S_{\F_p}$ is connected by reduction to the case~$G=GL(1)$. Under an extra assumption of the form~\eqref{pKN flat} and~\eqref{pKN stab}, we prove in Lemma~\ref{Lemma S S'} that the reduced subscheme~$(S_{\F_p})^{red.}$  is a single~$G_{\F_p}$-orbit.

In~\S\ref{secGIT} we consider smooth reductive group schemes~$F_{\Z_p}\leq G_{\Z_p}\leq GL(n)_{\Z_p}$, and study the scheme~$G/F=\Spec(\Z_p[G/F])$. We prove that~$(G/F)_{\F_p}$ is reduced and is a single~$G_{\F_p}$-orbit. We use results from Seshadri's~\cite{Sesh}, we prove that~$G/F$ can be written in the form~$S$ as in~\S\ref{sec:orbit varieties}, and we prove the assumptions of Lem.~\ref{Lemma S S'} are satisfied.

In~\S\ref{sec:integral}, we consider~$S$ as in~Lem.~\ref{Lemma S S'}, and~$G/F$ as in~\ref{secGIT}. We prove that~$G/F$ is the normalisation of~$S$. From~\S\ref{sec:orbit varieties} and~\S\ref{secGIT}, the morphism~$G/F\to S$ is quasi-finite, and, by Zariski's main theorem,~$G/F$ is open in the normalisation~$\wt{S}$ of~$S$ in~$G/F$. We prove that~$G/F\to\wt{S}$ is surjective by applying Lem.~\ref{Lemma S S'} to~$\wt{S}$.

In~\S\ref{sec:flat smooth}, we prove that the morphisms~$G\to G/F$ and~$G/F\to \Spec(\Z_p)$ are smooth. This uses the  "critère de lissité fibre par fibre" and, for the~$\F_p$-fibre, uses~\S\ref{secGIT}.

In~\S\ref{sec:flat recalls} and~\S\ref{sec:integral lift}, we recall some consequences of flatness and integrality.
The \S\S\ref{sec:flat recalls}-\ref{sec:integral lift} do not depend on the rest of \S\ref{sec:pKN}.

\subsection{Orbit closure over~$\Z_p$}\label{sec:orbit varieties}
%\subsubsection{}\label{notations hyperspecial}
%Let~$G_{\Z_p}\leq GL(n)_{\Z_p}$ be a closed smooth reductive subgroup scheme. We denote by~$G_{\Q_p}$ its generic fibre and we denote by~$\Q_p[G]$ the algebra of functions~$G_{\Q_p}\to\A^1_{\Q_p}$.
%
%%The group~$G(\Q_p)$  acts via the regular representation. 
%%By~\cite[Prop. 2.3.6]{SpringerLAG}, the vector space~$\Q_p[G]$ is a union of finite dimensional linear representations of~$G_{\Q_p}$.
%
%As~$G_{\Z_p}$ is flat, we have~$G_{\Z_p}=\Spec(\Z_p[G])=\ol{G_{\Q_p}}^{Zar(GL(n)_{\Z_p})}$, with~$\Z_p[G]\subseteq \Q_p[G]$ the image of~$\Z_p[GL(n)]\to \Q_p[G]$.
%%, that is,~$G_{\Z_p}:=\Spec(\Z_p[G])$. is the Zariski closure~$\ol{G}^{Zar(GL(n)_{\Z_p})}$ of~$G$ in~$GL(n)_{\Z_p}$. 
%
%%We say that a reductive~$\Q_p$-algebraic subgroup~$G\leq GL(n)$ is hyperspecial when the induced model~$G_{\Z_p}$ is a smooth reductive group scheme over~$\Z_p$.
%\subsubsection{}
\begin{proposition}\label{prop:Aorbite}
Let~$G_{\Z_p}\leq GL(n)_{\Z_p}$ be a closed smooth reductive subgroup scheme.%Let~$G\leq GL(n)$ be hyperspecial group over~$\Q_p$. 

Let~$A\subseteq \Z_p[G]$ be a~$G(\Z_p)$-stable subalgebra of
finite type over~$\Z_p$.

Then there exists a closed immersion~$G_{\Z_p}\to GL(m)_{\Z_p}$ and a vector~$\lambda\in {\Z_p}^m$
and a~$G_{\Z_p}$-equivariant isomorphism
\[
\Spec(A)\to \ol{G_{\Q_p}\cdot \lambda}^{Zar(\A^m_{\Z_p})}.
\]
\end{proposition}
\begin{proof}As~$A$ is of finite type over~$\Z_p$, there exists a finite generating family~$a_1,\ldots,a_N\in A$.
As~$G_{\Z_p}$ is flat over~$\Z_p$, we have~$\Z_p[G]\subseteq \Q_p[G]$.
By~\cite[Prop. 2.3.6]{SpringerLAG}, there exists a finite dimensional~$G(\Q_p)$-stable subspace~$V\leq \Q_p[G]$ containing~$\{a_1;\ldots;a_N\}$. As~$A\tens\Q_p\leq \Q_p[G]$ is~$G(\Q_p)$-invariant, we may replace~$V$ by~$V\cap A\tens\Q_p$. Then~$\Lambda:=A\cap V$ is a~$\Z_p$-lattice in~$V$. As both~$A$ and~$V$ are~$G(\Z_p)$-stable,~$\Lambda$ is~$G(\Z_p)$-stable.
There exists a finite $\Z_p$-linear basis~$v_1,\ldots,v_d\in \Lambda$. As~$a_1,\ldots,a_N\in A\cap V$, the family~$v_1,\ldots,v_d\in \Lambda$ generates the algebra~$A$. Let~$\Z_p[X_1,\ldots,X_d]$ denote the polynomial algebra.
The action of~$G(\Z_p)$ on~$\Lambda$ induces an action of~$G_{\Z_p}$ on~$\Spec(\Z_p[X_1,\ldots,X_d])$. The morphism
\[
X_i\mapsto v_i:
\Z_p[X_1,\ldots,X_d]\to A
\]
is surjective. It corresponds to a closed embedding~$\Spec(A)\to \A^d_{\Z_p}$. By construction, this embedding is~$G_{\Z_p}$ equivariant. 

As~$A\subseteq \Q_p[G]$, the corresponding morphism~$G_{\Q_p}\to \Spec(A)$ is dominant. Let~$a\in \Spec(A)$ be the image of~$1\in G(\Q_p)$.  It follows that~$\Spec(A)$ is the Zariski closure of the orbit of~$G(\Q_p)\cdot a$. As~$\Spec(A)$ is a closed subscheme of~$\A^d_{\Z_p}$, we have~$\Spec(A)=\ol{G_{\Q_p}\cdot a}^{Zar(\A^d_{\Z_p})}$.

The direct sum of the action of~$G_{\Z_p}$ on~$\A^d_{\Z_p}$ with the representation~$G_{\Z_p}\to GL(n)_{\Z_p}$   is a closed immersion~$G_{\Z_p}\to GL(d+n)_{\Z_p}$. The proposition follows for~$m=d+n$, and for~$\lambda=a\oplus 0$.
\end{proof}

\subsubsection{Connectedness and Special fibre}
\begin{lemma}\label{lemma KN connected}
Let~$G_{\Z_p}\leq GL(n)_{\Z_p}$ be a closed smooth reductive subgroup scheme, %Let~$G\leq GL(n)$ be hyperspecial group over~$\Q_p$,
 let~$v\in \ol{\Z_p}^n$
 and let
\begin{equation}\label{defi schematic closure S}
S=\ol{G_{\ol{\Q_p}}\cdot v}^{Zar(\A^n_{\ol{\Z_p}})}
\end{equation}
be the schematic closure of~${G_{\ol{\Q_p}}\cdot v}\subseteq \A^n_{\ol{\Z_p}}$ in~$\A^n_{\ol{\Z_p}}$.
Then~$S_{\ol{\F_p}}$ is connected.
\end{lemma}
We first treat the case~$T=G= GL(1)$.
\begin{proof}If~$S_{\ol{\F_p}}$ is a closed orbit of~$T$, then it is connected, as it is the image of~$GL(1)$, which is connected. We will show that, otherwise, we can decompose~$S_{\ol{\F_p}}$ under the form
\begin{equation}\label{decompo tore en + 0 -}
S_{\ol{\F_p}}(\ol{\F_p})=S^-\cup\{\ol{v_0}\}\cup S^+
\end{equation}
where each of~$S^-$ and~$S^+$ is either empty or of the form~$X=T(\ol{\F_p})\cdot\ol{w}$
with~$\{\ol{v_0}\}\in \ol{X}^{Zar}$. For every~$\ol{w}\in\ol{\F_p}^n$, because~$T=GL(1)$ is connected, so is~$T\cdot \ol{w}$, and so is its
Zariski closure. It follows that~$S^-$ and~$S^+$ are contained in the connected component of~$\ol{v_0}$, and finally that~$S_{\ol{\F_p}}$ is
the connected component of~$\ol{v_0}$.

By~Prop.~\ref{flat lift}, a point in~$S(\ol{\F_p})$ is of the form
\[
\ol{x}\text{ with }x=t\cdot v\in\ol{\Z_p}^n\text{ and }t\in T(\ol{\Q_p}).
\]

We identify~$X(T):=\Hom(GL(1),GL(1))$ with~$\Z$ and denote by
\[
v=\sum_{k\in\Z} v_k
\]
the eigendecomposition of~$v$ for the action of~$T$. Then~$x=t\cdot v=\sum_{k\in \Z} t^k\cdot v_k$,
and, by~\eqref{norm tore}, 
\begin{equation}\label{eq t x entier}
\norm{x}=\max\{\abs{t}^k\cdot \norm{v_k}\}\leq 1.
\end{equation}
Define
\[
c=\max_{k< 0} \norm{v_k}^{-1/k}\in[0;1]\text{ and }
c'=\min_{k> 0} \norm{v_k}^{-1/k}\in[1;+\infty].
\]
(Observe that, if~$v_k=0$ for all~$k<0$, then~$c=0$, and, if~$v_k=0$ for all~$k>0$, then~$c'=+\infty$.)

For~$t\in T(\ol{\Q_p})$ we have~$t\cdot v\in\ol{\Z_p}^n$ if and only if~$c\leq \abs{t}\leq c'$.
We define
\begin{eqnarray*}
T_-&=&\set{t\in T(\ol{\Q_p})}{\abs{t}=c}\\
T_0&=&\set{t\in T(\ol{\Q_p})}{c<\abs{t}<c'}\\
T_+&=&\set{t\in T(\ol{\Q_p})}{\abs{t}=c'}.
\end{eqnarray*}
and
\[
S^-=\set{\ol{t\cdot v}}{t\in T_-}\text{ and }
S^+=\set{\ol{t\cdot v}}{t\in T_+}.
\]
and
\[
v_-=\sum_{k< 0} v_k\text{ and }v_+=\sum_{k> 0} v_k.
\]
If~$c=0$, then~$T_-=S^-=\emptyset$. Otherwise, let us pick~$u\in T_-$.
% and~$u'\in T_+$.
Assume first~$c\neq c'$. Thus~$c<c'$, and~$u\cdot v^+=v^+=0$, or else~$c'<+\infty$, and, for some~$u'\in T_+\neq \emptyset$, we have
\[
\norm{u\cdot v_+}=\max_{k>0}{c^k}\cdot \norm{ v_k}<
\max_{k\in \Z}{c'^k}\cdot \norm{ v_k}=\norm{u'\cdot v}\leq 1.
\]
We then have~$\ol{u\cdot v_+}=0$ and
\[
\ol{w}:=\ol{u\cdot v}=\ol{u\cdot v_-+v_0}
\]
We then have
\[
S^-=T(\ol{\F_p})\cdot \ol{w}
\]
Because the weights of~$\ol{u\cdot v_-}$ are negative we have 
\[
\lim_{\ol{t}\to +\infty} \ol{t}\cdot\ol{u\cdot v_-+v_0}= \ol{u\cdot 0+v_0}=\ol{v_0},
\]
where limits are understood in the sense of the Hilbert-Mumford criterion, as in~\cite[Lem.\,1.3]{Kempf}.

Thus~$\{v_0\}\in \ol{S^-}^{Zar}$. The case of~$S^+$ is treated similarly
and we have obtained~\eqref{decompo tore en + 0 -} with the desired properties.

We now treat the remaining case~$c=c'$. We then have
\[
S_{\ol{\F_p}}(\ol{\F_p})=S^+=S^-=T(\ol{\F_p})\cdot\ol{v}.
\]
(This is then a closed orbit of~$T_{\ol{\F_p}}$, as~$S_{\ol{\F_p}}$ is closed).
\end{proof}
We reduce Lemma~\ref{lemma KN connected} to the case of a torus~$GL(1)\simeq T\leq G$.
\begin{proof}It is enough to prove that for an arbitrary~$\ol{x}\in S(\ol{\F_p})$,
this~$\ol{x}$ and~$\ol{v}$ belong to the same connected component of~$S_{\ol{\F_p}}$.

We may find~$x\in S(\ol{\Q_p})\cap\ol{\Z_p}^n$ with reduction~$\ol{x}$.
There exists~$g\in G(\ol{\Q_p})$ with~$g\cdot v=x$. From a Cartan decomposition~$g=k_1t'k_2\in G(\ol{\Z_p})\cdot T'(\ol{\Q_p})\cdot G(\ol{\Z_p})$ as in~\eqref{Cartanbar}, we construct~$k:=k_1k_2\in G(\ol{\Z_p})$ and a
maximal torus~$T:={k_2}^{-1} T' k_2\leq G$ and~$t={k_2}^{-1} t' k_2\in T(\ol{\Q_p})$ with~$k\cdot t=g$.
The torus~$T$ has good reduction by~\ref{Good torus}. There exists~$y:GL(1)\to T$ defined over~$\ol{\Q_p}$ and~$u\in T(\ol{\Z_p})$
and~$\lambda\in \ol{\Q_p}^\times$ with~$y(\lambda)=u\cdot t$. 
Because~$G_{\F_p}$ is connected and~${S}_{\ol{\F_p}}$ is~$G_{\ol{\F_p}}$-invariant, the orbit~$G_{\F_p}(\ol{\F_p})\cdot\ol{x}$ is connected and contained in~$S_{\ol{\F_p}}$. Thus~$\ol{x}$ and~$\ol{x'}={\ol{(k\cdot u)}^{\phantom{l}}}^{-1}\cdot \ol{x}$
lie in the same connected component of~$S_{\ol{\F_p}}$. We may thus replace~$\ol{x}$ by~$\ol{x'}$ 
and~$x$ by~$k\cdot u^{-1}\cdot x$, and~$g$ by~$y(\lambda)$. 

We have~$x\in GL(1)(\ol{\Q_p})\cdot v\cap \ol{\Z_p}^n$ and thus
\[
\ol{v},\ol{x}\in S_T(\ol{\F_p})\text{ with }S_T:=\ol{T\cdot v}^{Zar(\A^n_{\ol{\Z_p}})}\subseteq S.
\]
From the previous~$GL(1)$ case,~$S_T$ is connected. Thus~$\ol{x}$ and~$\ol{v}$ lie in the same connected component.
\end{proof}
\begin{lemma}\label{Lemma S S'} In the situation of Lemma~\ref{lemma KN connected},
we assume that
\begin{equation}\label{S S' orbite fermee}
\text{ the orbit $G_{\F_p}\cdot \ol{v}$ is Zariski closed in~$\A^n_{\F_p}$}
\end{equation}
and denote by~$S'$ the corresponding reduced subscheme of~$\A^n_{\F_p}$. 
We assume furthermore that, using Krull dimension,
\begin{equation}\label{S S' dimension}
\dim \Stab_{G_{\Q_p}}(v)=\dim \Stab_{G_{\F_p}}(\ol{v})
\end{equation}
Then
\[
S'=(S_{\F_p})^{red.}.
\]
%and furthermore
%\[
%\text{$S_{\F_p}$ is a reduced scheme(TODO-BYPASSED).}
%\]
\end{lemma}
\begin{proof}By construction~$S$ is the Zariski closure of its generic fiber. Hence, by Prop.~\ref{flat lift}, it is flat over~$\Z_p$. According to Lemma~\ref{lemma schemes} we have
\begin{equation}\label{semi dim S}
\dim S_{\F_p}\leq \dim S_{\Q_p}
\end{equation}
From~\eqref{dim formula}, we have
\begin{eqnarray}
\dim S_{\Q_p} &= &\dim G_{\Q_p} - \dim \Stab_{G_{\Q_p}}(v),\\
\dim S'_{\phantom{\Q_p}} &= &\dim G_{\F_p} - \dim \Stab_{G_{\F_p}}(\ol{v}).
\end{eqnarray}
We deduce~$\dim(S')\geq \dim(S_{\F_p})$. Because~$S'\subseteq S_{\F_p}$
we have actually
\[
\dim(S')=\dim(S_{\F_p}).
\]
Thus,~$S'$ contains a generic point of one 
irreducible component of~$S$, and thus\footnote{This~$S'$ is of finite type over~$S$ 
and its image will be constructible.} contains a non empty open subset 
of~$S$. Because~$S'$ is closed 
in~$\A^n_{\F_p}$, it is closed in~$S_{\F_p}$. Thus~$S'$ contains a connected
component of~$S_{\F_p}$, and because~$S_{\F_p}$ is connected and~$S'$ 
is reduced,
\[
S'=(S_{\F_p})^{red}.\qedhere
\]
%It suffice to prove that~$S_{\F_p}$ is reduced.
\end{proof}
\begin{corollary}\label{cor:T S}Let~$S$ be as in Lemma~\ref{Lemma S S'}.
 Let~$T$ be a~$\F_p$-algebraic variety on which~$G_{\F_p}$ acts,
and let
\[
T\to S_{\F_p}
\]
be a~$G_{\F_p}$-equivariant morphism. Then, for every~$t\in T(\ol{\F_p})$,
\[
\dim \Stab_{G_{\ol{\F_p}}}(t)\leq \dim \Stab_{G_{\F_p}}(\ol{v})
\]
\end{corollary}
\begin{proof}Let~$s\in S(\ol{\F_p})$ be the image of~$t$. We have~$\Stab_{G_{\ol{\F_p}}}(t)\leq \Stab_{G_{\ol{\F_p}}}(s)$. By Lemma~\ref{Lemma S S'}, we have~$G(\ol{\F_p})\cdot \ol{v}=S(\ol{\F_p})$.
Thus there exists~$g\in G(\ol{\F_p})$ such that~$g\cdot \ol{v}=s$. Thus~$\Stab_{G_{\ol{\F_p}}}(s)=g\Stab_{G_{\ol{\F_p}}}(\ol{v})g^{-1}$. We deduce
\[
\dim\Stab_{G_{\ol{\F_p}}}(t)\leq \dim\Stab_{G_{\ol{\F_p}}}(s)=\dim\Stab_{G_{\ol{\F_p}}}(\ol{v})=\dim\Stab_{G_{\F_p}}(\ol{v}).\qedhere
\]
\end{proof}

\subsection{Reductive GIT quotient over~$\Z_p$}\label{secGIT}
Let~$F_{\Z_p}\leq G_{\Z_p}\leq GL(n)_{\Z_p}$ be smooth reductive closed subgroup schemes. % in the sense of~\ref{notations hyperspecial}. 
 We define, following~\cite{Sesh},
\begin{equation}\label{defi G sur F}
\Z_p[G/F]=\Z_p[G]\cap \Q_p[G]^{F(\Q_p)}\text{ and }G/F:= \Spec(\Z_p[G/F]).
\end{equation}

We have~$\F_p[G_{\F_p}]\simeq\Z_p[G_{\Z_p}]\tens\F_p$. Thus there exists an homomorphism~$\Z_p[G_{\Z_p}]^{F}\tens\F_p\to\F_p[G_{\F_p}]^{F_{\F_p}}$, and hence a morphism
\begin{equation}\label{G/F Fp to G/F}
G_{\F_p}/F_{\F_p}\to (G/F)_{\F_p}.
\end{equation}
\begin{lemma}\label{Seshadri}
Let~$F_{\Z_p}\leq G_{\Z_p}\leq GL(n)_{\Z_p}$ be smooth reductive closed subgroup schemes.
%Denote by~$((G/F)_{\F_p})^{red}\subseteq (G/F)_{\F_p}$ the reduced subscheme.

Then~$(G/F)_{\F_p}$ is reduced: we have
\begin{equation}\label{proof:G/F reduced}
(G/F)_{\F_p}=((G/F)_{\F_p})^{red}.
\end{equation}

The map
\[
G_{\F_p}/F_{\F_p}\to G/F
\]
from~\eqref{G/F Fp to G/F} induces an isomorphism 
\begin{equation}\label{Sesh iso}
G_{\F_p}/F_{\F_p}\simeq ((G/F)_{\F_p})^{red}= (G/F)_{\F_p}.
%G_{\F_p}/F_{\F_p}\simeq ((G/F)_{\F_p})^{red}.
\end{equation}
\end{lemma}
We prove~\eqref{proof:G/F reduced} in~\ref{sec:G/F reduced}, and prove~\eqref{Sesh iso} in~\ref{sec:G/F iso} .
\subsubsection{Proof of~\eqref{proof:G/F reduced} of Lemma~\ref{Seshadri}}\label{sec:G/F reduced}
  We deduce~\eqref{proof:G/F reduced} from Corollary~\ref{cor:G/F reduced} below.
\begin{lemma}\label{lem:G/F reduced}
Let~$A$ be a flat~$\Z_p$-algebra such that~$A\tens\F_p$ is reduced. Let~$\Gamma$ be a group of automorphisms of the~$\ol{\Q_p}$-algebra~$A\tens \ol{\Q_p}$. Then~$(A^\Gamma)\tens \F_p$ is reduced.
\end{lemma}
\begin{proof}[Proof of Lemma~\ref{lem:G/F reduced}]
Let~$a\in pA\cap A^\Gamma$. As~$A$ is flat over~$\Z_p$, there exists a unique~$b$ such that~$p\cdot b=a$. For~$\gamma\in \Gamma$, we have~$p\cdot \gamma(b)=\gamma(a)=a$. Thus~$b=\gamma(b)$.
Thus~$b$ is~$\Gamma$-invariant. We deduce that~$pA\cap A^\Gamma\subseteq pA^\Gamma$. Equivalently the map
\[
A^\Gamma/pA^\Gamma\to A/pA
\]
is injective. Thus~$A^\Gamma\tens \F_p$ is isomorphic to a subalgebra of the reduced algebra~$A\tens\F_p$. This implies that~$A^\Gamma\tens \F_p$ is reduced.
\end{proof}
\begin{corollary}\label{cor:G/F reduced}
Let~$G$ be a smooth linear group scheme over~$\Z_p$, and let~$F\leq G$ be a subgroup scheme.
Then~$\Z_p[G]^F\tens\F_p$ is a reduced algebra.
\end{corollary}
\begin{proof}[Proof of Corollary~\ref{cor:G/F reduced}]
We apply Lem.~\ref{lem:G/F reduced} with~$A=\Z_p[G]$ and~$\Gamma=F(\ol{\Q_p})$.
\end{proof}

\subsubsection{Results from Geometric invariant theory over~$\Z_p$}
Before proving~\eqref{Sesh iso} of Lem.~\ref{Seshadri}, we recall some results from~\cite{Sesh}.
\begin{proposition}\label{prop:Sesh GIT} 
Let~$F_{\Z_p}\leq G_{\Z_p}\leq GL(n)_{\Z_p}$ be smooth reductive closed subgroup schemes.

Then, for every algebraically closed extension~$k$ of~$\Q_p$ or~$\F_p$, the map
\[
G_{\Z_p}\to G/F=\Spec(\Z_p[G]^F)
\]
induces bijections
\[
G(k)/F(k)\to (G/F)(k).
\]
\end{proposition}
This is an application of~\cite[Prop.\,6, \S{II.1}, Prop.\,9 Cor.\,2 (i) \S{II.3}]{Sesh} with~$X=G$ and~$Y=G/F$. 
 In our case, every geometric point of~$X$ is ``stable'' in the sense of~\cite[\S{II}.1, Def.~1]{Sesh}, every
geometric orbit is closed, and the closures of two distinct orbits have an empty intersection.

%\begin{proof}
%We apply~\cite[Prop. 6, \S{II.1}, p.~252]{Sesh} (where~$X/G$ of~\cite{Sesh} is our~$G/F$) with the~$F$-invariant closed
%\[
%X:=G\subseteq GL(n)\subseteq SL(n+1)\subseteq  V:=\A^{(n+1)^2},
%\]
%which is flat over~$\Z_p$, for the action of~$F$ on the right. In our case, every geometric point of~$X$ is ``stable'' in the sense of~\cite[\S{II}.1, Def.~1]{Sesh}, every
%geometric orbit is closed, and the closures of two distinct orbits have an empty intersection. 
%
%With~$B\subseteq \Z_p[X]^F$ the image of~$\Z_p[\A^{(n+1)^2}]^F$ in~$\Z_p[X]$, we have morphisms of schemes
%\[
% \Spec(\Z_p[X])\to X/F:= \Spec(\Z_p[X]^F)\to T:= \Spec(B).
%\]
%From~\cite[Prop.\,6, \S{II.1}, Prop.\,9 Cor.\,2 \S{II.3}]{Sesh} we have, on geometric points in an algebraically closed field~$k$ over~$\F_p$ or~$\Q_p$, 
%\[
%X(k)\to X(k)/F(k)\simeq T(k).
%\]
%This proves the proposition.
%\end{proof}
%In our terms, this implies that we have bijections
%\[
%(G/F)(\ol{\F_p})\to G(\ol{\F_p})/F(\ol{\F_p})\to T(\ol{\F_p}).
%\]
%It follows that the~$G_{\F_p}$-equivariant map
%\[
%(G_{\F_p}/F_{\F_p})\to (G/F)_{\F_p}
%\]
%induces bijections
%\[
%(G_{\F_p}/F_{\F_p})(\ol{\F_p})\simeq G(\ol{\F_p})/F(\ol{\F_p})\simeq (G/F)_{\F_p}(\ol{\F_p}).
%\]
%
%In particular the map~$G_{\F_p}/F_{\F_p}\to (G/F)_{\F_p}$ is surjective on geometric points and is quasi-finite

\subsubsection{Proof of~\eqref{Sesh iso} of Lemma~\ref{Seshadri}}\label{sec:G/F iso}
Before proving~\eqref{Sesh iso} of Lemma~\ref{Seshadri}, let us  recall some facts. 

We have~$\Z_p[G/F]\tens\Q_p=\Q_p[G]^F$. Thus
\[
(G/F)_{\Q_p}=G_{\Q_p}/F_{\Q_p}
\]
and, by Lem.~\ref{Flat orbit lemma},
\begin{equation}\label{G/F Qp dim}
\dim(G/F)_{\Q_p}=\dim G_{\Q_p} - \dim F_{\Q_p}.
\end{equation}
As~$\Z_p[G/F]=\Z_p[G]^F\subseteq \Q_p[G]$ has no torsion,~$G/F$ is flat over~$\Z_p$, 
and, by~\citestacks[Lemma 37.30.4.]{0D4H},
\begin{equation}\label{G/F Qp dim bis}
\dim((G/F)_{\F_p})\leq \dim((G/F)_{\Q_p})=\dim(G_{\Q_p})-\dim(F_{\Q_p}).
\end{equation}

By~\citestacks[Prop. 10.162.16.]{0335},
we may apply~\cite[\S{II}.4, Th.~2]{Sesh}. We have
\begin{equation}\label{G/F type fini}
\text{ $\Z_p[G/F]$ is of finite type over~$\Z_p$. }
\end{equation}

\begin{proof}[Proof of Lemma~\ref{Seshadri}]

By~\eqref{G/F type fini}, we may apply Prop.~\ref{prop:Aorbite} to the algebra~$A:=\Z_p[G]^F\subseteq \Z_p[G]$,
and write
\[
\Spec(A)\simeq \ol{G\cdot \lambda}^{Zar(\A^m_{\Z_p})}.
\]

In order to apply Lem.~\ref{Lemma S S'} to~$\ol{G\cdot \lambda}^{Zar(\A^m_{\Z_p})}$,  we prove the assumptions~\eqref{S S' orbite fermee} and~\eqref{S S' dimension}.

By Prop.~\ref{prop:Sesh GIT} the map~\eqref{G/F Fp to G/F}
is surjective (on geometric points). 

This implies that
\[
G(\ol{\F_p})\to (G_{\F_p}/F_{\F_p})(\ol{\F_p}) \to (G/F)(\ol{\F_p})
\]
is surjective. Thus~$(G/F)_{\F_p}(\ol{\F_p})$ is a single $G(\ol{\F_p})$-orbit. 
By construction~$(G/F)_{\F_p}\subseteq \A^m_{\F_p}$ is closed. 
Thus~$G(\ol{\F_p})\cdot\ol{\lambda}$ is a closed orbit. This proves~\eqref{S S' orbite fermee}.

We thus have
\begin{equation}\label{eq:x}
\dim (G/F)_{\F_p}=\dim G_{\F_p}\cdot \ol{\lambda}=\dim G_{\F_p}-\dim \Stab_{G_{\F_p}}(\ol{\lambda}).
\end{equation}

By Prop.~\ref{prop:Sesh GIT} the map~\eqref{G/F Fp to G/F}
is injective (on geometric points) and thus quasi-finite. 
This implies
\[
\dim((G/F)_{\F_p})\geq 
\dim(G_{\F_p}/F_{\F_p})=\dim(G_{\F_p})-\dim(F_{\F_p})=\dim(G_{\Q_p})-\dim(F_{\Q_p}).
\]
Together with~\eqref{G/F Qp dim bis}, this implies~$\dim((G/F)_{\F_p})=\dim(G_{\Q_p})-\dim(F_{\Q_p})$.
By~\eqref{eq:x}, we have
\[\dim \Stab_{G_{\F_p}}(\ol{\lambda})=\dim(F_{\F_p})=\dim(F_{\Q_p}).\]
This implies~\eqref{S S' dimension}.

We may thus apply  Lem.~\ref{Lemma S S'}. This implies that
\[
(G/F)_{\F_p}=((G/F)_{\F_p})^{red.}=G_{\F_p}\cdot \ol{\lambda}.
\]
From Lemma~\ref{Flat orbit lemma}, we know that the morphisms
\[
G_{\ol{\Q_p}}\to (G/F)_{\ol{\Q_p}}=G_{\ol{\Q_p}}/F_{\ol{\Q_p}}\text{ and~}G_{\ol{\F_p}}\to (G/F)_{\ol{\F_p}}=G_{\ol{\F_p}}\cdot \ol{\lambda}
\]
are flat morphisms of algebraic varieties. We know that~$G_{\Z_p}$ is flat and smooth over~$\Z_p$ by hypothesis.

By the "Critère de platitude par fibre", (\cite[Part 2, \S5.6, Lem.~5.21, p.\,132]{FGA} or~\citestacks[Lem. 37.16.3.]{039D}) the morphism
\[
G\to G/F
\]
is flat. Consider~$\Stab_{G_{\Z_p}}(\lambda)\leq G_{\Z_p}$ a subgroup scheme over~$\Z_p$. Let~$\xi:\Spec(\Z_p)\to G/F$ be the morphism corresponding to~$\lambda\in (G/F)(\Z_p)$. Then~$\Stab_{G_{\Z_p}}(\lambda)\to \Spec(\Z_p)$ is the pullback of~$\pi:G\to G/F$ by~$\xi$. As~$G\to G/F$ is flat, the morphism~$\Stab_{G_{\Z_p}}(\lambda)\to \Spec(\Z_p)$ is flat. By (3) of Prop.~\ref{flat lift}, the generic fibre~$F_{\Q_p}=\Stab_{G_{\Q_p}}(\lambda)$ is Zariski dense in~$\Stab_{G_{\Z_p}}(\lambda)$.
Thus~
\[\Stab_{G_{\Z_p}}(\lambda)=F_{\Z_p}.\]
 Thus~$\Stab_{G_{\F_p}}(\lambda)=F_{\F_p}$, and is smooth. Thus
\[
(G/F)_{\F_p}=G_{\F_p}\cdot \ol{\lambda}=G_{\F_p}/F_{\F_p}.
\]
This proves~\eqref{Sesh iso} and concludes the proof of Lemma~\ref{Seshadri}.
%By assumption,~$F_{\Z_p}$ is smooth. Thus~$F_{\F_p}\leq G_{\F_p}$ is smooth. By  Lemma~\ref{Flat orbit lemma}, the morphisms
%\[
%G_{\ol{\Q_p}}\to (G/F)_{\ol{\Q_p}}\cdot v=S_{\ol{\Q_p}}\text{ and~}G_{\ol{\F_p}}\to G_{\ol{\F_p}}/F_{\ol{\F_p}}
%\]
%are smooth morphisms of algebraic varieties. By the "Critère de lissité par fibre" (\cite[17.11.1 d)]{EGA44}, see also~\cite[$\text{VI}_\text{B}$ Prop.~9.2 (xii) (and~V Th.~10.1.2)]{SGA31}) the morphism
%\[
%G\to G/F
%\]
%is smooth.
%Thanks to~\eqref{proof:G/F reduced}, this proves Lemma~\ref{Seshadri}.
\end{proof}

%\begin{proof}
%Because~$S':=G_{\F_p}/F_{\F_p}$ is reduced, we have a factorisation
%\begin{equation}\label{S' to Sesh}
%S'\to ((G/F)_{\F_p})^{red}\to (G/F)_{\F_p}.
%\end{equation}
%Let~$x$ be a geometric generic point of~$((G/F)_{\F_p})^{red}$
%and let~$k$ be its residue field.  Then there is a unique inverse image~$x'\in S'(k)$, and~$S'\to ((G/F)_{\F_p})^{red}$ will be an isomorphism from a neighbourhood~$U'$ of~$x'$ onto a neighbourhood~$U$ of~$x$ (cf. e.g.~\citestacks{0BXP}). Using the action of~$G_{\F_p}$
%we may assume~$U$ and~$U'$ are~$G_{\F_p}$-invariant.
%We have necessarily~$U'=S'$, and~\eqref{S' to Sesh} is an open immersion. It is also surjective on~$\ol{\F_p}$-points, hence surjective.
%\end{proof}
%\subsubsection{}\label{japan}
%We recall that the ring~$\Z_p$ is universally Japanese (?\cite{}).
%In particular, the integral closure~$A$ of the finite type~$\Z_p$-algebra~$\Z_p[S]$ in the finite type~$\Z_p$-algebra~$\Z_p[G/F]$ is of finite type over~$\Z_p$. As~$\Z_p[S]$ and~$\Z_p[G/F]$ are~$G(\Z_p)$-stable subalgebras of~$\Q_p[G]$,
%the subalgebra~$A$ is also~$G(\Z_p)$-stable.

\subsection{Normalisation and Integrality}\label{sec:integral}

%\begin{lemma}\label{Lemma integral} We keep the situation of Lemma~\ref{Lemma S S'}, and assume that~$F:=\Stab_{G_{\Q_p}}(v)\leq GL(n)$ is hyperspecial. We use the notations of~\ref{secGIT}.
%
%The morphism~$G/F\to S$ is integral, and finite.
%%Then~$S$ and~$G/F$ have the same normalisation in~$G_{\Q_p}$.
%%The normalisation~$\wt{S}$ of~$S$ in~$G$ is that~$\wt{G/F}$ of~$G/F= \Spec(\Z_p[G]^F)$.
%\end{lemma}
\begin{lemma}\label{Lemma integral} We keep the situation of Lemma~\ref{Lemma S S'}.

We assume that~$G_{\Q_p}\cdot v\subseteq \A^n_{\Q_p}$ is Zariski closed.

We assume that there exists smooth reductive closed subgroup scheme~$F_{\Z_p}\leq GL(n)_{\Z_p}$ such  that~$F_{\Q_p}=\Stab_{G_{\Q_p}}(v)$. We use the notations of~\S\ref{secGIT}.

Then the map~$G/F\to S$ identifies~$G/F$ with the normalisation of~$S$ (in its fraction field). 
\end{lemma}
The map~$G/F\to S$ is obtained as follow. As~$G_{\Q_p}\cdot v$ is closed, and by definition of~$F$, we have~$G_{\Q_p}/F_{\Q_p}\simeq S_{\Q_p}$, or equivalently~$\Q_p[G]^F=\Q_p[S]$.
As~$S$ is flat, we have~$\Z_p[S]\subseteq \Q_p[S]$. As~$v\in{\Z_p}^n$ and~$G_{\Z_p}\leq GL(n)_{\Z_p}$ the map~$g\to g\cdot v$ is defined over~$\Z_p$. Thus~$\Z_p[S]\subseteq \Z_p[G]$. We thus have,
\[
\Z_p[S]\subseteq \Z_p[G]\cap \Q_p[G]^F=\Z_p[G/F].
\]

\begin{proof}[Proof of Lemma~\ref{Lemma integral}]Let us denote by~$\wt{\Z_p[S]}$ the integral closure of~$\Z_p[S]$ in~$\Z_p[G/F]$, and define~$\wt{S}=\Spec(\wt{\Z_p[S]})$. We denote by~$\pi:G/F\to S$ the morphism from the statement of Lemma~\ref{Lemma integral}, and denote by
\[
\wt{\pi}:G/F\to \wt{S}\text{ and }\nu:\wt{S}\to S
\]
the  morphisms given by the inclusions~$\Z_p[S]\subseteq\wt{\Z_p[S]}\subseteq \Z_p[G/F]$.

Let us prove that Lemma~\ref{Lemma integral} is the consequence of the following claims.
\begin{enumerate}
\item The scheme~$G/F$ is normal. \label{claim1}
\item The morphism~$\pi:G/F\to S$ is quasi-finite. \label{claim2}
\item The map~$G(\ol{\F_p})\to \wt{S}(\ol{\F_p})$ is surjective. \label{claim3}
\end{enumerate}
By~\ref{claim2} and~\eqref{G/F type fini}, we may apply Zariski's Main Theorem in the form
~\citestacks[Theorem 29.55.1 (Algebraic version of Zariski's Main Theorem)]{03GT}.
%~\citestacks[Th. 10.123.12 (Zariski's Main Theorem)]{00Q9}.
It follows that the map~$\wt{\pi}:G/F\to \wt{S}$ is an open immersion. Recall that the morphism~$(G/F)_{\Q_p}=G_{\Q_p}/F_{\Q_p}\to S_{\Q_p}$ is an isomorphism. As normalisation commutes with localisation,~$(G/F)_{\Q_p}$ is normal, and~$\wt{S}_{\Q_p}$ is the normalisation of~$S_{\Q_p}$ in~$(G/F)_{\Q_p}$. Thus, the morphism~$(G/F)_{\Q_p}\to \wt{S}_{\Q_p}$ is an isomorphism. 
By~\ref{claim3} the map~$G(\ol{\F_p})\to (G/F)(\ol{\F_p})\to \wt{S}(\ol{\F_p})$ is surjective. It follows that the map~$(G/F)(\ol{\F_p})\to \wt{S}(\ol{\F_p})$ is surjective. Thus, the open immersion~$\wt{\pi}:G/F\to \wt{S}$ is surjective, and it is thus an isomorphism.

Let~$\Z_p[S]^{norm}$ be the integral closure of~$\Z_p[S]$ in~$\Q_p(S)$. By~\ref{claim1}, the scheme~$G/F$ is normal. That is,~$\Z_p[G/F]$ is integrally closed in~$\Q_p(G/F)$. It follows that~$\Z_p[S]^{norm}\subseteq \Z_p[G/F]^{norm}=\Z_p[G/F]$. Thus~${\Z_p[S]}^{norm}$ is integral over~$\Z_p[S]$ and~${\Z_p[S]}^{norm}\subset \Z_p[G/F]$. This proves~${\Z_p[S]}^{norm}\subseteq \wt{\Z_p[S]}$. As~$S_{\Q_p}\simeq (G/F)_{\Q_p}$, we have~$\Q_p[S]=\Q_p[G/F]$ and~$\Q_p(S)=\Q_p(G/F)$. Thus~$\wt{\Z_p[S]}$ is integral over~$\Z_p[S]$ and~$\wt{\Z_p[S]}\subset \Q_p(S)$. This proves~$\wt{\Z_p[S]}\subseteq {\Z_p[S]}^{norm}$.

We conclude that~$G/F=\wt{S}$ and that~$\wt{S}$ is the normalisation of~$\wt{S}$ in its fraction field.

We have proved Lemma~\ref{Lemma integral} assuming the claims~\ref{claim1},~\ref{claim2} and~\ref{claim3}.
We now prove each of the claims.

\begin{proof}[Proof of claim~\ref{claim1}]
By assumption, the scheme~$G_{\Z_p}$ is smooth over the regular ring~$\Z_p$. Thus~$G_{\Z_p}$ is regular,
and normal, and~$\Z_p[G]$ is integrally closed in its fraction field~$\Q_p(G)$.

Let~$\lambda\in \Q_p(G/F)$ be integral over~$\Z_p[G/F]$. Then~$\lambda$ is integral over~$\Z_p[G]$: thus~$\lambda\in\Z_p[G]$. We also have~$\lambda\in\Q_p(G/F)\subseteq \Q_p(G)^F$. Therefore~$\lambda\in \Z_p[G]^F$.

This proves that~$\Z_p[G/F]$ is integrally closed in~$\Q_p(G/F)$, that is, that~$G/F$ is normal.
\end{proof}

\begin{proof}[Proof of claim~\ref{claim2}]
By assumption~\eqref{S S' dimension}, and by definition of~$F_{\Q_p}$, and as~$F_{\Z_p}$ is smooth over~$\Z_p$, we have
\[
\dim(\Stab_{G_{\F_p}}(\ol{v}))=\dim(\Stab_{G_{\Q_p}}(v))=\dim(F_{\Q_p})=\dim(F_{\F_p}).
\]
By~\eqref{dim formula}, we have~$\dim(S_{\F_p})=\dim(G_{\F_p})-\dim(\Stab_{G_{\F_p}}(\ol{v}))$.
By Lem.~\ref{Seshadri}, we have~$(G/F)_{\F_p}=G_{\F_p}/F_{\F_p}$, and by~\eqref{dim formula}, we have~$\dim(G_{\F_p}/F_{\F_p})=\dim(G_{\F_p})-\dim(F_{\F_p})$.

Thus~$\dim((G/F)_{\F_p})=\dim(S_{\F_p})$. 

Lem.~\ref{Lemma S S'} implies that the map~$G_{\F_p}\to S_{\F_p}$ is dominant and that~$S_{\F_p}$ is irreducible. 
Thus, the map~$(G/F)_{\F_p}\to S_{\F_p}$ is dominant. Lem.~\ref{Seshadri} implies that~$(G/F)_{\F_p}$ is irreducible.
By~\cite[3.22 (e), p.\,95]{Ha} (with~$e=0$), there is a non empty Zariski open subset~$U\subseteq S_{\F_p}$, such that for~$u\in U(\ol{\F_p})$ the fibre of~$(G/F)_{\F_p}\to S$ above~$u$ is of dimension~$0$. As, by Lem.~\ref{Lemma S S'}, the action of~$G(\ol{\F_p})$ on~$S(\ol{\F_p})$ is transitive, we may take~$U=S_{\F_p}$.

The claim follows. %(cf Lemma 29.29.5. 0397)
\end{proof}

\begin{proof}[Proof of claim~\ref{claim3}] Our goal is to prove that~$\wt{S}(\ol{\F_p})$ is a single~$G(\ol{\F_p})$-orbit.

By~\citestacks[Prop. 10.162.16.]{0335},~$\Z_p$ is universally Japanese. Thus the algebra~$\Z_p[S]$, which is an integral domain of finite type over~$\Z_p$, is Japanese.  Namely, the integral closure~$\wt{\Z_p[S]}$,  of~$\Z_p[S]$ in its field of fractions is of finite type over~$\Z_p$.

We apply Proposition~\ref{prop:Aorbite} and write accordingly~$\wt{S}\simeq \ol{G\cdot \lambda}^{Zar(\A^m_{\Z_p})}$ for some~$m\in\Z_{\geq0}$. 

We prove that the assumptions~\eqref{S S' orbite fermee} and~\eqref{S S' dimension} of Lemma~\ref{Lemma S S'} are satisfied. The claim will then follow from Lemma~\ref{Lemma S S'}.

By Cor.~\ref{cor:T S} for~$T=(\wt{S}_{\F_p})^{red}$ and~$t=\ol{\lambda}$, and by~\eqref{S S' dimension} for~$v\in S$, 
we have
\[
\dim \Stab_{G_{\F_p}}(\ol{\lambda})\leq \dim \Stab_{G_{\F_p}}(\ol{v})=\dim \Stab_{G_{\Q_p}}(v)=\dim F_{\Q_p}.
\]
As~$\wt{S}$ is integral over~$S$, we have~$\dim \wt{S}_{\F_p}\leq \dim S_{\F_p}$. Thus
\begin{multline}
\dim G_{\F_p}-\dim  \Stab_{G_{\F_p}}(\ol{\lambda})= \dim(G_{\F_p}\cdot \ol{\lambda})\leq \dim \wt{S}_{\F_p}
\\
\leq 
\dim S_{\F_p}=\dim G_{\F_p} - \dim \Stab_{G_{\F_p}}(\ol{v})=\dim G_{\F_p} - \dim F_{\F_p}.
\end{multline}
Thus~$\dim \Stab_{G_{\F_p}}(\ol{\lambda})=\dim F_{\F_p}=\dim F_{\Q_p}=\dim \Stab_{G_{\Q_p}}(\lambda)$.

This proves~\eqref{S S' dimension}.

Assume by contradiction that~$G_{\F_p}\cdot \ol{\lambda}$ is not closed in~$\wt{S}_{\F_p}$.
Then there exists~$t\in \ol{G_{\F_p}\cdot \ol{\lambda}}\smallsetminus G_{\F_p}\cdot \ol{\lambda}$.
Let~$T=\ol{G_{\F_p}\cdot t}^{red}$. 
Then~$T$ is~$G_{\F_p}$-stable. 
We have~$\dim(T)=\dim(G_{\F_p}\cdot t)=\dim(G_{\F_p})-\dim\Stab_{G_{\F_p}}(t)$, and~$\dim(T)\leq \dim \ol{G_{\F_p}\cdot \ol{\lambda}}\smallsetminus G_{\F_p}\cdot \ol{\lambda}\leq \dim (G_{\F_p}\cdot \ol{\lambda} )-1<\dim G_{\F_p}- \dim\Stab_{G_{\F_p}}(\ol{\lambda}) $.
Thus~$\dim\Stab_{G_{\F_p}}(t)\geq \dim\Stab_{G_{\F_p}}(\ol{\lambda})+1$. This contradicts Cor.~\ref{cor:T S} for~$t\in T=\ol{G_{\F_p}\cdot t}^{red}$.

This proves~\eqref{S S' orbite fermee}.

We may thus apply Lemma~\ref{Lemma S S'}. We deduce~$\wt{S}_{\F_p}=G_{\F_p}\cdot  \ol{\lambda}$.
This proves the claim.
\end{proof}
This concludes the proof of Lemma~\ref{Lemma integral}.
\end{proof}

\subsection{Flatness and Smoothness}\label{sec:flat smooth}
\begin{lemma}\label{platitude}Let~$F_{\Z_p}\leq G_{\Z_p}\leq GL(n)_{\Z_p}$ be a closed smooth reductive subgroup schemes.

Then the maps
\begin{equation}\label{univ orbit map}
%\omega:
G\to G/F\text{ and }G/F\to \Spec(\Z_p) 
\end{equation}
are smooth.
%Under hypotheses of lemmas~\ref{} and~\ref{}, the orbit map morphism,
%of schemes over~$\Z_p$,
%\[
%\omega:G_{\ol{\Z_p}}\to S_{\ol{\Z_p}}
%\]
%is flat.
\end{lemma}
%\subsubsection{Remark} Provided~$\Stab_{G_{\F_p}}(\ol{v})$ is a reduced scheme (equivalently~$\Stab_{G_{\F_p}}(\ol{v})=F_{\F_p}$
%as a scheme), one can even prove that~$\omega$ is a smooth map, using~\cite[17.11.d]{EGA4} instead and the analogue of Lemma~\ref{Flat orbit lemma}.)
%Check https://math.stackexchange.com/questions/4293553/local-flatness-criterion-a-morphism-x-to-y-of-schemes-over-s-is-flat-if-an

%\cite[5.21]{FGA}
%\cite[17.11.d]{EGA4}
\begin{proof} We apply the "Critère de lissité par fibre"~\cite[17.8.2]{EGA44} with~$h=f\circ g:X\to Y\to S$ the morphisms~$G\to G/F\to \Spec(\Z_p)$.

By assumption~$G\to \Spec(\Z_p)$ is smooth, and in particular flat. For~$s=\Spec(\Q_p)$, the morphism~$X_s\to Y_s$ is~$G_{\Q_p}\to (G/F)_{\Q_p}\simeq G_{\Q_p}/F_{\Q_p}$. By Lem.~\ref{Seshadri}, for~$\ol{s}=\Spec(\F_p)$, the morphism~$g_{\ol{s}}:X_{\ol{s}}\to Y_{\ol{s}}$ is~$G_{\F_p}\to (G/F)_{\F_p}\simeq G_{\F_p}/F_{\F_p}$. By assumption~$F_{\Z_p}$ is smooth, and thus~$F_{\Q_p}$ and~$F_{\F_p}$ are smooth. By Lem.~\ref{Flat orbit lemma}, we deduce that~$g_s$ and~$g_{\ol{s}}$ are smooth morphisms.

Thus, by~\cite[17.8.2]{EGA44}, the morphism~$G\to G/F$ is smooth.

Recall that~$G\to \Spec(\Z_p)$ is smooth.  Thus, by~\cite[17.11.1 b)\,$\Rightarrow$\,a)]{EGA44}, the morphism~$G/F\to \Spec(\Z_p)$ is smooth on the image of~$G\to G/F$. 

Note that~$g_s$ and~$g_{\ol{s}}$ are surjective. Thus~$G/F\to \Spec(\Z_p)$ is smooth.
\end{proof}

\begin{lemma}\label{Flat orbit lemma}Over a field~$\kappa$
let~$G\leq GL(n)_\kappa$ be a algebraic subgroup (smooth closed group subscheme), and choose~$v\in\kappa^n$.
Then the map ``orbit through~$v$'' map
\begin{equation}\label{omega smooth?}
\omega:G\to G\cdot v
\end{equation}
is flat, where~$G\cdot v\simeq G/\Stab_G(v)$ is locally closed and endowed with a reduced scheme structure.

We have, using Krull dimension,
\begin{equation}\label{dim formula}
\dim(G\cdot v)=\dim(G)-\dim(\Stab_G(v)).
\end{equation}

If~$\Stab_G(v)$ is smooth as a group scheme\footnote{In practice~$\dim \Stab_{\mathfrak{g}}(v)=\dim \Stab_{G}(v)$.}, then~$\omega$ is a smooth map and~$G\cdot v$ is smooth (regular).
\end{lemma}
\begin{proof} According to the Orbit Lemma~\cite[\S{I} 1.8]{BorelLAG},  the orbit~$G\cdot v$ is locally closed.

Because~$G\cdot v$ is reduced, by~\citestacks[Prop. 29.27.2]{052B}, there exists a non empty open subset~$U\subseteq G\cdot v$ such that~$\omega$ is flat above~$U$. As~$U$ is~$G$-invariant, the map~$\omega$ is flat everywhere.

We deduce~\eqref{dim formula} from the flat case of~\citestacks[Lem. 29.28.2.]{02JS} and~\citestacks[Lem. 29.29.3.]{02NL}, (using Krull dimension, cf.~\citestacks[Def. 5.10.1.]{0055}).  (One can also find~\eqref{dim formula} in~\cite[p.\,7]{GIT}).)

Concerning smoothness, see
%for instance~Prop.~\ref{LMB Prop} and~\ref{LMB vs Sesh} or
~\cite[$\text{VI}_\text{B}$ Prop.~9.2 (xii) (and~V Th.~10.1.2)]{SGA31}.
\end{proof}

\subsection{Flatness and lifting of closed points}\label{sec:flat recalls}
%For reference and on request, we recall the following.

Let~$\ol{\Q_p}$ be an algebraically closed algebraic extension of~$\Q_p$, and denote by~$\ol{\Z_p}$
the integral closure of~${\Z_p}$ in~$\ol{\Q_p}$.
We fix a ring homomorphism (the "reduction map")
\[
r:\ol{\Z_p}\to \ol{\F_p},
\]
and denote the induced morphism~$(x_1,\ldots,x_n)\mapsto (r(x_1),\ldots,r(x_n))$ by
\[
r^n:\A^n(\ol{\Z_p})=\ol{\Z_p}^n
\to 
\A^n(\ol{\F_p})=\ol{\F_p}^n.
\]

\begin{proposition}[cf.{~\cite[Prop. 14.5.6, Rem.~14.5.7]{EGA43}}]\label{flat lift}
Consider a Zariski closed  subscheme~$X\subseteq \A^n_{\ol{\Z_p}}$  of finite presentation, and denote by~$X_{red}$ its reduced subscheme.

Then the following properties are equivalent.
\begin{enumerate}
\item The scheme~$X_{red}$ is flat over~$\ol{\Z_p}$.
\item We have
\begin{equation}\label{FlatLiftRecallEq}
X(\ol{\F_p})
= 
r^n(X(\ol{\Z_p})).
\end{equation}
\item The generic fibre~$X_{\ol{\Q_p}}$ is Zariski dense in~$X$.
\item No irreducible component
of~$X$ is fully contained in~$\A^n_{\ol{\F_p}}$.
\end{enumerate}
%if and only if
%\begin{equation}\label{FlatLiftRecallEq}
%X(\ol{\F_p})
%= 
%r^n(X(\ol{\Z_p})),
%\end{equation}
%and if and only if , and if and only if 
\end{proposition}
\begin{proof}This is a translation of~\cite[Prop. 14.5.6, Rem.~14.5.7]{EGA43}.\footnote{Loc. cit. assumes the base is noetherian. Even though our base~$\ol{\Z_p}$ is not noetherian,~$X$ can be obtained by base change from a model over a ring of integers~$O_K$ of a finite extension~$K/\Q_p$. Then~Rem.~14.5.7 of loc. cit. applies.}
\end{proof}
\subsubsection{Remark}\label{rem:flat red} If~$X$ is flat over~$\ol{\Z_p}$, then~$X\to \Spec(\ol{\Z_p})$ is universaly open, by~\cite[Th.~(2.4.6)]{EGA42}, and thus~$X^{red}\to \Spec(\ol{\Z_p})$ is flat, by~\cite[Prop. 14.5.6, Rem.~14.5.7]{EGA43}.
%Compared to using Hensel Lemma\footnote{Compare~\cite[Th.~18.5.17]{EGA44}.}, there are several differences: one assumes only \emph{flatness} instead of \emph{smoothness}; but it may also be that none of the lifts is defined over unramified extensions; this is not a statement over~$\Q_p$, one has to pass to~$\ol{\Q_p}$. Consider for instance~$X=\Spec(\ol{\Z_p}[X]/(X^2-p))$.

\begin{proposition}\label{flat lift 2}
Let~$\pi:X\to Y$ be a morphism of reduced affine schemes of finite presentation over~$\ol{\Z_p}$.

Assume that~$\pi$ is flat. Then,  
\[
\forall y\in Y(\ol{\Z_p}), \ol{y}\in \pi(X(\ol{\F_p})) \Leftrightarrow y\in \pi(X(\ol{\Z_p})).
\]
\end{proposition}
\begin{proof}If~$y=\pi(x)$, then~$\ol{y}=\pi(\ol{x})$. This proves one implication.

Let~$\ol{y}=\pi(\ol{x})$ with~$\ol{x}\in X(\ol{\F_p})$. Let~$X_y$ be the fibre of~$X$ over~$y$.
Then~$\ol{x}\in X_y(\ol{\F_p})$. We have~$x\in X_y(\ol{\Z_p})$ if and only if:~$x\in  X(\ol{\Z_p})$
and~$\pi(x)=y$.

Let~$Z=\Spec(\ol{\Z_p})$, and denote by~$z\in Z(\ol{\Z_p})$ the unique element. Let~$\xi:Z\to Y$ 
be the morphism such that~$\xi(z)=y$. Then~$X_y\to Z$ is the pullback of~$\pi:X\to Y$ by~$\xi$.
As~$\pi$ is flat, the morphism~$X_y\to Z$ is flat. This implies, by Remark~\ref{rem:flat red}, that~${X_y}^{red.}\to Z^{red.}=Z$ is flat. As~$X$ is affine of finite presentation, the fibre subscheme~$X_y$ is affine of finite presentation. By Prop.~\ref{flat lift}, there exists~$x\in X_y(\ol{\Z_p})$ such 
that its reduction in~$X_y(\ol{\F_p})$ is~$\ol{x}$. Thus~$y\in\pi(X(\ol{\Z_p}))$. We proved the second implication.
\end{proof}
\subsection{Integrality and lifting}\label{sec:integral lift}
\begin{proposition}\label{prop:integral lift}
 Let~$\pi:X\to Y$ be a morphism of affine schemes over~$\ol{\Z_p}$.

Assume that~$\pi$ is integral. Then
\[
\forall x\in X(\ol{\Q_p}), x\in X(\ol{\Z_p})\Leftrightarrow \pi(x)\in Y(\ol{\Z_p}).\\
\]
\end{proposition}
\begin{proof}If~$x\in X(\ol{\Z_p})$, then~$\pi(x)\in Y(\ol{\Z_p})$. We prove the other implication.
 
We write~$X=\Spec(B)$ and~$Y=\Spec(A)$ and~$\pi^*:A\to B$ the morphism corresponding to~$\pi$. Let~$x\in X(\ol{\Q_p})$, and let~$\phi:B\to \ol{\Q_p}$ be the corresponding morphism.

Assume~$\pi(x)\in Y(\ol{\Z_p})$. Then~$\phi\circ \pi^*(A)\subseteq \ol{\Z_p}$.
As~$B$ is integral over~$A$, the algebra~$\phi(B)$ is integral over~$\phi\circ\pi^*(A)$, and is integral over~$\ol{\Z_p}$.
As~$\ol{\Z_p}$ is integrally closed in~$\ol{\Q_p}$, we have~$\phi(B)\subseteq \ol{\Z_p}$. Thus~$x\in X(\ol{\Z_p})$.
\end{proof}

\section{Slopes weights estimates}\label{sec:slopes}

We consider an integer~$n\in\Z_{\geq0}$ and an Euclidean distance~$d(~,~)$ on~$\R^n$. The quantities~$c,c'$ and~$\gamma$ will implicitly also depend on~$d(~,~)$.

\begin{lemma}\label{lemma cvx 1}
Let~$\Sigma$ be a finite set of linear forms on~$\R^n$, let the function~$h_{\Sigma}:\R^n\to\R_{\geq0}$
be given by
\[
h_{\Sigma}(x)=\max_{\lambda\in\{0\}\cup\Sigma}\lambda(x),
\]
and define~$C=C(\Sigma):=\{x\in\R^n|h_{\Sigma}(x)=0\}$. %We denote by~$d(~,~)$ the Euclidean distance.

Then there exist~$c(\Sigma),c'(\Sigma)\in\R_{>0}$ such that: for all~$x\in\R^n$ satisfying
\begin{equation}\label{mini to C}
d(x,C)=d(x,0)
\end{equation}
we have
\begin{equation}\label{cvx c c'}
c(\Sigma)\cdot d(0,x)\leq h_{\Sigma}(x)\leq c'(\Sigma)\cdot d(0,x).
\end{equation}
\end{lemma}
\begin{proof}Let~$x\in \R^n$ be arbitrary such that~$d(x,C)=d(x,0)$. Let~$x'=\gamma\cdot x$ 
with~$\gamma\in\R_{>0}$. For~$v\in \R^n$ we have~$h_\Sigma(\gamma\cdot v)=\gamma\cdot h_\Sigma(v)$.
It follows that~$\gamma\cdot C=C$. For every~$c'\in C$, we have~$c:=c'/\gamma\in C$.
We have
\[
d(x',c')=d(\gamma\cdot x,\gamma\cdot c)=\gamma\cdot d(x,c)\geq \gamma\cdot d(x,0)=d(\gamma\cdot x,0)=d(x',0).
\]
%
%with~$0<\lambda<1$, resp.~$\lambda>1$. Since~$x'$ is on the segment~$[0;x]$, resp.~$x\in [0;x']$, we have for~$c\in C$.
%\[
%d(c,x')\geq d(c,x)-d(x,x')\geq d(0,x)-d(x,x')=d(0,x'),
%\]
%resp.
%\[
%d(c,x')\geq d(c,x)+d(x,x')\geq d(0,x)+d(x,x')=d(0,x').
%\]

Thus~$d(0,x')=d(C,x')$. Since the inequalities~\eqref{cvx c c'} are homogeneous, we may assume~$x\neq0$, and substituting~$x$ with~$x':=x/d(0,x)$, we may assume that
\begin{equation}\label{sphere}
d(0,x)=1.
\end{equation}
We can rewrite the condition~\eqref{mini to C} as
\begin{equation}\label{mini to c}
\forall c\in C,~d(0,x)\leq d(c,x).
\end{equation}
The set~$C^\bot:=\set{x\in\R^n}{d(0,x)=d(C,x)}$ is an intersection of affine half-spaces, and is a closed set~$C^\bot\subseteq \R^{n}$ (it is the \emph{polar dual cone} to~$C$). The intersection~$K$ of~$C^\bot$ with the unit sphere~$\set{x\in\R^n}{d(0,x)=1}$ is thus a compact set. 
We have~$x\in K$, by~\eqref{mini to C} and~\eqref{sphere}.

The continuous function~$h_{\Sigma}$ has a minimum value and maximum value on the compact~$K$, which we denote by
\begin{equation}\label{def cvx c c'}
c(\Sigma):=\min_{k\in K}h_{\Sigma}(k)\text{ and }c'(\Sigma):=\max_{k\in K}h_{\Sigma}(k).
\end{equation}
By definition,~\eqref{cvx c c'} is satisfied and we have~$0\leq c(\Sigma)\leq c'(\Sigma)<+\infty$. 
It will be enough to prove~$0<c(\Sigma)$. 

Assume by contradiction that~$c(\Sigma)= 0$ and choose~$k\in K$ such that~$h_{\Sigma}(k)=0$.
Then~$k\in C$. From~\eqref{mini to c} for~$x=c=k$, we deduce~$d(0,k)\leq d(k,k)=0$, contradicting~\eqref{sphere}.
\end{proof}
\begin{proposition}\label{prop slopes comparison}
We keep the setting of Lemma~\ref{lemma cvx 1}. 
%We have
%\begin{equation}\label{cvx lin}
%\forall x\in \R^d, c(\Sigma)\cdot d(C,x)\leq h_{\Sigma}(x)\leq c'(\Sigma)\cdot d(C,x).
%\end{equation}

Let us fix a map~$\mu:\Sigma\to\R_{\leq 0}$ and let~$h_{\mu}:\R^n\to \R_{\geq0}$ be defined by
\[
h_{\mu}(x)=\max\{0;\max_{\lambda\in\Sigma}\lambda(x)+\mu(\lambda)\}.
\]
Define~$C=C_{\mu}=\{x\in\R^n|h_{\mu}(x)=0\}$.

Then, %using notations~\eqref{def cvx c c'},
\begin{equation}\label{cvx affine}
\forall x\in \R^n, c(\Sigma)\cdot d(C,x)\leq h_{\mu}(x)\leq c'(\Sigma)\cdot d(C,x).
\end{equation}
\end{proposition}
\begin{proof}
Define~$\ol{\Sigma}:=\{\lambda\in\Sigma|\mu(\lambda)=0\}$ and~$\ol{C}=\{x\in\R^n|h_{\ol{\Sigma}}(x)=0\}$.
We have~$h_{\ol{\Sigma}}\leq h_{\mu}$ and thus
\[
C\subseteq \ol{C}
\]

In a first step we prove~\eqref{cvx affine} with the extra condition
\begin{equation}\label{cvx extra}
d(x,C)=d(x,0)\text{ (that is:~$\forall c\in C, d(x,c)\geq d(x,0)$)}. 
\end{equation}

Let~$\norm{~}$ denote the euclidean norms induced by~$d(~,~)$ on~$\R^n$ and its dual.
For~$a\in \R^n$ we have 
\[
\max_{\sigma\in\Sigma}\sigma(a)\leq \norm{a}\cdot \max_{\sigma\in\Sigma}\norm{\sigma}.
\]
Define~$\mu_0:=\max\set{\mu(\sigma)}{\sigma\in\Sigma\smallsetminus\ol{\Sigma}}<0$.
Then, if~$a\in\R^n$ satisfies
\begin{equation}\label{petit a}
\norm{a}\cdot \max_{\sigma\in\Sigma}\norm{\sigma}\leq-\mu_0,
\end{equation}
we have
\begin{equation}\label{petit a et negatif}
h_{\mu}(a)-h_{\ol{\Sigma}}(a)\leq 0,\text{ and thus }h_{\mu}(a)=h_{\ol{\Sigma}}(a).
\end{equation}
Let us prove that~$d(x,\ol{C})=d(x,0)$.
\begin{proof}We want to prove that for an arbitrary~$b\in \ol{C}$ we have
\begin{equation}\label{cvx claim 1}
d(x,b)\geq d(x,0).
\end{equation}
Let~$\lambda\in\R_{>0}$ be sufficiently small so that~$a:=\lambda\cdot b$ satisfies~\eqref{petit a}.
We deduce from~\eqref{petit a et negatif} that~$a\in C$, and from~\eqref{cvx extra} that
\[
d(x,a)\geq d(x,0).
\]
Equivalently, denoting by~$(~,~)$ the euclidean scalar product,~$(a,x)\leq 0$.
It follows~$(b,x)=\lambda\cdot (a,x)\leq 0$ and, equivalently,~\eqref{cvx claim 1}.
\end{proof}
Applying Lemma~\ref{lemma cvx 1} to~$\ol{\Sigma}$, we deduce~\eqref{cvx affine} under the assumption~\eqref{cvx extra}.

Note that~$\ol{\Sigma}$ is a subset of the finite set~$\Sigma$. Thus, there are only finitely many possibilities for~$\ol{\Sigma}$. As a consequence, we may assume, when applying Lemma~\ref{lemma cvx 1},  that~$c(\ol{\Sigma})$ and~$c'(\ol{\Sigma})$ do not depend on~$\ol{\Sigma}$.

We now reduce the general case to the first step, by a translation of the origin of~$\R^n$.

Let~$x_0\in C$ be such that
\[
d(x,C)=d(x,x_0)
\]
and define
\[
\mu'(\lambda)=\lambda(x_0)+\mu(\lambda),
\]
so that
\[
h_{\mu'}(y)=h_\mu(y+x_0).
\]
and~$C_{\mu'}=C_{\mu}-x_0$. Thus
\[
d(x-x_0,C_{\mu'})=d(x-x_0,C_{\mu}-x_0)=d(x,C_{\mu})=d(x,x_0)=d(x-x_0,0).
\]
From~$x_0\in C_\mu$, we deduce~$h_{\mu'}(0)=h_\mu(x_0)=0$ and~$\forall\sigma\in\Sigma,\mu'(\sigma)\leq 0$.

Then~\eqref{cvx affine} for~$x$ follows from the first step applied to~$x-x_0$.
\end{proof}
Defining
\[
\gamma(\Sigma_0)=
\frac{\max\set{c'(\Sigma)}{\Sigma\subseteq \Sigma_0}}
{\hspace{2.2pt}\min\set{\phantom{{}'}c(\Sigma)}{\Sigma\subseteq \Sigma_0}}
\] 
we deduce the following.
\begin{corollary}\label{coro slopes}
Let~$\Sigma_0$ be a finite set of linear forms on~$\R^n$. There exists~$\gamma(\Sigma_0)\in\R_{>0}$
such that for~$\Sigma,\Sigma'\subseteq \Sigma_0$
and~$\mu:\Sigma\to\R_{\leq0}$ and~$\mu':\Sigma'\to\R_{\leq0}$ such that~$C_{\mu}=C_{\mu'}$,
we have
\[
\forall x\in \R^n, h_{\mu}(x)\leq \gamma(\Sigma_0)\cdot h_{\mu'}(x).
\]
\end{corollary}

\appendix

\section{On Complete reducibility and closed orbit of Lie algebras over~$\F_p$}\label{AppA}
\begin{lemma}Let~$V$ be a finite dimensional vector space over a field~$K\neq \F_2$, and let~$G'\leq H\leq G\leq GL(V)$ be subgroups.

Assume that~$V$ is semisimple as a representation of~$G$ and as a representation of~$G'$,
and that
\[
Z_{GL(V)}(G)=Z_{GL(V)}(G').
\]
Then~$V$ is semisimple as a representation of~$H$ and
\[
Z_{GL(V)}(H)=Z_{GL(V)}(G).
\]
\end{lemma}
\begin{proof}Let~$W\leq V$ be a~$H$-stable vector subspace.
Then~$W$ is~$G'$-invariant. By assumption, there exists~$W'\leq V$
a supplementary~$G'$-invariant subspace. Let~$\pi$ be the projector with image~$W$ and kernel~$W'$,
and choose~$\lambda\in K\smallsetminus\{0;1\}$. Define~$g=\pi+\lambda\cdot (1_V-\pi)\in \End(V)$.
Then
\[
g\in 
Z_{GL(V)}(G')=Z_{GL(V)}(G)\leq
Z_{GL(V)}(H).
\]
Thus~$g$ is~$H$-equivariant, and thus~$\pi=\frac{1}{1-\lambda}\cdot (g-\lambda \cdot 1_V)$ is~$H$-equivariant.
Thus~$W'=\ker(\pi)$ is~$H$-invariant, and is supplementary to~$W$.

Since~$W$ is arbitraty,~$V$ is semisimple as a representation of~$H$.

Finally, using~$G'\leq H\leq G$, we get
\[
Z_{GL(V)}(H)\leq Z_{GL(V)}(G')=Z_{GL(V)}(G)\leq Z_{GL(V)}(H).\qedhere
\]
\end{proof}

%Let~$G\leq GL(\ol{\F_p})$ be a reductive subgroup, and let~$H\leq G(\F_p)$ be a~$G$-cr subgroup. 
\subsubsection{}\label{AppA:def:cr}
We recall that (see~\cite[\S3.2]{SCR}) for an algebraic group~$G$ over a field~$k$, an abstract subgroup~$\Gamma\leq G(\ol{k})$ \emph{is~$G$-cr}, resp. a $k$-algebraic subgroup~$H\leq G$ \emph{is~$G$-cr} if and only if: for every $k$-parabolic subgroup~$P\leq G$ such that~$\Gamma\leq P(\ol{k})$, resp. such that~$H(\ol{k})\leq P(\ol{k})$, there exists a Levi subgroup~$L\leq P$ such that~$\Gamma\leq L(\ol{k})$, resp. such that~$H(\ol{k})\leq L(\ol{k})$.

\begin{lemma}\label{lem:Gcr}
Let~$G$ be a reductive group over ${\ol{\F_p}}$, and let~$H\leq L\leq M\leq G$ be $\ol{\F_p}$-algebraic subgroups.

Assume that~$H$ and~$M$ are~$G$-cr %\footnote{Recall that a subgroup~$L\leq G$ is~$G$-cr if, for any parabolic subgroup~$P\leq G$ with~$L\leq P$, there exists a Levi subgroup~$Q\leq P$ such that~$L\leq Q$.}  
and that
\[
Z_{G}(H)=Z_{G}(M).
\]

Then, for every parabolic subgoup~$P\leq G$ defined over~$\ol{\F_p}$, we have
\[
H\leq P\Leftrightarrow L\leq P\Leftrightarrow M\leq P.
\]
\end{lemma}
\begin{proof}
It suffices to prove that, for every parabolic subgroup~$P\leq G$, we have
\[
H\leq P\Rightarrow M\leq P.
\]
Let~$P\leq G$ be a parabolic subgroup such that~$H\leq P$. Since~$H$ is~$G$-cr, there exists, by definition,
a Levi subgroup~$Q\leq G$ such that~$H\leq Q$.  Let~$Z(Q)$ be the center of~$Q$. Then~$Z(Q)\leq Z_G(Q)\leq Z_{G}(H)=Z_{G}(M)$. Thus~$M\leq Z_G(Z(Q))$. Since~$Q$ is the Levi subgroup of a parabolic subgroup, we have~$Z_G(Z(Q))=Q$ (\cite[proof of 16.1.1., p.\,269]{SpringerLAG}. We conclude~$M\leq Z_G(Z(Q))=Q\leq P$.
\end{proof}
\begin{corollary}\label{cor:Gcr}
The subgroup~$L$ is~$G$-cr.
\end{corollary}
\begin{proof}Let~$P\leq G$ be a parabolic subgroup such that~$L\leq P$. We have~$H\leq P$, and, since~$H$ is~$G$-cr,
there exists an opposite parabolic subgroup~$P'$ such that, with~$Q=P\cap P'$, we have~$H\leq Q$. By Lem.~\ref{lem:Gcr},
we have~$M\leq P'$, and thus~$L\leq M\leq P'$. 

Thus~$L$ is contained in the Levi subgroup~$Q\leq P$, and since~$P$ is arbitrary,~$L$ is, by definition,~$G$-cr.
\end{proof}

\begin{proposition}\label{propMNich}
Let~$M\leq G\leq GL(n)$ be connected reductive algebraic subgroups defined over~$\Q$.  We denote by~$M_{\F_p}\leq G_{\F_p}$ the~$\F_p$-algebraic groups induced by the model~$GL(n)_{\Z}$.
There exists~$c(M,G)$ such for every~$p\geq c(M,G)$ the following holds. 

Let~$V\leq M(\F_p)$ and define~$V^\dagger$ defined as in~Rem.~\ref{rem:Tate} \ref{rem2}.
We also view~$V$ and~$V^\dagger $ as~$0$-dimensional algebraic subgroups of~$M_{\F_p}$.
We assume
\[
V\leq V^\dagger\cdot Z(M_{\F_p}).
\]
and~$Z_{G_{\F_p}}(V)=Z_{G_{\F_p}}(M_{\F_p})$, and that the action of~$V$ on~${\F_p}^n$ is semisimple.

Let~$W$ be the algebraic group associated to~$V^\dagger$ by Nori (\cite{N}), and~$L=W\cdot Z(M)$, and let~$\mathfrak{l}\leq\mathfrak{g}_{\F_p}$ be the Lie algebra of~$L$.

Then~$L_{\ol{\F_p}}$ is~$G_{\ol{\F_p}}$-cr and~$\mathfrak{l}\tens{\ol{\F_p}}\leq \mathfrak{g}_{\ol{\F_p}}$ is~$G_{\ol{\F_p}}$-cr (in the sense of~\cite{McN}).
\end{proposition}
\begin{proof}We know that there exists~$c'(M,G)$ such that, for~$p\geq c'(M,G)$, the algebraic groups~$M_{\F_p}$ and~$G_{\F_p}$ are reductive.
By~\cite[Th. 5.3.]{SCR}, we have, for~$p\geq h(G)$ and~$p\geq h(GL(n))$, that~$M_{\ol{\F_p}}$ and~$G_{\ol{\F_p}}$ are~$GL(n)_{\ol{\F_p}}$-cr and~$M_{\ol{\F_p}}$ is~$G_{\ol{\F_p}}$-cr.

Let~$n_p$ be the invariant~$n(\Lambda)$ of~\cite[\S\,5.2.]{SCR} for the representation~$\Lambda={\F_p}^n$ of~$G_{\F_p}$.
The definition of~$n_p$ is given in terms of weights, and, for~$p\gg0$, does not depend on~$p$. Thus~$n_{\max}:=\max_p n_p<+\infty$. Recall that the action of~$V\leq M(\F_p)$ is semisimple. By \eqref{rem4} of \S\ref{rem:Tate}, the~$\ol{\F_p}$-linear action of~$V$ on~$\Lambda\tens\ol{\F_p}$ is semisimple. By~\cite[Th. 5.4.]{SCR} this implies, for~$p\geq n_{\max}$, that~$V$ is~$G_{\ol{\F_p}}$-cr.

We apply Lem.~\ref{lem:Gcr} and Cor.~\ref{cor:Gcr} with~$H=V$ (viewed as an $\ol{\F_p}$-algebraic group of dimension~$0$), $L=L_{\ol{\F_p}}$ and~$M=M_{\ol{\F_p}}$ and~$G=G_{\ol{\F_p}}$. It follows that~$L_{\ol{\F_p}}$ is~$G_{\ol{\F_p}}$-cr. According to~\cite[Th.~1(2)]{McN}, it follows that~$\mathfrak{l}_{\ol{\F_p}}:=\mathfrak{l}\tens{\ol{\F_p}}$ is~$G_{\ol{\F_p}}$-cr.

This proves the Proposition with
\[
c(M,G)=\max\{c'(M,G);h(G);h(GL(n));n(V)\}.\qedhere\]
\end{proof}
\begin{corollary}\label{corMNich}
Let~$x_1,\ldots,x_k\in \mathfrak{gl}(n,\F_p)$ and~$y_1,\ldots,y_l\in  \mathfrak{gl}(n,\F_p)$ be such that~$V^\dagger$ is generated by~$\{\exp(x_1);\ldots;\exp(x_k)\}$ and~$\{y_1;\ldots;y_l\}$ generates~$\mathfrak{z}(\mathfrak{m})$.

Then the orbit (as an algebraic variety over~$\F_p$) of~$(x_1,\ldots,x_k,y_1,\ldots,y_l)$ in~$\mathfrak{gl}^{k+l}$ under the action of~$G_{\F_p}$ by conjugation, is Zariski closed.
\end{corollary}
\begin{proof}Let~$F=\Hom(\mathfrak{l}_{\ol{\F_p}},\mathfrak{g}_{\ol{\F_p}})$ be the vector space of ${\ol{\F_p}}$-linear maps. Let~$Z\leq F$ be the subset of Lie algebra homomorphisms. Then~$Z$ is a Zariski closed subvariety.
Let~$E:F\to \mathfrak{g}_{\ol{\F_p}}^{k+l}$ be the evaluation map~$\phi\mapsto (\phi(x_1),\ldots,\phi(x_k),\phi(y_1),\ldots,\phi(y_l))$.

Since~$\mathfrak{l}$ is generated by~$x_1,\ldots,x_k,y_1,\ldots,y_l$, the map~$E$ is injective on~$Z$. This is also a linear map.
It thus identifies the vector subspace generated by~$Z$ with a vector subspace of~$\mathfrak{g}_{\ol{\F_p}}^{k+l}$,
and identifies~$Z$ with a closed subvariety~$Z'\subseteq\mathfrak{g}_{\ol{\F_p}}^{k+l}$.

Let~$G\cdot \phi$ be the conjugacy class of the inclusion~$\phi:\mathfrak{w}\to\mathfrak{g}$. According to \cite[Th.~1(1)]{McN}, this is a closed subvariety in~$L$. As it is contained in~$Z$, this is a closed subvariety of~$Z$. Its image in~$Z'$
is thus a closed subvariety of~$Z'$. This image is~$E(G\cdot \phi)=G\cdot (x_1,\ldots,x_k,y_1,\ldots,y_l)$. Since~$Z'$ is closed in~$\mathfrak{g}_{\ol{\F_p}}^{k+l}$, the conjugacy class~$G\cdot (x_1,\ldots,x_k,y_1,\ldots,y_l)$ is closed in~$\mathfrak{g}_{\ol{\F_p}}^{k+l}$.
\end{proof}

\section{Consequences of Uniform integral Tate property for $\ell$-independence}\label{sec:AppB}

Throughout App.~\ref{sec:AppB}, we fix reductive~$\Q$-algebraic groups~$M\leq G\leq GL(n)$.

The following is proved in~\S\ref{AppB:sec:proof}. 

\begin{theorem}\label{appB:thm}
Let~$U\leq M(\widehat{\Z})$ be a compact subgroup satisfying Def~.\ref{defi:Tate}.
Then
\begin{equation}\label{B1eq}
\exists e\in \Z_{\geq 1}, \forall u\in U, u^e\in \left(\prod_p U\cap M^{der}(\Z_p)\right)\cdot \left(U\cap Z(M)(\widehat{\Z})\right).
\end{equation}
\end{theorem}

We deduce the following.
%For~$H\leq GL(n,\F_p)$ resp.~$U\leq GL(n,\Z_p)$ we denote by~$H^\dagger$, resp.~$U^\dagger$ the subgroup generated by~$\{h\in H|\exists k\in \Z_{\geq 1},h^{(p^k)}=1\}$, resp the closed group generated by~$\{u\in U|\lim_{ k\in \Z_{\geq 1}},u^{(p^k)}=1\}$.
\begin{corollary}\label{AppB:main cor}
 Let~$U$ be as in Theorem~\ref{appB:thm}, let~$W$ be the image of~$U$ by~$ab_M:M(\A_f)\to M^{ab}(\A_f)$ with~$\A_f=\widehat{\Z}\tens\Q$, let~$W_p:=W\cap M^{ab}(\Q_p)$ and assume that
\begin{equation}\label{hypothese Cor B2}
\exists f, \forall w\in W,w^f\in \prod_p W_p.
\end{equation}
Then
\(
\exists e\in \Z_{\geq 1}, \forall u\in U, u^e\in \prod_p \left(U\cap M^{der}(\Z_p)\right)\cdot \left(U\cap Z(M)({\Z_p})\right).
\)
\end{corollary}
\begin{proof}[Proof of Corollary~\ref{AppB:main cor}]
For~$d\in\Z_{\geq1}$ and an abelian group~$A$, let~$[d]:A\to A$ denote the multiplication-by-$d$ map~$a\mapsto d\cdot a:=a+\ldots+a$. For any group homomorphism~$A\to B$ such that~$\ker(\phi)\leq \ker([d])$, we have
\begin{equation}\label{exists psi}
\exists\psi:\phi(A)\to A,~\psi\circ \phi=[d]
\end{equation}
because we can factor~$[d]:A\to \phi(A)\to A$ using 
\[
A\to \phi(A)\simeq A/\ker(\phi)\rightarrow A/\ker([d])\simeq [d](A)\hookrightarrow A.
\]

Let us denote by~$\phi_p:Z(M)(\Z_p)\to M^{ab}(\Q_p)$ and~$\phi=(\phi_p)_p:Z(M)(\widehat{\Z})\to M^{ab}(\Q\tens\widehat{\Z})$ the abelianisation maps. Because~$Z(M)\cap M^{der}$ is a finite algebraic group, we have
\[
\abs{\ker(\phi_p)}= \abs{Z(M^{der})(\Z_p)}\text{ divides }d:=\abs{Z(M^{der})(\ol{\Q})}<+\infty.
\]
We deduce~$[d]\left(\ker(\phi_p)\right)=\{1\}$. By~\eqref{exists psi} there exists~$\psi_p:\phi_p(Z(M)(\Z_p))\to Z(M)(\Z_p)$ such that~$\psi_p\circ \phi_p=[d]$. Thus, with~$\psi:=(\psi_p)_p$, we have~$\psi\circ\phi=[d]$.

The kernel of~$ab_M$ is~$M^{der}(\A_f)$. Thus, with~$U':=U\cap Z(M)(\widehat{\Z})$ and~$W'=\phi(U')$, we have
\[
ab_M\left(\left(\prod_p U\cap M^{der}(\Z_p)\right)\cdot \left(U\cap Z(M)(\widehat{\Z})\right)\right)=ab_M(U')=W'.
\]
Let~$e$ be as in Theorem~\ref{appB:thm}. Then we have~$[e](W)\leq W'$ and thus~$[e](W_p)\leq [e](W)\cap M^{ab}(\Q_p) \leq W'_p:=W'\cap M^{ab}(\Q_p)$. We deduce from~\eqref{hypothese Cor B2} that
\[
[e\cdot f](W')\leq 
[e\cdot f](W)\leq [e]\left(\prod_p W_p\right)\leq \prod _p W'_p\leq W'.
\]
Applying~$\psi$, we deduce
\[
[e\cdot f\cdot d](U')\leq [f\cdot d]\left(\prod_p\psi_p(W'_p)\right)\leq [d](U').
\]
As~$\psi_p(W'_p)\leq \psi(W)\cap Z(M)(\Q_p)\leq U'_p:=U'\cap Z(M)(\Q_p)$, we have
\begin{equation}\label{BCorfactor}
[e\cdot f\cdot d](U')\leq \prod_p U'_p.
\end{equation}

By Theorem~\ref{appB:thm}, for~$u\in U$, we can write~$u^e=(m_p)_p\cdot u'$ with~$u'\in U'$ and~$m_p\in U\cap M^{der}(\Z_p)$. By~\eqref{BCorfactor}, we can write~${u'}^{e\cdot f\cdot d}=(u'_p)_p$ with~$u'_p\in U'_p\leq U\cap Z(M)(\Z_p)$. We conclude
\[
u^{e\cdot e\cdot f\cdot d}=(m_p^f)_p\cdot (u'_p)_p\in \prod_p \left(U\cap M^{der}(\Z_p)\right)\cdot \left(U\cap Z(M)({\Z_p})\right).
\]
The Corollary~\ref{AppB:main cor} follows.
\end{proof}
%\begin{proof}Let~$U_p=U\cap M^{der}(\Q_p)$.
%By Theorem~\ref{appB:thm}, we may assume~$U= \left(\prod_p U_p\right)\cdot \left(U\cap Z(M)(\widehat{\Z})\right)$. 
%We have then a surjection~$U\cap Z(M)(\widehat{\Z})\to W$. Its kernel is
%\[
%K:=U\cap Z(M)(\widehat{\Z}) \cap M^{der}(\widehat{\Z})=U\cap F
%\]
%where~$F=Z(M)(\A_f)\cap M^{der}(\A_f)$.  Let~$d=\# Z(M)(\ol{\Q})\cap M^{der}(\ol{\Q})$. Then~$\forall f\in F, f^d=1$ and thus~$\forall k\in K, k^d=1$.
%
%%Any~$k\in K$ can be writen~$k=z=f\cdot u$ with~$z\in Z(M)(\widehat{\Z})$ and~$f\in F$ and~$u\in \prod_p U_p$. We have~$u=zf^{-1}\in Z(M)(\widehat{\Z})$, and thus~$u\in \prod_p Z_p$.
%%\[
%%\frac{K}{\prod_p Z_p}
%%=\frac{U\cap (F\cdot \prod_p Z_p)}{\prod_p Z_p}
%%\]
%%is a subgroup of the~$(F\cdot \prod_p Z_p)/\prod_p Z_p\simeq F/(F\cap \prod_p Z_p)$. Let~$d=\# Z(M)(\ol{\Q})\cap M^{der}(\ol{\Q})$. Then~$\forall f\in F, f^d=1$ and thus~$\forall k\in K, k^d=1$. We have an extension of abelian groups
%%\[
%%\frac{K}{\prod_p Z_p}\to \frac{U\cap Z(M)(\widehat{\Z})}{\prod_p Z_p}\to \frac{W}{Y}.
%%\]
%%As~$K$ is killed by~$d$... wma~$K=\{1\}$.
%%
%%\[
%%\prod_p Z_p\to U\cap Z\to \prod_p W_p/Y_p.
%%\]
%\end{proof}

\subsubsection{}\label{p0fact}
 Let us recall the following fact. By~\cite[\S3.9.1, p.\,55]{Tits}  there exists~$p_0(M,G)$ such that for every prime~$p\geq p_0(M,G)$, the induced~$\F_p$-algebraic groups~$M_{\F_p}\leq G_{\F_p}\leq GL(n)_{\F_p}$ are reductive  and that~$Z(M)_{\Z_p}$ is smooth over~$\Z_p$ and~$Z(M)_{\F_p}$ is a torus.

\subsection{Some properties over~$\F_p$}
\begin{definition}\label{def: Tate Fp}
For a subgroup~$H\leq M(\F_p)$ and~$C\in\Z_{\geq0}$,
let~$Tate(H,C)$ be the following property: for every subgroup~$H'\leq H$ such that~$[H:H']\leq C$, 
\begin{enumerate}[label={(T\arabic*)}]
\item the action of~$H'\leq GL(n,\F_p)$ on~${\F_p}^n$ is semisimple \label{Tate Fp semisimplicity}
\label{AppBT1}
\item and~$Z_{G_{\F_p}}(H')=Z_{G_{\F_p}}(M_{\F_p})$.
\label{AppBT2}
\end{enumerate}
\end{definition}
By \eqref{rem4} of \S\ref{rem:Tate}, we have~\ref{AppBT1} if and only if the action of~$H'$ on~$\ol{\F_p}^n$ is semisimple.
By~\cite[Th.\,5.4]{SCR}, there exists~$n(M,G)$ such that, for~$p\geq n(M,G)$, we have~\ref{AppBT1} $\Leftrightarrow$~$H'$ is~$G_{\ol{\F_p}}$-cr~$\Leftrightarrow$~$H'$ is~$M_{\ol{\F_p}}$-cr, in the sense of~\S\ref{AppA:def:cr}. If~$p\geq \max\{p_0(M,G);n(M,G)\}$, then, by~\cite[Th.\,5.4]{SCR}, any~$H'$ satisfying~\ref{AppBT1} and~\ref{AppBT2} is not contained in a proper parabolic subgroup~$P\leq M_{\ol{\F_p}}$.

\begin{lemma}\label{TateFp monter}
Let~$p\geq p_1(M,G):=\max\{p_0(M,G);n(M,G)\}$, let~$H\leq L\leq M(\F_p)$ be subgroups and let~$C\in\Z_{\geq0}$ be such that the property~$Tate(H,C)$ holds. Then the property~$Tate(L,C)$ holds.
\end{lemma}
\begin{proof}We prove~$Tate(L,C)$, assuming~$p\geq p_1(M,G)$ and~$Tate(H,C)$. Let~$L'\leq L$ be such that~$[L:L']\leq C$.
Then~$H':=H\cap L'$ verifies~$[H:H']\leq C$. As~$H'\leq L'\leq M(\ol{\F_p})$, we have~$Z_{G_{\F_p}}(H')\geq Z_{G_{\F_p}}(L')\geq Z_{G_{\F_p}}(M_{\F_p})$.
Using~\ref{AppBT2}, we deduce~$Z_{G_{\F_p}}(M_{\F_p})=Z_{G_{\F_p}}(L')$. This proves~\ref{AppBT2} for~$L'$. Note that~$G_{\F_p}$ is reductive since~$p\geq p_0(M,G)$
and that~$H'$ is~$G_{\ol{\F_p}}$-cr since~$p\geq n(M,G)$. We may apply Lemma~\ref{lem:Gcr} for~$H'\leq L'\leq M_{\ol{\F_p}}\leq G_{\ol{\F_p}}$. By~Cor~\ref{cor:Gcr}, the group~$L$ is~$G_{\ol{\F_p}}$-cr. As~$p\geq n(M,G)$, this proves~\ref{AppBT1} for~$L'$. As~$L'$ is arbitrary, this proves~$Tate(L,C)$.
\end{proof}

%We recall that, for a group~$\Gamma\leq GL(n,\F_p)$, we denote by~$\Gamma^\dagger$ the subgroup generated by~$\{\gamma\in\Gamma|\exists k\in\Z_{\geq1}, \gamma^{(p^k)}=1\}$.
\begin{lemma} \label{lem:B1}
There exist~$a_0,a_1:\Z_{\geq1}\to \Z_{\geq1}$ such that the following holds.
Let~$p\geq a_1(n)$ and~$\Gamma\leq M(\F_p)\leq GL(n,\F_p)$ be such that property~$Tate(\Gamma,a_0(n))$ holds,
and define~$\Gamma^\dagger$ as in \S\ref{rem:Tate}\eqref{rem2}.

%\begin{itemize}
%\item Tate 1: the~$\Gamma$-module~${\F_p}^n$ is semisimple.
%\item Tate 2: for every~$C$, there exists~$M(C)$ such that if~$\Gamma'\leq\Gamma$ is a subgroup of index less than~$C$
%and~ if~$p>M(C)$, then ...
%\end{itemize}
Then~$\Lambda:=\Gamma^\dagger\cdot (\Gamma\cap Z(M)(\F_p))$ satisfies
\[
[\Gamma:\Lambda]\leq a_0(n)
\]
and~$\Gamma^\dagger=H(\F_p)^\dagger$ for a semisimple~$\F_p$-algebraic group~$H$.
\end{lemma}
The proof will use the the following facts.
\begin{itemize}
\item By~\cite[No\,137, \S1 (p.\,38), cf.~Rem.\,38.3 (p.\,667) and No\,164]{S4}, there exists~$c_2(n)$ such that for every prime~$p$, there are at most~$c_2(n)$ $GL(n,\ol{\F_p})$-conjugacy classes of semisimple subgroups~$S\leq GL(n)_{\ol{\F_p}}$, and they "come from characteristic~$0$" for~$p\gg0$.
\item In particular, there exists~$c_3(n)$ such that, 
\begin{equation}\label{p nmid Z}
\forall p\geq c_3(n),~p\nmid \#Z(S)(\ol{\F_p}).
\end{equation}
\end{itemize}
By Lang's Theorem \cite[Prop.\,6.3, p.\,290]{PR},~$\abs{ S^{ad}(\F_p)/ad(S(\F_p))}=\abs{Z(S)({\F_p})}$. For~$p\geq c_3(n)$, this implies that~$p\nmid \# S^{ad}(\F_p)/ad(S(\F_p))$. By~Cor.~\ref{Sylow} this implies that
\[
\forall p\geq c_3(n),
ad(S(\F_p)^{\dagger})=ad(S(\F_p))^{\dagger}=S^{ad}(\F_p)^\dagger.
\]
\begin{itemize}
\item We have~$Aut(S)^0=S^{ad}$ and~$\pi_0(Aut(S))$ is finite. In particular, there exists~$c_4(n)$ such that we have~$\#\pi_0(Aut(S))\leq c_4(n)$.
\item We have~$\#S^{ad}(\F_p)/S^{ad}(\F_p)^\dagger\leq 2^n$. (\cite[(3.6(v)), p\,270]{N})
\end{itemize}
\begin{proof}
According to~\cite[Th. B]{N}, there exists~$c_1(n)$ such that for~$p\geq c_1(n)$, we have~$\Gamma^\dagger=S(\F_p)^\dagger$ where~$S\leq GL(n,\F_p)$ is a~$\F_p$-algebraic group generated by connected unipotent subgroups. From Nori's construction, the algebraic group~$S$ is invariant under the adjoint action of~$\Gamma$. We denote this action by
\[
\alpha:\Gamma\to Aut(S)(\F_p).
\]
By definition~$\ker(\alpha)$ is the centraliser~$Z_{\Gamma}(S)$. 

By~\cite[Th. 5.3]{SCR} and property~\ref{AppBT1} for~$H'=\Gamma$, the algebraic group~$S$ is semisimple.% is semisimple for~$p\gg0$ (cf Serre).

From now on, let us assume~$p\geq a_1(n):=\max\{c_1(n);c_3(n)\}$. The facts preceding the proof imply that, for~$\Lambda':=\stackrel{-1}{\alpha}(S^{ad}(\F_p)^\dagger)$, we have
\[
[\Gamma:\Lambda']\leq 
\#Aut(S)(\F_p)/ad(S(\F_p)^\dagger)\leq c_5(n):=c_4(n)\cdot 2^n.
\]
By construction~$\ker(\alpha)\leq \Lambda'$ and~$S(\F_p)^\dagger\leq \Lambda'$. Note that~$\Gamma^\dagger=S(\F_p)^\dagger\to ad(S(\F_p)^\dagger)=S^{ad}(\F_p)^\dagger$ is surjective: for~$\lambda\in \Lambda'$, there is~$s\in \Gamma^\dagger$ such that~$\alpha(s)=\alpha(\lambda)$. Equivalently~$\lambda\in\ s\cdot \ker(\alpha)$. This proves
\[
\Lambda'=\Gamma^\dagger\cdot \ker(\alpha).
\]
From~$\ker(\alpha)\leq \Lambda'\leq \Gamma$ we deduce~$\ker(\alpha)^\dagger\leq \Gamma^\dagger\leq S$.
Thus~$\ker(\alpha)^\dagger\leq \ker(\alpha)\cap S\leq Z_\Gamma(S)\cap S\leq Z(S)$. By~\eqref{p nmid Z} we have~$\ker(\alpha)^\dagger=\{1\}$. Let~$n\mapsto d(n)$ be as in~\cite[\S4]{SCrit}. By Jordan's theorem~\cite[\S5.2.2]{SCrit},
there exists an abelian subgroup~$K\leq \ker(\alpha)$ such that~$[\ker(\alpha):K]\leq d(n)$. Thus~$\Lambda'':=\Gamma^\dagger\cdot K$ satisfies~$[\Gamma:\Lambda'']\leq a_0(n):=c_5(n)\cdot d(n)$.

Let us assume~$Tate(\Gamma,a_0(n))$. By~\ref{AppBT2} for~$H':=\Lambda''$, we have
\[
Z(M_{\F_p})\leq Z_{M_{\F_p}}(\Lambda'')= Z_{G_{\F_p}}(\Lambda'')\cap M_{\F_p}=Z_{G_{\F_p}}(M_{\F_p})\cap M_{\F_p}=Z(M_{\F_p}).
\]
As~$K$ is abelian and~$K\leq \ker(\alpha)$, we have
\[
K\leq Z_{M_{\F_p}}(K)\cdot Z_{M_{\F_p}}(S)=Z_{M_{\F_p}}(K\cdot S)\leq  Z_{M_{\F_p}}(K\cdot \Gamma^\dagger)=Z_{M_{\F_p}}(\Lambda'')=Z(M_{\F_p}).
\]
Thus~$\Lambda''=\Gamma^\dagger\cdot K\leq \Lambda:=\Gamma^\dagger\cdot (\Gamma\cap Z(M))$. Thus~$[\Gamma:\Lambda]\leq [\Gamma:\Lambda'']\leq a_0(n)$.
\end{proof}
%\paragraph{Remark} In~Lemma~\ref{lem:B1}, after replacing~$a_0(n)$ by~$\max\{a_0(n);a_1(n)^{n^2}\}$, we may replace~$a_1(n)$ by~$1$.
%\begin{lemma}\label{appB:lem:affine}
%There exists~$n(M)$ such that for~$p\geq n(M)$ and for any~$H'\leq M(\F_p)$ satisfying~\ref{AppBT1} and~\ref{AppBT2}
%\begin{enumerate}
%\item \label{appB:lem:affine1}
% any affine subspace~$A\leq \mathfrak{m}_{\F_p}$ has an~$H'$-fixed element;
%\item \label{appB:lem:affine2} any $H'$-fixed element of~$\mathfrak{m}_{\F_p}$  belongs to~$\mathfrak{z}(\mathfrak{m}_{\F_p})$.
%\end{enumerate}
%\end{lemma}
%\begin{proof}By~\cite[Th.\,5.4(i)]{SCR}, for~$p\geq n({\F_p}^n)$,~\ref{AppBT1} implies that~$H'$ is~$M_{\ol{\F_p}}$-cr, and by~\cite[Th.\,5.4(i)]{SCR}, for~$p\geq n(M):=\max\{n(V);n({\F_p}^n)\}$
%this implies that action of~$\ad_{M_{\F_p}}(H')$ on~$V:=\mathfrak{m}_{\F_p}$ is semisimple. We can write~$A=v+W$ with~$v\in V$ and~$W\leq V$ a vector subspace. Then~$W$ is stable under the action of~$\ad_{M_{\F_p}}(H')$ and there is a supplementary $\ad_{M_{\F_p}}(H')$-stable vector subspace~$W'$. Then~$A\cap W'$ is of the form~$v'+W\cap W'$. As ~$A\cap W'$ is stable and~$W\cap W'=\{0\}$, the element~$v'$ of~$A$ is fixed by~$\ad_{M_{\F_p}}(H')$.
%This proves~\ref{appB:lem:affine1}.
%
%By~\ref{AppBT2}, we have~$Z_{M_{\F_p}}(H')=Z(M_{\F_p})$. By Lem.~\ref{conj orbit lemma}, for~$p\geq c_3(\dim(M))$,
%the centraliser~$Z_{M_{\F_p}}(H')$ is smooth. 
%The conclusion~\eqref{appB:lem:affine2} follows.
%\end{proof}
\subsection{Independence properties over~$\prod_p\F_p$}\label{appB:sec:R}
Let~$R:=\prod_p\F_p$, and let~$W\leq M(R)$ be a closed subgroup (for the Tychonov product topology). We denote
by~$V(p)\leq M(\F_p)$ the image of~$U$ by the projection~$M(R)\to M(\F_p)$, and we define~$V=\prod_p V(p)$ and~$W(p):=M(\F_p)\cap W\leq V(p)$.

Let~$Tate(W)$ be the following property: there exists~$m=m_W:\Z_{\geq1}\to\Z_{\geq1}$ such that for every~$c\in\Z_{\geq1}$ and every~$p\geq m(c)$ we have~$Tate(W(p),c)$. Then~$Tate(W)$ implies $Tate(V)$ by Lemma~\ref{TateFp monter}.

\begin{proposition}\label{app:prop R}
Assume~$Tate(V)$ as defined above.

Then there exists~$p(W)\in\Z_{\geq1}$ such that,
\begin{equation}\label{app:eq:R:1}
\forall p\geq p(W), V(p)^\dagger\leq W
\end{equation}
and
\begin{equation}\label{eq:e}
\exists e\in\Z_{\geq 1}, \forall w\in W, w^e\in Y:= \left(\prod_{p\geq p(W)}V(p)^\dagger\right)\cdot (W\cap Z(M)(R)).
\end{equation}
\end{proposition}
Proposition~\ref{prop:B7}, used below, will be proved in~\S\ref{sec:B8}.
\begin{proof}Let us define~$R':=\prod_{\ell\neq p}\F_\ell$. 
We denote by~$W^{R'}$ the image of~$W$ by the projection~$GL(n,R)\to GL(n,R')$ and we define~$W_{R'}:=W\cap GL(n,R')\leq W^{R'}$. 
We apply Goursat's Lemma to~$W\leq GL(n,R)=GL(n,\F_p)\times GL(n,R')$: there exists an isomorphism
\[
V(p)/W(p)\to W^{R'}/W_{R'}.
\]
Let us define~$p(W):=\max\{5;d(n);c(n)^{(n^2)};a_1(n);m_W(a_0(n))\}$ with~$c(n)$ as in Prop.~\ref{prop:B7}. 
For~$p\geq p(W)$, Lemma~\ref{lem:B1} implies that we have~$V(p)^\dagger=H(\F_p)^\dagger$ for a semisimple~$H\leq GL(n)_{\F_p}$, and Prop.~\ref{prop:B7} implies that the morphism~$V(p)^\dagger\to V(p)/W(p)\to W^{R'}/W_{R'}$ is trivial. Thus~$V(p)^\dagger\leq W(p)\leq W$. We proved~\eqref{app:eq:R:1}.

Let us define~$\Lambda(p):=V(p)^\dagger\cdot (Z(M)(\F_p)\cap V(p))$ and~$\Lambda:=\prod_p\Lambda(p)$. For~$p\geq c_0(n):=\max\{m(a_0(n));a_1(n)\}$ we apply Lemma~\ref{lem:B1} for~$\Gamma=V(p)$ and deduce~$[V(p):\Lambda(p)]\leq a_0(n)$. For~$p<c_0(n)$ we have~$[V(p):\Lambda(p)]\leq \#GL(n,\F_p)\leq c_0(n)^{n^2}$. Note that~$\Lambda(p)\leq V(p)$ is a normal subgroup. We deduce, with~$e':=\max\{a_0(n);c_0(n)^{n^2}\}$, that
\(
\forall p, \forall v\in V(p), v^{e'}\in \Lambda(p).
\)
It follows that
\begin{equation}\label{eq:e'}
\forall w\in W, w^{e'}\in W':=W\cap \Lambda.
\end{equation}

Because~$V(p)^\dagger\leq W$ and~$V(p)^\dagger\leq \Lambda(p)\leq \Lambda$, we have~$V(p)^\dagger\leq W'$. As~$\Lambda$ and~$W$ are closed, we deduce that
\begin{equation}\label{eq:Btruc}
\prod_{p\geq p(W)} V(p)^\dagger\leq W'.
\end{equation}

Let~$X:=\ker\left( W'\to \prod_{p\leq p(W)} GL(n,\F_p)\right)$.  Then
\begin{equation}\label{eq:e''}
[W':X]\leq e'':=\# \prod_{p\leq p(W)} GL(n,\F_p)\leq (p(W)!)^{(n^2)}<+\infty.
\end{equation}
Let~$x\in X$ be arbitrary and let~$x=(\lambda_p)_p$ be its coordinates in~$\prod_p\Lambda(p)$. For~$p\geq p(W)$ we can write~$\lambda_p=v_p\cdot z_p$ with~$v_p\in V(p)^\dagger$ and~$z_p\in Z(M)(\F_p)$.
For~$p< p(W)$ we define~$v_p=z_p=1$. Let~$v:=(v_p)_p$ and~$z=(z_p)_p$. By~\eqref{eq:Btruc} we have~$v\in X$. Thus~$z=x\cdot v^{-1}\in X$. Thus~$X\leq Y$. Together with~\eqref{eq:e'} and~\eqref{eq:e''}, this implies~\eqref{eq:e} with~$e:=e'\cdot e''$.
\end{proof}

\subsubsection{}\label{sec:B8}
We will prove~Prop.~\ref{prop:B7},  using arguments from~\cite{SCrit}.
%\begin{lemma}\label{appB:lem:perfect}
%Let~$p$ be a prime.
%Let~$1\to F\to \Gamma\to \Sigma\to 1$ be a short exact sequence of finite groups such that~$F^\dagger=\{1\}$ and~$\Gamma^\dagger=\Gamma$ and~$\Sigma$ is simple and not abelian. Then~$\Gamma$ has no non trivial abelian quotient.
%\end{lemma}
%\begin{proof} Let~$q:\Gamma\to A$ be an abelian quotient and~$N$ be the kernel of~$q$.  and~$M$ the image of~$N$ in~$\Sigma$. Because~$\Sigma$ is simple, the image~$M$ of~$N$ in~$\Sigma$ is~$M=\{1\}$ or~$M=\Sigma$.
%
%If~$M=\{1\}$, then~$N\leq F$, and we can factor~$\Gamma\to A\to \Sigma$. As~$\Sigma$ is not abelian, this is not possible.
%
%Thus~$M=\Sigma$. Equivalently,~$\Gamma=F\cdot N$. We thus have a isomorphism~$F/(F\cap N)\to A$. We recall that~$\Gamma=\Gamma^\dagger$. This implies~$A=A^\dagger$. We recall that~$F^\dagger=\{1\}$. By~Cor.~\ref{Sylow}, this implies~$A^\dagger=\{1\}$. We conclude~$A=\{1\}$.
%\end{proof}
\begin{proposition}\label{app:cor:ab quot}
Let~$H\leq GL(n)_{\F_p}$ be a semisimple algebraic group.
If~$5\leq p$, then~$\Gamma:=H(\F_p)^\dagger$ has no non-trivial abelian quotient. 
\end{proposition}
%GEO SIMPLE?
\begin{proof}
Let us write~$H=H_1\cdot \ldots\cdot H_c$ as an almost direct product of its quasi-simple factors~$H_1,\ldots,H_c$. 
 For~$p\geq5$, \cite[Th.\,2.2.7, p\,38]{GLS} implies that, for~$i=1,\ldots,c$, the group~$\Gamma_i:=H_i(\F_p)^\dagger$ is quasisimple: in particular~$\Gamma_i$ is its own derived subgroup. We deduce that~$\Gamma_1\cdot\ldots\cdot \Gamma_c$ is its own derived subgroup.
By \cite[Prop 2.2.11, p.\,40]{GLS}, we have~$\Gamma=\Gamma_1\cdot\ldots\cdot \Gamma_c$. 
%As~$p\geq5$,  the quotient~$\Gamma_i$ is quasisimple: In particular$\Gamma_i$ is its own derived subgroup.  (\cite[\S6.1]{SCrit} \cite[Th.\,2.2.7, p\,38]{GLS}).  As~$p\nmid \#F$, we have~$F^\dagger=\{1\}$. By Lemma~\ref{appB:lem:perfect}, this implies that~$\Gamma_i$ is its own derived subgroup. We deduce that~$\Gamma_1\cdot\ldots\cdot \Gamma_c$ is its own derived subgroup. Finally~$\Gamma=\Gamma_1\cdot\ldots\cdot \Gamma_c$ by \cite[Prop 2.2.11, p.\,40]{GLS}
\end{proof}
\begin{lemma}\label{lem:product subquotient}
Let~$G=G_1\times\ldots \times G_c$ be a product of groups, let~$H\leq G$ be a subgroup, and let~$q:H\to \Sigma$ be a simple quotient. Then there exists~$i\in\{1;\ldots;c\}$ and~$H_i\leq G_i$ such that~$\Sigma$ is a quotient of~$H_i$.
\end{lemma}
\begin{proof}We may assume~$c=2$ by induction. The projection~$p:G_1\times G_2\to G_2$ induces a short exact sequence~$1\to H_1\to H\to H_2\to 1$ where~$H_1:=H\cap G_1$ and~$H_2:=p(H)$. Because~$\Sigma$ is simple, we have~$q(H_1)=\Sigma$ or~$q(H_1)=\{1\}$. In the first case, we may take~$i=1$. In the other case, we can factor~$H\to p(H)\to \Sigma$, and we take~$i=2$.
\end{proof}
\begin{proposition}\label{prop:B7}
Let~$H\leq GL(n)_{\F_p}$ be a semisimple algebraic group, and assume~$p> \max\{d(n);c(n)^{(n^2)};5\}$
with~$c(n)$ be as in~\cite[\S6.3, Th.\,4]{SCrit}. Let~$G=GL(n,\F_{\ell_1})\times\ldots \times GL(n,\F_{\ell_c})$ where~$p\not\in\{\ell_1;\ldots;\ell_c\}$. Let~$S/N$ be a quotient of a subgroup~$S\leq G$ and~$\phi:\Gamma\to S/N$ be an homomorphism (of abstract
 groups).  
Then~$\phi(\Gamma)=\{1\}$.
\end{proposition}
\begin{proof}
Let~$\Lambda:=\phi(\Gamma)$ and assume by contradiction that~$\Lambda\neq 1$. Then there exists a simple quotient~$\Lambda\to \Sigma$. 
%Let us prove the following fact.
%Let~$G=G_1\times\ldots G_c$ be a product of groups, let~$H\leq G$ be a subgroup, and let~$q:H\to \Sigma$ be a simple quotient. Then there exists~$i\in\{1;\ldots;c\}$ and~$H_i\leq G_i$ such that~$\Sigma$ is a quotient of~$H_i$.
Lemma~\ref{lem:product subquotient} implies that~$\Sigma$ is a subquotient of~$GL(n,\F_{\ell})$ for some~$\ell\in\{\ell_1;\ldots;\ell_c\}$.

According to Jordan's theorem~\cite[\S4]{SCrit}, there is a sequence of normal subgroups~$\Lambda^\ddagger\triangleleft \Lambda'\triangleleft \Lambda$ with~$\#\Lambda/\Lambda'\leq d(n)$ and~$\Lambda'/\Lambda^\ddagger$ abelian, where~$\Lambda^\ddagger$ denotes he subgroup generated by~$\{\lambda\in\Lambda|\exists k\in\Z_{\geq1} \lambda^{(\ell^k)}=1\}$.

As~$\Gamma=\Gamma^\dagger$, every non-trivial quotient~$\Gamma\to \Gamma/N$ satisfies~$p\mid \# \Gamma/N$. As~$p>c(n)^{(n^2)}$, we have~$\ell>c(n)$. As~$p>d(n)$, we have~$\Lambda=\Lambda'$.  By~Cor.~\ref{app:cor:ab quot},~$\Gamma$ has non non trivial abelian quotient. Thus~$\Lambda=\Lambda^\ddagger$. 
We deduce that every non-trivial quotient~$\Lambda\to \Lambda/N$ satisfies~$\ell\mid \# \Lambda/N$.

According to~\cite[\S6, Th.\,6.4]{SCrit}, $\Sigma$ is in~$\Sigma_\ell$ (in the sense of~\cite[\S6]{SCrit}). According to~\cite[\S6, Lem.\,6.1]{SCrit}, ~$\Sigma$ and is abelian or in~$\Sigma_p$.
According to~Prop.~\ref{app:cor:ab quot},~$\Sigma$ is not abelian. By assumption~$p\neq \ell$ and thus, by~\cite[\S6.4]{SCrit},~$\Sigma$ cannot be in both~$\Sigma_p$ and~$\Sigma_\ell$. This leads to a contradiction.
\end{proof}

\subsection{Proof of Theorem~\ref{appB:thm}}\label{AppB:sec:proof} 
Let~$U$ be as in Theorem~\ref{appB:thm}. 
Let us write~$U_p:=U\cap M(\Z_p)$. Let~$\pi_R:GL(n,\widehat{\Z})\to GL(n,R)$ be the map induced by~$\widehat{\Z}\to R:=\prod_p\F_p$, and let us define~$U(R):=\pi_R(U)$.
Def.~\ref{defi:Tate}\eqref{defi:Tate2} implies $Tate(\pi_R(\prod_p U_p))$ in the sense of~\S\ref{appB:sec:R}.
 Using Lemma~\ref{TateFp monter}  we deduce~$Tate(U(R))$. By Prop.~\ref{app:prop R} for~$W=U(R)$ and~$p_0\geq p(W)$ and~$U'$ the inverse image of~$\prod_{p\geq p_0} V(p)^\dagger\cdot \left(U(R) \cap Z(M)(R)\right)$ in~$U$, we have~$\exists e\in \Z_{\geq1} \forall e\in U, u^e\in U'$. It is thus enough to prove~\eqref{B1eq} for~$u\in U'$. Furthermore, for every~$u\in U_p$ we have~$u^e\in U'_p:=U'\cap M(\Q_p)$. By~Cor.~\ref{cor:burnside}, the group~$U'$ satisfies Def.~\ref{defi:Tate}. We may thus substitute~$U$ with~$U'$ in Theorem~\ref{appB:thm}, and we can thus assume from now on that 
\[
U(R)=\prod_{p\geq p_0} V(p)^\dagger\cdot \left(U(R) \cap Z(M)(R)\right).
\]
The Theorem~\ref{appB:thm} will be a consequence of~\eqref{factors as der times z} and~\eqref{factors by p}.
\begin{lemma}\label{lem:burnside}
 For every~$e,n\in\Z_{\geq1}$, there exists~$k(e,n)\in\Z_{\geq1}$ such that for every prime~$p$ and every compact subgroup~$K\leq GL(n,\Z_p)$, and every~$L\leq K$ such that~$\forall k\in K,k^e\in L$, we have
\[
[K:L]\leq k(e,n)
\]
\end{lemma}
\begin{proof} It follows from~\cite[Prop~6.7 and Lem. 6.8 (2)]{RY}, there exists~$c(n)$ such that~$K$ is topologocally generated by at most~$c(n)$ elements. From the restricted Burnside problem~\cite[\S1.1 Introduction]{Vaugh}, we can take~$k(e,n)=R(e,c(n))$, with~$R(r,n)$ as in loc. cit.
\end{proof}
\begin{corollary}\label{cor:burnside}
Let~$V\leq U$ be a closed subgroup such that~$\exists e, \forall u\in U, u^e\in V$.
If~$U$ satisfies Def.~\ref{defi:Tate}, then~$U'$ satisfies Def.~\ref{defi:Tate}.
\end{corollary}
\begin{proof}Note that~$\forall u\in U_p:=U\cap M(\Q_p)$, we have~$u^e\in V_p:=V\cap M(\Q_p)$. By Lemma~\ref{lem:burnside} we have~$[U_p:V\cap M(\Q_p)]\leq k(e,n)$ and~\S\ref{rem:Tate} \eqref{rem(9)} implies the conclusion.
\end{proof}
\subsubsection{}
Let~$p$ be a prime an let~$Y=Y(p)$ be the image of~$U$ by~$M(\widehat{\Z})\to M(\Q_p)\to M^{ad}(\Q_p)$.
Let~$W\leq Y$ be a closed subgroup which is invariant under conjugation by~$U$ and such that~$Y/W$ is abelian. We note that the conjugation action of~$U$ on~$Y/W$ is trivial.
We claim that
\begin{equation}\label{preuveB1claim finite}
\abs{Y/W}<+\infty.
\end{equation}
\begin{proof}
Let~$\mathfrak{w}\leq \mathfrak{y}\leq \mathfrak{m}^{ad}_{\Q_p}$ be the Lie~$\Q_p$-algebras of the~$p$-adic Lie groups~$W\leq Y\leq M^{ad}(\Q_p)$. Then~$\mathfrak{y}/\mathfrak{w}$ is a subquotient of~$\mathfrak{m}^{ad}_{\Q_p}$ as a representation of~$U$, and~$U$ acts trivially on~$\mathfrak{y}/\mathfrak{w}$. By~\eqref{defi:tate eq 1.2}, the representation~$\mathfrak{m}^{ad}_{\Q_p}$ of~$U_p:=M(\Q_p)\cap U$ is semisimple, and by~\eqref{defi:tate eq1} it has no nonzero factor with a trivial $U_p$-action. We deduce that~$\mathfrak{w}= \mathfrak{y}$, that~$W$ is open in the compact group~$Y$, and~\eqref{preuveB1claim finite} follows.
\end{proof}
We claim that
\begin{equation}\label{preuveB1claim equal}
\forall p\gg0, \abs{Y/W}=1.
\end{equation}
\begin{proof}Let~$M^{ad}\to GL(m)$ be an embedding and consider the induced~$\Z$-structure on~$M^{ad}$. Then we can define, for every prime~$p$ and every~$i\in\Z_{\geq1}$,
\[
K_i:=\ker M(\Z_p)\to M(\Z/(p^i)).
\]
Let~$Y_i$ and~$W_i$ denote the image of~$Y\cap K_i$ and~$W\cap K_i$ in~$K_i/K_{i+1}$.
Then~$\abs{Y/W}=1$ is equivalent to 
\begin{equation}\label{yiarewi}
\forall i\in\Z_{\geq1}, Y_i=W_i.
\end{equation}

For~$p\gg0$, there is a map~$\alpha:M(\F_p)\to M^{ad}(\F_p)$ such that the composed map~$M(\Z_p)\to M^{ad}(\Z_p)\to M^{ad}(\F_p)$ is the same as the composed map~$M(\Z_p)\to M(\F_p)\to M^{ad}(\F_p)$. We deduce that the image of~$U$ in~$M^{ad}(\F_p)$ is
\begin{equation}\label{y0isalpha}
Y_0=\alpha(V(p)^\dagger).
\end{equation}
By Lem.~\ref{lem:B1} and Prop.~\ref{app:cor:ab quot}, we have
\begin{equation}\label{leq Vpdagger to abelian}
\text{ for~$p\gg0$, the image of~$V(p)^\dagger$ in an abelian group is trivial. }
\end{equation}
Thus~$Y_0/W_0$ is trivial. This implies
\begin{equation}\label{y0isw0}
Y_0=W_0.
\end{equation}

For~$p\gg0$ and~$i\in\Z_{\geq1}$, we have~$K_i=\exp(p^i\cdot \mathfrak{m}^{ad}_{\Z_p})$ and this induces an identification of abelian groups, compatible with the action of~$U$ by conjugation,
\[
K_i/K_{i+1}\simeq \frac{p^i\cdot \mathfrak{m}^{ad}_{\Z_p}}{p^{i+1}\cdot \mathfrak{m}^{ad}_{\Z_p}}\simeq \mathfrak{m}^{ad}_{\F_p}.
\]
We note that~$Y_i/W_i$ is a subquotient of~$\mathfrak{m}^{ad}_{\F_p}$ as a representation of~$U$, and~$U$ acts trivially on~$Y_i/W_i$. Thus~$\alpha(U(p))\leq Y_0$ acts trivially on~$Y_i/W_i$.%The action of~$U$ on~$\mathfrak{m}_{\F_p}$ factors through the action of~$U(p)$.

For~$p\gg0$, the map~$M\to M^{ad}$ induces a bijection~$\mathfrak{m}^{der}_{\F_p}\to \mathfrak{m}^{ad}_{\F_p}$. 
For~$p\gg 0$ this bijection is equivariant for the action of~$M(\Z_p)$ via~$M(\Z_p)\to M(\F_p)$ on~$\mathfrak{m}^{der}_{\F_p}$ and via~$M(\Z_p)\to M(\F_p)\xrightarrow{\alpha} M^{ad}(\F_p)$ on~$\mathfrak{m}^{ad}_{\F_p}$.  We deduce a $M(\Z_p)$-equivariant embedding~$\mathfrak{m}^{ad}_{\F_p}\simeq \mathfrak{m}^{der}_{\F_p}\leq \mathfrak{m}_{\F_p}$. In particular it is~$U$-equivariant.

By~\eqref{defi:tate eq 2.2} and~\cite[Th.~5.4]{SCR} the~$U(p)$-representation~$\mathfrak{m}_{\F_p}$ is semisimple for~$p\gg0$. By~\eqref{defi:tate eq 2} and Lem.~\ref{conj orbit lemma}, the centraliser of~$U(p)$ in~$\mathfrak{m}_{\F_p}$ is~$\mathfrak{z(m)}_{\F_p}$. Thus~$\mathfrak{m}^{ad}_{\F_p}$ is semisimple as a representation of~$U(p)$ and, for~$p\gg0$, its~$U(p)$-invariants span~$\{0\}\simeq \mathfrak{z(m)}_{\F_p}\cap \mathfrak{m}^{der}_{\F_p}$.

We deduce
\begin{equation}\label{yiiswi}
Y_i=W_i.
\end{equation}

For~$p\gg0$, we have~\eqref{y0isw0} and~\eqref{yiiswi}. We deduce~\eqref{yiarewi}, which implies the claim~\eqref{preuveB1claim equal}.
\end{proof}

\subsubsection{}

We denote by~$M^{der}\leq M$ the derived subgroup, by~$Z(M)\leq M$ the centre, and we write~$ab_M:M\to M^{ab}=M/M^{der}$ and~$ad_M:M\to M^{ad}=M/Z(M)$ the quotient maps. Let~$\A_f=\widehat{\Z}\tens\Q$, let~$G_1=M^{ab}(\A_f)$ and~$G_2=M^{ab}(\A_f)$. We define
\[
\Gamma:=(ab_M,ad_M)(U)\leq G_1\times G_2
\]
and~$\Gamma_1=ab_M(U)$ and~$\Gamma_2=ad_M(U)$ the projections of~$\Gamma$ on~$G_1$ and on~$G_2$.
According to Goursat's lemma,~$\Gamma/((\Gamma\cap G_1)\times (\Gamma\cap G_2))$ is the graph of an isomorphism
\begin{equation}\label{GoursatB}
\Gamma_1/(\Gamma\cap G_1)\to \Gamma_2/(\Gamma\cap G_2).
\end{equation}
We note that~$G_1$ and~$G_2$ are stable under conjugation by~$U$, and that~$\Gamma$ is a~$U$-stable subgroup of~$G_1\times G_2$. The isomorphism~\eqref{GoursatB} is thus~$U$-equivariant. Note that~$U$ acts trivially on~$G_1$. This implies that~$U$ acts trivially on~$\Gamma_2/(\Gamma\cap G_2)$.

We can thus apply~\eqref{preuveB1claim finite} and~\eqref{preuveB1claim equal} to~$Y=Y(p)$ and~$W=Y(p)\cap \Gamma\cap G_2$. We deduce
\[
\abs{\Gamma_2/(\Gamma\cap G_2)}\leq \prod_p\abs{Y(p)/(Y(p)\cap \Gamma\cap G_2)}<+\infty.
\]
It follows
\[
[\Gamma:\Gamma_1\cap \Gamma_2]<+\infty.
\]
The inverse image of~$\Gamma_1$ and~$\Gamma_2$ in~$U$ is~$U\cap M^{der}(\widehat{\Z})$ and~$U\cap Z(M)(\widehat{\Z})$. We have proved
\begin{equation}\label{factors as der times z}
[U:(U\cap M^{der}(\widehat{\Z}))\cdot (U\cap Z(M)(\widehat{\Z}))]<+\infty.
\end{equation}
\subsubsection{}For~$u\in GL(n,\widehat{\Z})$, we denote by~$\ol{u^{\Z}}$ the closed subgroup generated by~$u$ (for the adelic topology).
Let us write~$\pi_R(u)=(v_p)_p\in \prod_p GL(n,\F_p)$. We claim that if
\begin{equation}\label{unipotent everywhere}
\text{ for all~$p$, the order of~$v_p$ is a power of~$p$ }
\end{equation}
then
\begin{equation}\label{monogene is product}
u\in\prod _p \ol{u^{\Z}}\cap GL(n,\Z_p).
\end{equation}
\begin{proof}The map~$k\to u^k:\Z\to u^\Z$ extends by continuity to a map~$\widehat{\Z}\to \ol{u^\Z}$. For every prime~$\ell$, we denote by~$\gamma_\ell(u)$ the image of~$1\in \Z_\ell$ by~$\Z_\ell\to \widehat{\Z}\to \ol{u^\Z}$.
We have in particular~$\gamma_\ell(u)\in \ol{u^\Z}$.
If we write~$u=(u_p)_p\in \prod_\ell GL(n,\Z_p)$, then we have
\begin{equation}\label{uZbar1}
\gamma_p(u)=(\gamma_\ell(u_p))_p,
\end{equation}
where~$\gamma_\ell(u_p)$ is defined similarly. We note the following: if~$u^{p^{i}}\to 1$ as~$i\to\infty$,
we have
\begin{equation}\label{uZbar2}
\gamma_\ell(u)=1\text{ if~$\ell\neq p$,}\qquad\text{ and}~\gamma_\ell(u)=u\text{ if~$\ell=p$}.
\end{equation}
The property~\eqref{unipotent everywhere} implies that~$\forall p, (u_p)^{p^{i}}\to 1$. From~\eqref{uZbar1} and~\eqref{uZbar2}, we deduce that for every~$\ell$,~$\gamma_\ell(u)$ is the image of~$u$ by~$GL(n,\widehat{\Z})\to GL(n,\Z_p) \to GL(n,\widehat{\Z})$. It follows that~$\gamma_\ell(u)\in \ol{u^\Z}\cap GL(n,\Z_p)$ and that~$u=\prod_\ell \gamma_\ell(u)$.
The claim follows.
\end{proof}
Let~$\Gamma\leq GL(n,\widehat{\Z})$ be a subgroup such that~$\pi_R(\Gamma)$ is generated by elements~$v=(v_p)_p$ satisfying~\eqref{unipotent everywhere}. We claim that
\begin{equation}\label{Gammabar is product}
\ol{\Gamma}=\prod_p\ol{\Gamma}\cap GL(n,\Q_p).
\end{equation}
\begin{proof}We argue by double inclusion, only one of which is non trivial. Because the right hand-side is a closed group, it is enough to prove that it contains a set of generators of~$\Gamma$. By~\eqref{monogene is product}, we can take the set of~$u\in\Gamma$ such that~$(v_p)_p:=\pi_R(u)$ satisfy~\eqref{unipotent everywhere}.
\end{proof}
\subsubsection{}
By~\eqref{leq Vpdagger to abelian}, there exists~$p_0$ such that
\begin{equation}\label{p0 Vp ab}
\forall~p\geq p_0, V(p)^{\dagger}\text{ has no abelian quotient.}
\end{equation}
We recall that~$\Lambda:=\oplus_{p\geq p_0} V(p)^\dagger$ (the subgroup generated by~$\bigcup_{p\geq p_0} V(p)^\dagger$) is dense in~$V_0:=\prod_{p\geq p_0} V(p)^\dagger$. 

%By definition of~$V(p)^\dagger$, the group~$\Lambda$ is generated by elements~$v=(v_p)_p$ satisfying~\eqref{unipotent everywhere}. By~\eqref{Gammabar is product} any subgroup~$\Gamma\leq GL(n,\widehat{\Z})$ such that~$\pi_R(\Gamma)=\Lambda$ satisfies
%\[
%\ol{\Gamma}=\prod_p \ol{\Gamma}\cap GL(n,\Q_p).
%\]
Let~$U_0=\ker\left(U\to \prod_{p<p_0} GL(n,\F_p)\right)$. We have
\begin{equation}\label{B3finite}
[U:U_0]\leq \prod_{p<p_0}\abs{GL(n,\F_p)}<+\infty.
\end{equation}
Let~$U'=U\cap M^{der}(\widehat{\Z})$ and let~$U'_0=U'\cap U_0$. Using~\eqref{factors as der times z}, we get
\begin{equation}\label{B3 fracs}
\frac{U_0}{U'_0}\simeq \frac{U_0\cdot U'}{U'}\leq \frac{U}{U'}\simeq \frac{U\cap Z(M)(\widehat{Z})}{U\cap Z(M^{der})(\widehat{\Z})}.
\end{equation}
These are thus abelian groups. It follows that~$\pi_R(U_0)/\pi_R(U'_0)$ is abelian.

We may assume that ~$p_0$ is chosen big enough, so that for~$p\geq p_0$, there is a map~$M(\F_p)\to M^{ab}(\F_p)$ such that we have~$p\nmid \# M^{ab}(\F_p)$ and~$M^{der}(\F_p)=\ker\left( M(\F_p)\to M^{ab}(\F_p)\right)$. Thus~$V(p)^\dagger\leq M^{der}(\F_p)$ for~$p\geq p_0$, and thus~$V_0\leq M^{der}(R)$ and
\[
M^{der}(R)\cap (V_0\cdot Z(M)(R))=V_0\cdot (M^{der}(R)\cap Z(M)(R))=V_0\cdot Z(M^{der})(R).
\]
We deduce~$\pi_R(U'_0)\leq M^{der}(R)\cap U(R)\leq V_0\cdot Z(M^{der})(R)$. Let~$U_0''$ be the inverse image of~$V_0$ in~$U'_0$. There exists~$f\in\Z_{\geq1}$ such that~$\#{Z(M^{der}(\F_p))}$ divides~$f$ for all primes~$p$. We deduce
\[\forall u\in U'_0, u^f\in U''_0.\]

%\begin{multline}
%\prod_{p\geq p_0}M^{der}(\F_p)\cap\left( \prod_{p\geq p_0} V(p)^\dagger\cdot(U(R)\cap Z(M)(\R))\right)\\= 
%\prod_{p\geq p_0} V(p)^\dagger\cdot \left (\prod_{p\geq p_0}M^{der}(\F_p)\cap (U(R)\cap Z(M)(\R))\right)
%\\ \leq \prod_{p\geq p_0} V(p)^\dagger\cdot Z(M^{der}(\F_p)).
%\end{multline}
%As~$\pi_R(U'_0)\leq \prod_{p\geq p_0}M^{der}(\F_p)$ and~$\exists f, \forall p, \#{Z(M^{der}(\F_p))} \mid f$, we have
%\[
%\forall v\in \pi_R(U'_0), v^f \in\prod_{p\geq p_0} V(p)^{\dagger}=V_0\leq \pi_R(U).
%\]
%Let~$U''_0$ be the inverse image of~$V_0$ in~$U'_0$. We have~$\forall u\in U_0', u^f\in U_0''$.
By~\eqref{p0 Vp ab}, the morphism~$V(p)^\dagger\to V_0/\pi_R(U_0)\leq \pi_R(U_0)/\pi_R(U'_0)$ is constant.
Thus~$V(p)^\dagger\leq \pi_R(U'_0)$, and thus~$\Lambda\leq \pi_R(U''_0)\leq V_0$. Let~$\Gamma$
be the inverse image of~$\Lambda$ in~$U''_0$. Then~$\pi_R(\Gamma)=\Lambda$. Moreover~$\Gamma\leq U''_0$ is a dense subgroup and~$U''_0$ is compact, and thus~$\ol{\Gamma}=U''_0$. 

By definition of~$V(p)^\dagger$, the group~$\Lambda$ is generated by elements~$v=(v_p)_p$ satisfying~\eqref{monogene is product}. From~\eqref{Gammabar is product} we deduce
\[
U''_0=\ol{\Gamma}=\prod_p \ol{\Gamma}\cap M(\Q_p)\leq \prod_p U'\cap M(\Q_p).
\]
By~\eqref{B3finite}, we have~$[U':U'_0]\leq [U:U_0]<\infty$, and we conclude
\begin{equation}\label{factors by p}
\exists e, \forall u\in U'=U\cap M^{der}(\widehat{\Z}), u^e\in U''_0\leq \prod_p U'\cap M(\Q_p).
\end{equation}

\end{document}